\documentclass[10pt,oneside,a4paper]{amsart}
\usepackage{amsmath,amsthm, xy,latexsym,verbatim}
\usepackage{amsxtra}
\usepackage{mathabx}
\usepackage{pxfonts}
\usepackage[psamsfonts]{amssymb}
\usepackage[applemac]{inputenc}
\usepackage{fancyhdr}
\usepackage{mathrsfs}
\usepackage{xcolor,soul}
\usepackage[bookmarks,colorlinks]{hyperref}
\usepackage{enumerate}
\usepackage{nameref}
\usepackage{empheq}
\makeatletter
\let\orgdescriptionlabel\descriptionlabel
\renewcommand*{\descriptionlabel}[1]{%
	\let\orglabel\label
	\let\label\@gobble
	\phantomsection
	\edef\@currentlabel{#1}%
	\let\label\orglabel
	\orgdescriptionlabel{#1}%
}
\makeatother
\fancypagestyle{plain}{
	\fancyhead{}
	
}

\newtheorem{definition}{Definition}[section]
\newtheorem{theorem}[definition]{Theorem}

\newcommand{\rubrik}{}
\newtheorem{proposition}[definition]{Proposition}
\newtheorem{lemma}[definition]{Lemma}
\newtheorem{corollary}[definition]{Corollary}
\newtheorem{remark}[definition]{Remark}
\newtheorem{example}[definition]{Example}

\renewcommand{\theequation}{\thesection.\arabic{equation}}

\def\supp{ {\operatorname{supp}} }
\newcommand{\cl}{\textnormal{cl}}
\newcommand{\m}{\mathbf{m}}
\def\Symp#1{ {{\sigma_{p}}} \left( #1 \right) }
\newcommand{\x}{\langle x\rangle}
\newcommand{\csi}{\langle \xi \rangle}
\def\jb#1{ {\langle} #1 {\rangle}\ }

\makeatletter
\newcommand{\neutralize}[1]{\expandafter\let\csname c@#1\endcsname\count@}
\makeatother

\newtheorem{assumption}{Assumption}

\def\fios{\text{Fourier integral operators}}

\def\R {\mathbb{R}}       
\def\N {\mathbb{N}}       
\def\Z {\mathbb{Z}}
\def\Rn {\mathbb{R}^{n}}
\def\SG {{S}}

\def\ds{\displaystyle}

\def\cB{{\mathcal B}}
\def\cC{{\mathcal C}}

\def\cL{{\mathcal L}}

\def\cS{{\mathcal S}}

\def\1{\lambda}
\def\2{\Sigma_{2}}

\def\e{\epsilon}
\def\vp{\varphi}
\def\s{\sigma}

\def\g{\gamma}

\def\bt{\mathbf{t}}
\def\bL{\mathbf{L}}

\def\supp#1{ {\mathrm{supp}}\left( #1 \right) }

\def\diag{ \mathrm{diag} }
\def\mod{ \mathrm{mod} }

\def\SX{ {\cal{S}} }

\def\Op#1{ {\mathrm{Op}} \left( #1 \right) }

\def\Symp#1{ {\sigma_{p}} \left( #1 \right) }

\def\ssum#1#2{ { \displaystyle{\sum_{{#1}}^{#2}} } }

\def\tD{\Gamma}
\def\card#1{ \mathbf {card}(#1)}

\def\perm{M}

\def\perms{ \perm^\nmid }

\def\<{{\langle}}
\def\>{{\rangle}}

\def\a{{\alpha}}

\def\b{{\beta}}

\def\cl{\mathrm{cl}}

\def\SX{ {\mathcal{S}} }

\newcommand{\afrac}[2]{\genfrac{}{}{0pt}{1}{#1}{#2}} 
\newcommand{\beqsn}{\arraycolsep1.5pt\begin{eqnarray*}}
	\newcommand{\eeqsn}{\end{eqnarray*}\arraycolsep5pt}
\newcommand{\beqs}{\arraycolsep1.5pt\begin{eqnarray}}
\newcommand{\eeqs}{\end{eqnarray}\arraycolsep5pt}
	
\newcommand{\op}{\operatorname{Op}}
\def\Op{ {\operatorname{Op}} }
\newcommand{\Ph}{\mathcal P}
\newcommand{\Phr}{\mathcal P_r}
\def\fy{\varphi}

\newcommand*{\scrB}{\ensuremath{\mathscr{B}}}	
\newcommand*{\scrF}{\ensuremath{\mathscr{F}}}	
\newcommand*{\caP}{\ensuremath{\mathcal{P}}}		
\newcommand*{\caH}{\ensuremath{\mathcal{H}}}
\newcommand*{\caF}{\ensuremath{\mathcal{F}}}		
\newcommand*{\caE}{\ensuremath{\mathcal{E}}}		
\newcommand*{\E}{\mathbb{E}}										
\renewcommand*{\P}{\mathbb{P}}									
\def\tv{\mathrm{tv}}
\newcommand{\WFgB}{\operatorname{WF}_{\cB}}

\newcommand{\direz}{{\operatorname{dir}}}
\newcommand{\Char}{\operatorname{Char}}
\newcommand{\back}[1]{\backslash {#1}}
\newcommand{\WFF}{\operatorname{WF}}

\newcommand{\sets}[2]{\{ \, #1\, ;\, #2\, \} }

\author{Ahmed Abdeljawad}

\address{Dipartimento di Matematica ``G. Peano'', Universit\`a degli Studi di Torino, via Carlo Alberto 10. 10123 Torino,  Italy}

\email{ahmed.abdeljawad@unito.it}

\author{Alessia Ascanelli}

\address{Dipartimento di Matematica, Universit\`a degli Studi di Ferrara, via Machiavelli 30, 44121 Ferrara, Italy}

\email{alessia.ascanelli@unife.it}

\author{Sandro Coriasco}

\address{Dipartimento di Matematica ``G. Peano'', Universit\`a degli Studi di Torino, via Carlo Alberto 10. 10123 Torino,  Italy}

\email{sandro.coriasco@unito.it}

\title[Cauchy problems for hyperbolic operators on $\R^n$]
	{Deterministic and stochastic Cauchy problems\\for a class of weakly hyperbolic operators on $\R^n$}

\keywords{Fourier integral operator, multi-product, hyperbolic Cauchy problem, involutive characteristics, stochastic PDEs}

\subjclass[2010]{Primary: 58J40; Secondary: 35S05, 35S30, 47G30, 58J45}

\begin{document}
	
	\begin{abstract}
		We study a class of hyperbolic Cauchy problems, associated with linear operators and systems
		with polynomially bounded coefficients, 
		variable multiplicities and involutive characteristics, globally defined on $\R^n$. We prove well-posedness
		in Sobolev-Kato spaces, with loss of smoothness and decay at infinity. We also obtain results about propagation of singularities,
		in terms of wave-front
		sets describing the evolution of both smoothness and decay singularities of temperate distributions. Moreover,
		we can prove the existence of random-field solutions for the associated stochastic Cauchy problems.
		To this aim, we first discuss algebraic properties for iterated integrals of suitable
		parameter-dependent families of Fourier integral operators,
		associated with the characteristic roots, which are involved in the construction of the fundamental solution. 
		In particular, we show that, also for this operator class, the involutiveness of the characteristics implies
		commutative properties for such expressions.
	\end{abstract}
	
	\maketitle
	\setcounter{tocdepth}{1}
	
	
	\section{Introduction}\label{sec:intro}
	\setcounter{equation}{0}
	%
	%
	In the present paper we focus on the Cauchy problem
	\begin{equation}
\label{cpintro}
\left\{
\begin{array}{ll}
L u(t) = f(t),          & t\in (0,T]
\\[1ex]
D_{t}^k u(0) = g_{k}, & k=0,\dots,m-1,
\end{array}
\right.
\end{equation}
where $L=L(t,D_t;x,D_x)$ is a linear partial differential operator of the form
\begin{equation}\label{Lintro} 
L(t,D_t;x,D_x)= D_{t}^m + \sum _{j=1}^{m}\sum_{|\alpha|\le j}c_{j\alpha}(t;x)D_x^\alpha D_t^{m-j},
\end{equation}
$D=-i\partial$, with $(t,x)-$smoothly depending coefficients $c_{j\alpha}$, possibly admitting a polynomial growth, namely,
\begin{equation}\label{grw}
|\partial_t^k\partial_x^\beta c_{j\alpha}(t;x)|\le C_{k\alpha }\langle x\rangle^{j-|\beta|},\qquad \alpha,\beta\in\Z^n_+,\ |\alpha|\leq j,\ j=1,\ldots,m,
\end{equation}
for $\langle x \rangle:=(1+|x|^2)^{1/2}$, some $C_{k\alpha}>0$ and all $t\in[0,T]$, $x\in\R^n$.
The hypothesis of smoothness with respect to $t$ is assumed only for the sake of simplicity, since here we will not deal with questions concerning low regularity in time for the coefficients of the Cauchy problem. A very wide literature concerning this topic has been developed along the past 40 years, starting with the paper \cite{cdgs}. 
 
The assumption \eqref{grw} suggests to set the problem within the framework of the so-called SG calculus (see \cite{CO,PA72}), defined through symbols satisfying global estimates on $\R^n\times\R^n$.
Explicitly, given $(m,\mu)\in\R^2$, the class of SG symbols $S^{m,\mu}$ consists of all symbols $a\in C^\infty(\R^n\times\R^n)$ such that 
 $$
 	| \partial_x^{\a} \partial_{\xi}^{\b} a(x,\xi)| \leq C_{\a\b}\,\jb{x}^{m-|\a|}\jb{\xi}^{\mu-|\b|}
$$
for all $ x, \xi \in \R^n$, $\alpha,\beta\in \Z^n_+$, and some $C_{\a\b}>0$,
see Section \ref{sec:sgcalc} below for the precise definitions and some basic properties of the corresponding calculus. 
Notice that for $m=0$ we recover the class $S^\mu=\SG^{0,\mu}$ of (H\"ormander's) symbols uniformly bounded with respect to the space variable $x$.
We are interested in existence and uniqueness of the solution to \eqref{cpintro} in suitable weighted Sobolev-type spaces of functions or distributions, according to the regularity of the Cauchy data and of the operator $L$.
More precisely, it is well-estabilished (see, e.g., \cite{M2} for PDEs with uniformly bounded coefficients, and \cite{CO} for PDEs in the SG framework) that, to have existence of a unique solution to \eqref{cpintro} in Sobolev-type spaces, an hyperbolicity assumption is needed. The operator $L$ is said to be \emph{hyperbolic} if its principal symbol 
$$
L_m(t,\tau;x,\xi):=\tau^m + \sum _{j=1}^{m}\sum_{|\alpha|= j}c_{j\alpha}(t;x)\xi^\alpha \tau^{m-j}
$$
factorizes as 
\begin{equation}
\label{hypintro}L_m(t,\tau;x,\xi)=\prod_{j=1}^m (\tau-\tau_j(t;x,\xi)),
\end{equation}
with real-valued and smooth roots $\tau_j$, usually called \emph{characteristic roots} of the operator $L$.
A sufficient condition for existence of a unique solution is the \textit{separation between the roots}, which reads either as 
\begin{align}
\tag{H}\quad  |\tau_j(t;x,\xi)-\tau_k(t;x,\xi)|\geq & C\langle\xi\rangle &&\text{(uniformly bounded coefficients) or},
\\
\tag{SGH}\quad  |\tau_j(t;x,\xi)-\tau_k(t;x,\xi)|\geq & C\langle x\rangle\langle\xi\rangle &&\text{(SG type coefficients)},
\end{align}
for a suitable $C>0$ and every $(t;x,\xi)\in [0,T]\times\R^{2n}$. Conditions (H) or (SGH), respectively, are assumed to hold true
either for every $j\neq k$ (\emph{strict hyperbolicity}), or at least for every $\tau_j$, $\tau_k$ belonging to different groups of coinciding roots (\emph{weak hyperbolic with roots of constant multiplicity}).
The strict hyperbolicity condition (H) (respectively, (SGH)) allows to prove existence and uniqueness in Sobolev spaces (respectively, in Sobolev-Kato spaces) of a solution to \eqref{cpintro} for every 
$f,g_k$, $k=0,\dots,m-1$, in Sobolev spaces (respectivey, Sobolev-Kato spaces). A weak hyperbolicity condition with roots of constant multiplicity, together with a condition on the lower order terms of the operator $L$ (a so called \emph{Levi condition}) allows to prove a similar result, and the phenomena of loss of derivatives and modification of the behavior as $|x|\to\infty$ with respect to the initial data appear. The first phenomenon has been observed for the first time in \cite{cdgs}, the second one in \cite{scncpp}.
If there is no separation condition on the roots, then the only available results in literature that we are aware of are those due to Morimoto \cite{Morimoto} and Taniguchi \cite{Taniguchi02}, dealing with \emph{involutive roots} in the uniformly bounded coefficients case. The condition of involutiveness, which is weaker than the separation condition (H), requires that, for all $j,k\in\N$, $t\in[0,T]$, $x,\xi\in\R^n$, the Poisson brackets
$$\{\tau-\tau_j,\tau-\tau_{k}\}:=\partial_t\tau_j-\partial_t\tau_{k}+ \tau'_{j,\xi}\cdot \tau'_{k,x}-\tau'_{j,x}\cdot  \tau'_{k,\xi}$$
may be written as 
\begin{equation}\label{invintro}
\{\tau-\tau_j,\tau-\tau_{k}\}
=
b_{jk}\cdot(\tau_j-\tau_{k})+d_{jk},
\end{equation}
for suitable parameter-dependent, real-valued symbols $b_{jk},d_{jk}\in C^\infty([0,T];S^{0}(\R^{2n}))$, $j,k\in\N$.
Under condition \eqref{invintro} and adding a Levi condition, the authors reduced \eqref{cpintro} to an equivalent first order system
\begin{equation}\label{sysint}\begin{cases}
	\bL U(t) = F(t), & 0\leq t\leq T, \qquad \bL = D_t + \Lambda(t,x,D_x) +R(t,x,D_x),
	\\[1ex]
	U(0)  = G, & 
	\end{cases}
\end{equation}
%
with $\Lambda$ a diagonal matrix having entries given by operators with symbol coinciding with the $\tau_j$, $j=1,\dots, m$,
and $R$ a (full) matrix of operators of order zero. They constructed the fundamental solution to \eqref{sysint}, that is, a smooth family $\{E(t,s)\}_{0\leq s\leq t\leq T}$ of operators such that  
\begin{equation*}\begin{cases}
	\bL E(t,s) = 0 & 0\leq s\leq t\leq T,
	\\[1ex]
	E(s,s)  = I, & s\in[0,T),
	\end{cases}
\end{equation*}
so obtaining by Duhamel's formula the solution
$$U(t)=E(t,0)G+i\int_0^t E(t,s)F(s)ds$$
to the system and then the solution $u(t)$ to the corresponding higher order Cauchy problem \eqref{cpintro} by the reduction procedure.
\vskip+0.5cm
In the present paper we consider the Cauchy problem \eqref{cpintro} in the SG setting, that is, under the growth condition \eqref{grw}, the hyperbolicity condition \eqref{hypintro} and the involutiveness assumption \eqref{invintro} for suitable parameter-dependent, real-valued symbols $b_{jk},d_{jk}\in C^\infty([0,T];S^{0,0}(\R^{2n}))$, $j,k\in\N$.  The involutiveness of the characteristic roots of the operator, or, equivalently, of the eigenvalues of the (diagonal) principal part of the corresponding first order system is here the main assumption (see Assumption \ref{assumpt:invo} in Section \ref{subsec:commlaw} for the precise statement of such condition).

We obtain several results concerning the Cauchy problem \eqref{cpintro}:
\begin{itemize}
\item we prove, in Theorem \ref{thm:mainLastSecThm}, the existence, for every $f\in C^\infty([0,T]; H^{r,\rho}(\R^n))$ and $g_k\in H^{r+m-1-k,\rho+m-1-k}(\R^n)$, $k=0,\ldots,m-1$, of a unique 
$$
	u\in \bigcap_{k\in\Z^+}C^k([0,T']; H^{r-k,\rho-k}(\R^n)),
$$ 
for a suitably small $0< T'\leq T$, solution of \eqref{cpintro}; we also provide, in \eqref{eq:soleqordm}, the explicit expression of $u$, using SG Fourier integral operators (SG FIOs);
\item we give a precise description, in Theorem \ref{thm:propwfs}, of a global wave-front set of the solution, under a (mild) additional condition on the operator $L$; namely, we prove that, when $L$
is SG-classical, the (smoothness and decay) singularities of the solution of \eqref{cpintro} with $f\equiv0$ are points of unions of arcs of bicharacteristics, generated by the phase functions of the SG FIOs appearing in the expression \eqref{eq:soleqordm}, and emanating from (smoothness and decay) singularities of the Cauchy data $g_k$, $k=0,\dots,m-1$;
\item we deal, in Theorem \ref{thm:linearin}, with a stochastic version of the Cauchy problem \eqref{cpintro}, namely, with the case when $$f(t;x)=\gamma(t;x)+\sigma(t;x)\dot\Xi(t;x),$$ and $\Xi$ is a random noise of Gaussian type, white in time and with possible correlation in space; we give conditions on the noise $\Xi$, on the coefficients $\gamma,\sigma$, and on the data $g_k$, $k=0,\dots,m-1$, such that there  exists
a so-called \emph{random field solution} of \eqref{cpintro}, that is, a stochastic process which gives, for every fixed $(t,x)\in [0,T]\times\R^n$, a well-defined random variable (see Section \ref{sec:sghypspde} below for details).
\end{itemize}
	
All the results described above are achieved through the use of Fourier integral operators of SG-type, which are recalled in Section \ref{sec:fioprelim}. 
The \emph{involutiveness assumption} is the key to prove that the fundamental solution of the  first order system \eqref{sysint} may be written as a \emph{finite sum} of (iterated integrals of smooth parameter-dependent families of) 
SG FIOs (modulo a regularizing term). This is proved in Theorem \ref{thm:mainthm}, and follows from the analysis performed in Section \ref{sec:sgphfcomlaw}. The (rather technical) results 
proved in Section \ref{sec:sgphfcomlaw} are indeed crucial to achieve our main results. Namely, we prove that, under the involutiveness assumption, multi-products of regular SG-phase functions given as solutions of eikonal equations satisfy a \emph{commutative law} (Theorem \ref{cruthm}). Similar to the classical situation, this extends to the SG case the possibility to reduce the fundamental solution to a finite sum of terms (modulo a regularizing operator). This also implies commutativity properties of the flows \textit{at infinity} generated by the corresponding (classical) SG-phase functions. Their relationship with
the symplectic structures studied in \cite{CS17} will be investigated elsewhere.

The paper is organized as follows. In Section \ref{sec:sgcalc} we recall basic elements of the SG calculus of pseudo-differential and Fourier integral operators. In Section
\ref{sec:sgphfcomlaw} we prove our first main theorem, namely, a commutative property for multi-products of suitable families of regular SG-phase functions, determined
by involutive families of SG-symbols. In Section \ref{sec:fundsol}
we prove our second main theorem, showing how such commutative property can be employed to obtain the fundamental solution for certain $N\times N$ first order linear systems
with diagonal principal part and involutive characteristics. In Section \ref{sec:sginvcp} we study the Cauchy problem for SG-hyperbolic, involutive linear differential operators. 
Here two more main theorems are proved: a well-posedness result (with loss of smoothness and decay) and a propagation of singularity result, for wave-front sets of global type, whose definition
and basic properties are also recalled. Finally, in Section \ref{sec:sghypspde} we prove
our final main theorem, namely, the existence of random field solutions for a stochastic version of the Cauchy problems studied in Section \ref{sec:sginvcp}.

	\section{Symbolic calculus for pseudo-differential and Fourier integral operators of SG type}\label{sec:sgcalc}
	\setcounter{equation}{0}
	%
	%
	In this section we recall some properties of SG pseudo-differential operators,
	an operator class defined through symbols satisfying global estimates on $\R^n\times\R^n$.
	Moreover, we state several fundamental theorems for SG Fourier integral operators.
	Further details can be found, e.g., in \cite{AsCor,CO,EgorovSchulze,Kumano-go:1,SH87}
	and the references therein, from which we took most of the materials included in this section.
	Here and in what follows, $A\asymp B$ means that $A\lesssim B$ and $B\lesssim A$,
	where $A\lesssim B$ means that $A\le c\cdot B$, for a suitable constant $c>0$.
	
	\subsection{Pseudodifferential operators of SG type}
	
	\par
	
	The class $\SG^{m,\mu}(\R^{2n})$ of SG symbols of order $(m,\mu) \in \R^2$,
	is given by all the functions $a(x,\xi) \in C^\infty(\R^{2n})$ with the property that, 
	for any multiindices $\a,\b \in \Z_+^n$,  there exist
	constants $C_{\a\b}>0$ such that the conditions 
	\begin{equation}\label{eq:disSG}
	| D_{\xi}^{\a}D_x^{\b} a(x,\xi)| 
	\leq 
	C_{\a\b}\,\jb{x}^{m-|\b|}\jb{\xi}^{\mu-|\a|}
	\end{equation}
	hold, for all $ x, \xi \in \R^n$.
	Here $\langle x \rangle=(1+|x|^2)^{1/2}$ when $x\in\R^n$, and $\Z_+$ is the set of non-negative integers.
	For $a(x,\xi)\in\SG^{m,\mu}(\R^{2n})$ , $m,\mu\in\R$, we define the semi-norms $\vvvert a\vvvert _l^{m,\mu}$ by
	\begin{align}\label{semi_norm_MSG}
	\vvvert a\vvvert_l^{m,\mu}
	=\max_{|\a+\b|\leq l}
	\sup_{x,\xi\in\R^n}\jb{x}^{-m+|\b|}\jb{\xi}^{-\mu+|\a|}| D_{\xi}^{\a}D_x^{\b} a(x,\xi)|,
	\end{align}
	where $l\in\Z_+$.
	The quantities~\eqref{semi_norm_MSG} define a Fr\'echet topology of $\SG^{m,\mu}(\R^{2n})$.
	Moreover, let
	$$ 
	\SG^{\infty,\infty}(\R^{2n})=\bigcup_{m,\mu\in\R}\SG^{m,\mu}(\R^{2n}),
		\qquad
		\SG^{-\infty,-\infty}(\R^{2n})=\bigcap_{m,\mu\in\R}\SG^{m,\mu}(\R^{2n}).
	$$
	
	\par
	
	The functions $a\in\SG^{m,\mu}(\R^{2n})$ can be $(\nu\times\nu)$-matrix-valued.
	In such case the estimate \eqref{eq:disSG}~~must be valid for each entry of the matrix.
	The next technical lemma is useful when dealing with compositions of SG symbols.

	\begin{lemma} \label{compositionwithe1symbol}
		Let $f\in\SG^{m,\mu}(\R^{2n})$, $m,\mu\in\R$, and $g$ vector-valued
		in $\R^n$ such that $g\in\SG^{0,1}(\R^{2n})\otimes\R^n$ and $\jb{g(x,\xi)}\sim\jb{\xi}$. 
		Then $f(x,g(x,\xi))$ belongs to $\SG^{m,\mu}(\R^{2n})$.
	\end{lemma}
	\par
	The previous result can be found in \cite{Cothesis},
	and can of course be extended to the other composition case, namely
	$h(x,\xi)$ vector valued in $\R^n$  such that it
	belongs to $\SG^{1,0}(\R^{2n})\otimes\R^n$ and $\jb{h(x,\xi)}\sim \jb{x}$,
	implying that $f(h(x,\xi),\xi)$ belongs to $\SG^{m,\mu}(\R^{2n})$.
	
	\par
	
	We now recall definition and properties of the pseudo-differential operators
	$a(x,D)=\Op(a)$ where $a\in\SG^{m,\mu}$.
	%
	%
%
Given $a\in\SG^{m,\mu}$, we define $\Op(a)$ through the left-quantization (cf. Chapter XVIII in \cite {Ho1})
\begin{equation}\label{eq:psidos}
(\Op(a)u)(x)= (2 \pi)^{-n}\int_{\R^n} e^{i x \cdot \xi} a(x, \xi) \widehat{u}(\xi) \,d\xi, \quad u\in\SX,
\end{equation}
with $\widehat{u}$ the Fourier transform of $u\in\SX$, given by
\begin{equation}\label{eq:tfu}
\widehat{u}(\xi)=\int_{\R^n} e^{-ix\cdot\xi}u(x)\,dx.
\end{equation}
%
	
\par
	
The operators in \eqref{eq:psidos} form a
graded algebra with respect to composition, i.e.,
	$$
	\op (\SG ^{m_1,\mu _1})\circ \op (\SG ^{m_2,\mu _2})
	\subseteq \op (\SG ^{m_1+m_2,\mu _1+\mu _2}).
	$$
	The symbol $c\in \SG ^{m_1+m_2,\mu _1+\mu _2}$ of the composed
	operator $\Op(a)\circ\op(b)$, where
	$a\in \SG ^{m_1,\mu _1}$, $b\in \SG ^{m_2,\mu _2}$,
	admits the asymptotic expansion
	\begin{equation}
	\label{eq:comp}
	c(x,\xi)\sim \sum_{\alpha}\frac{i^{|\alpha|}}{\alpha!}
		\,D^\alpha_\xi a(x,\xi)\, D^\alpha_x b(x,\xi),
	\end{equation}
	which implies that the symbol $c$ equals $a\cdot b$ modulo 
	$S ^{m_1+m_2-1,\mu _1+\mu _2-1}$.
	
	The residual elements of the calculus are operators with symbols in
	\[
	\SG ^{-\infty,-\infty}=\SG ^{-\infty,-\infty}
	=
	 \bigcap_{(m,\mu) \in \R^2} \SG ^{m,\mu} 
	=\SX,
	\]
	that is, those having kernel in $\SX$, continuously
	mapping  $\SX^\prime$ to $\SX$. 
	
	\par
	
	An operator $A=\Op(a)$, is called \emph{elliptic}
	(or $\SG ^{m,\mu}$-\emph{elliptic}) if
	$a\in \SG ^{m,\mu}$ and there
	exists $R\geq0$ such that
	%
	\begin{equation}\label{eq:ellcond}
	C\jb{x}^{m} \jb{\xi}^{\mu}\le |a(x,\xi)|,\qquad 
	|x|+|\xi|\geq R,
	\end{equation}
	for some constant $C>0$. An elliptic SG  operator $A \in \op (\SG ^{m,\mu})$ admits a
	parametrix $P\in \op (\SG ^{-m,-\mu})$ such that
	\[
	PA=I + K_1, \quad AP= I+ K_2,
	\]
	for suitable $K_1, K_2\in\Op(\SG^{-\infty,-\infty})$, where $I$ denotes the identity operator. 
	
	\par
	
	It is a well-known fact that SG-operators give rise to linear continuous mappings 
	from $\SX$ to itself, extendable as linear continuous mappings
	from $\SX'$ to itself. They also act continuously between
	the so-called Sobolev-Kato (or weighted Sobolev) space, that is from 
	$H^{r,\varrho}(\R^n)$ to $H^{r-m,\varrho-\mu}(\R^n)$,
	where $H^{r,\varrho}(\R^n)$,
	$r,\varrho \in \R$, is defined as
	\begin{equation*}
	H^{r,\varrho}(\R^n)= \left\{u \in \cS^\prime(\R^{n}) \colon \|u\|_{r,\varrho}=
	\|{\jb .}^r\jb D^\varrho u\|_{L^2}< \infty\right\}.
	\end{equation*}

	\par
	
	The so-called SG manifolds were introduced in 1987 by E.~Schrohe \cite{Sc}. 
	This is a class of non-compact
	manifolds, on which a version of SG calculus can be defined.
	Such manifolds admit a finite atlas, whose changes of
	coordinates behave like symbols of order $(0,1)$ (see \cite{Sc}
	for details and additional technical hypotheses). A relevant
	example of  SG manifolds are the manifolds with
	cylindrical ends, where also the concept of classical SG
	operator makes sense,
	see, e.{\,}g. \cite{BC10a,MP02}.

\par

We close this section with the next result, which deals with the composition (or multi-product)
of $(M+1)$ SG pseudo-differential operators, where $M\ge1$. The proof of Theorem 
\ref{thm:mpsgpdosymestim} can be found in \cite{Abdeljawadthesis}.

\begin{theorem}\label{thm:mpsgpdosymestim}
	Let $M\geq 1$, $p_j(x,\xi)\in\SG^{m_j,\mu_j}(\R^{2n})$, $m_j,\mu_j\in\R$,
	$j=1,\dots,M+1$. Consider the multi-product
	\begin{equation}\label{eq:qm}
	Q_{M+1}=P_1 \cdots P_{M+1}
	\end{equation}
	of the operators $P_j=\Op(p_j)$, $j=1,\dots,M+1$, and denote by $q_{M+1}(x,\xi)\in\SG^{m,\mu}(\R^{2n})$
	the symbol of $Q_{M+1}$, where	$m=m_1+\dots+m_{M+1}$, $\mu=\mu_1+\dots+\mu_{M+1}$.
	Then, the boundedness of $\displaystyle\oplus_{j=1}^{M+1} p_j$
	in $\displaystyle\oplus_{j=1}^{M+1}\SG^{m_j,\mu_j}(\R^{2n})$,
	implies the boundedness of 
	$q_{M+1}(x,\xi)$ in $\SG^{m,\mu}(\R^{2n})$.
\end{theorem}

\par
	
	\subsection{Fourier integral operators of SG type}\label{sec:fioprelim}
	%
	%

	\par
	
	Here we recall the class of Fourier integral operators we are interested in.
	Also, we show their compositions with the SG pseudo-differential operators,
	and the compositions between Type I and Type II operators.
	
	\par
	
	The simplest \fios\ are defined, for $u\in\SX$, as
	\begin{align}
	\label{eq:typei}
	u\mapsto (\Op _\fy (a)u)(x)&= (2\pi )^{-n}\int_{\R^n}
	e^{i\fy (x,\xi )} a(x,\xi )
	\widehat u(\xi )\, d\xi ,
	\end{align}
	and
	\begin{align}
	\label{eq:typeii}
	u\mapsto (\Op^*_\fy (a)u)(x)&= (2\pi )^{-n}\iint_{\R^{2n}}
	e^{i(x\cdot \xi -\fy (y,\xi ))} \overline {a(y,\xi )} u(y)\, dyd\xi,
	\end{align}
	with suitable choices of phase function $\varphi$ and symbols $a$ and $b$.
	
	\par
	
	The operators $\op _\fy (a)$ and
	$\op _\fy ^*(a)$ are sometimes called SG \fios\
	of type I and type II, respectively, with symbol $a$ and
	phase function $\fy$. Note that a type II operator satisfies
	$\op^*_\fy(a)=\op _\fy (a)^*$, that is, it is the formal $L^2$-adjoint of the 
	type I operator $\op_\fy(a)$.

	We now introduce the definition of the class of admissible phase functions
	in the SG context, as it was given in \cite{Coriasco:998.1}. 

	\begin{definition}[SG phase function]\label{def:phase}
		A real valued function $\varphi\in C^\infty(\R^{2n})$  belongs to the class $\mathcal P$
		of SG phase functions if it satisfies the following conditions:
		\begin{enumerate}
			\item $\varphi\in \SG^{1,1}(\R^{2n})$;
			\item $\<\varphi'_x(x,\xi)\>\asymp\<\xi\>$ as $|(x,\xi)|\to\infty$;
			\item $\<\varphi'_\xi(x,\xi)\>\asymp\<x\>$ as $|(x,\xi)|\to\infty$.	
		\end{enumerate}
	\end{definition}

	The SG \fios\ of type I and type II, $\Op_\fy(a)$ and
	$\Op^*_\fy(b)$, are defined as in \eqref{eq:typei} and \eqref{eq:typeii},
	respectively, with $\fy\in\Ph$ and $a,b\in \SG^{m,\mu}$.
	Notice that we do not assume any homogeneity hypothesis on the phase function $\varphi$.
	The next theorem, treating composition between SG pseudo-differential operators and SG \fios\,
	was originally proved in \cite{Coriasco:998.1}, see also \cite{CJT4, CoPa, CT14}.
	\begin{theorem}\label{thm:compi}
		Let $\fy\in\Ph$ and assume
		$p\in \SG^{t,\tau}(\R^{2n})$, $a,b\in \SG^{m,\mu}(\R^{2n})$. Then,
		\begin{align*}
		\Op (p)\circ \Op _\fy (a) &= \Op _\fy (c_1+r_1) = \Op _\fy (c_1)\
				 \mod\ \Op (\SG^{-\infty,-\infty}(\R^{2n})),
		\\[1ex]
		\Op (p)\circ \Op ^* _\fy(b) &= \Op ^* _\fy(c_2+r_2) = \Op^* _\fy (c_2)\ 
				\mod\ \Op (\SG^{-\infty,-\infty}(\R^{2n})),
		\\[1ex]
		\Op _\fy (a) \circ \Op (p) &= \Op _\fy (c_3+r_3) = \Op _\fy (c_3)\ 
				\mod\ \Op (\SG^{-\infty,-\infty}(\R^{2n})),
		\\[1ex]
		\Op _\fy ^*(b) \circ \Op (p) &= \Op ^* _\fy(c_4+r_4) = \Op^* _\fy (c_4)\
				\mod\ \Op (\SG^{-\infty,-\infty}(\R^{2n})),
		\end{align*}
		for some $c_j\in \SG^{m+t,\mu+\tau}(\R^{2n})$,
		$r_j\in \SG^{-\infty,-\infty}(\R^{2n})$, $j=1,\dots ,4$.
	\end{theorem}
	In order to obtain the composition of SG \fios\ of type I and type II,
	some more hypotheses are needed, leading to the definition
	of the classes $\Phr$ and $\Phr(\tau)$ of regular SG phase functions,
	cf. \cite{Kumano-go:1}.
	\begin{definition}[Regular SG phase function]\label{def:phaser}
		Let $\tau\in [0,1)$ and $r>0$. A function $\varphi\in\mathcal P$ belongs
		to the class $\mathcal P_r(\tau)$ if it satisfies the following conditions:
		\begin{description}
			\item[1\label{1}] $\vert\det(\varphi''_{x\xi})(x,\xi)\vert\geq r$,
					$\text{ for any } x,\xi\, \in \Rn$;
			\item[2\label{2}] the function  $J(x,\xi):=\vp(x,\xi)-x\cdot\xi$ is such that
			\beqs\label{hyp}
			\ds\sup_{\afrac{x,\xi\in\R^n}{|\alpha+\beta|\leq 2}}
			\frac{|D_\xi^\alpha D_x^\beta J(x,\xi)|}{\x^{1-|\beta|}\<\xi\>^{1-|\alpha|}}\leq \tau.
			\eeqs
		\end{description}
		If only condition (1) holds, we write $\fy\in\Phr$.
	\end{definition}
	\begin{remark}
		Notice that if condition \eqref{hyp} actually holds true for any
		$\a,\,\b\, \in \Z_+^n$, then $J(x,\xi)/\tau$ is bounded with constant $1$
		in $\SG^{1,1}(\R^{2n})$. Notice also that condition (1) in Definition \ref{def:phaser}
		is authomatically fulfilled when condition (2) holds true for a sufficiently small $\tau\in[0,1)$.
	\end{remark}

\par
	
For $\ell\in\N$, we also introduce the semi-norms
	\[
	\|J\|_{2,\ell}:=
	\ds\sum_{2\leq |\alpha+\beta|\leq 2+\ell}\sup_{(x,\xi)\in \R^{2n}}
	\ds\frac{|D_\xi^\alpha D_x^\beta J(x,\xi)|}{\x^{1-|\beta|}\<\xi\>^{1-|\alpha|}},
	\]
	and
	\[
	\|J\|_\ell:=\ds\sup_{\afrac{x,\xi\in\R^n}{|\alpha+\beta|\leq 1}}\frac{|D_\xi^\alpha D_x^\beta J(x,\xi)|}{\x^{1-|\beta|}\<\xi\>^{1-|\alpha|}}+\|J\|_{2,\ell}.
	\]
	We notice that $\varphi\in\mathcal P_r(\tau)$ means that \eqref{1} of Definition \ref{def:phaser} and $\|J\|_0\leq \tau$ hold, and then we define the following subclass of the class of regular SG phase functions:
	\begin{definition}\label{def:phaserell}
		Let $\tau\in [0,1)$, $r>0$, $\ell \geq 0$. A function $\varphi$ belongs to the class $\mathcal P_r(\tau,\ell)$ if $\varphi\in\mathcal P_r(\tau)$ and $\|J\|_\ell\leq \tau$ for the corresponding $J$.
	\end{definition}
	Theorem \ref{thm:compii} below shows that the composition of SG \fios\ of type I and type II 
	with the same regular SG phase functions is a SG pseudo-differential operator.
	\begin{theorem}\label{thm:compii}
		Let $\fy\in\Phr$ and assume $a\in \SG^{m,\mu}(\R^{2n})$, $b\in \SG^{t,\tau}(\R^{2n})$. Then,
		\begin{align*}
		\Op _\fy(a)\circ \Op _\fy^*(b)= \Op(c_5+r_5)=\Op(c_5)\ \mod\ \Op (\SG^{-\infty,-\infty}(\R^{2n})),
		\\[1ex]
		\Op _\fy ^*(b)\circ \Op _\fy (a)=\Op(c_6+r_6)=\Op(c_6)\ \mod\ \Op (\SG^{-\infty,-\infty}(\R^{2n})),
		\end{align*}
		for some $c_j\in \SG^{m+t,\mu+\tau}(\R^{2n})$, $r_j\in \SG^{-\infty,-\infty}(\R^{2n})$, $j=5,6$.
	\end{theorem}

\par

Furthermore, asymptotic formulae can be given for $c_j$, $j=1,\dots ,6$,
in terms of $\fy$, $p$, $a$ and $b$, see \cite{Coriasco:998.1}.
Finally, when $a\in \SG^{m,\mu}$ is elliptic and $\fy\in\Phr$,
the corresponding SG \fios\ admit a parametrix, that is, 
there exist $b_1,b_2\in \SG^{-m,-\mu}$ such that
	\begin{align}
	\label{eq:parami}
	\Op _\fy(a)\circ \Op _\fy^*(b_1)= \Op _\fy^*(b_1)\circ\Op _\fy(a) 
			&= I\ \mod\ \Op (\SG^{-\infty,-\infty}),
	\\[1ex]
	\label{eq:paramii}
	\Op _\fy ^*(a)\circ \Op _\fy (b_2)=\Op_\fy(b_2)\circ\Op^*_\fy(a) 
			&= I\ \mod\ \Op (\SG^{-\infty,-\infty}),
	\end{align}
	where $I$ is the identity operator. For all these results, as well as for Theorem \ref{thm:fiocont} below,
	see again \cite{Coriasco:998.1,CJT4,CT14}.
	
	\par
	
	\begin{theorem}\label{thm:fiocont}
		Let $\varphi\in\Phr$ and $a\in\SG^{m,\mu}(\R^{2n})$, $m,\mu\in\R$. Then, for any $r,\varrho\in\R$,
		$\Op_\varphi(a)$ and $\Op^*_\varphi(a)$ continuously map $H^{r,\varrho}(\R^n)$ to $H^{r-m,\varrho-\mu}(\R^n)$.
	\end{theorem}
	
	\par

	%
	\subsection{Multi-products of regular SG phase functions and of
				regular SG Fourier integral operators}\label{sec:mpsgphf}
				
	For the convenience of the reader, here 
	we recall a few results from \cite{AsCor}.
	Let us consider a sequence $\{\vp_j\}_{j\in\N}$ of regular
	SG phase functions $\vp_j(x,\xi)\in \mathcal P_r(\tau_j)$, $j\in\N$, with 
	\beqs\label{1/4}
	\ds\sum_{j=1}^\infty \tau_j=:\tau_0<1/4.
	\eeqs
	By \eqref{1/4} we have that  for every $\ell\in\N$
	there exists a constant $c_\ell>0$ such that 
	\beqs\label{marr}
	\|J_k\|_{2,\ell}\leq c_\ell\tau_k\quad\ {\rm and}
		\quad  \ds\sum_{k=1}^\infty\|J_k\|_{2,\ell}\leq c_\ell\tau_0.
	\eeqs
	We set $\overline{\tau}_M=\sum_{j=1}^{M}\tau_j.$
	
	\par
	
	With a fixed integer $M\geq 1$, we denote 
	\begin{align}\label{multivardef}\begin{split}
	(X,\Xi)&=(x_0,x_1,\ldots,x_M,\xi_1,\ldots,\xi_M,\xi_{M+1}):=(x,T,\Theta,\xi),
	\\
	(T,\Theta)&=(x_1,\ldots,x_M,\xi_1,\ldots,\xi_M),
	\end{split}
	\end{align}
	and define the function of $2(M+1)n$ real variables
	\beqs\label{serveilnome}
	\psi(X,\Xi):=\ds\sum_{j=1}^M\left(\vp_j(x_{j-1},\xi_j)-x_j\cdot\xi_j\right)
		+\vp_{M+1}(x_M,\xi_{M+1}).
	\eeqs
	For every fixed $(x,\xi)\in\R^{2n}$, the critical points $(Y,N)=(Y,N)(x,\xi)$ 
	of the function of $2Mn$ variables $\widetilde{\psi}(T,\Theta)=\psi(x,T,\Theta,\xi)$
	are the solutions to the system
	\begin{align}\label{critpntprob}
	\begin{cases}
	\psi'_{\xi_j}(X,\Xi)=\vp'_{j, \xi}(x_{j-1},\xi_{j})-x_j=0 & j=1,\ldots,M,
	\\
	\psi'_{x_j}(X,\Xi)=\vp'_{j+1, x}(x_j,\xi_{j+1})-\xi_j=0 & j=1,\ldots,M,
	\end{cases}
	\end{align}
	in the unknowns $(T,\Theta)$. That is $(Y,N)=(Y_1,\ldots,Y_M,N_1,\ldots,N_M)(x,\xi)$
			satisfies, if $M=1$,
	\beqs\label{(C1)}\begin{cases}
		Y_1(x,\xi)=\vp'_{1, \xi}(x,N_{1}(x,\xi)) 
		\\
		N_1(x,\xi)=\vp'_{2, x}(Y_1(x,\xi),\xi),
	\end{cases}
	\eeqs
	or, if $M\ge2$,
	\beqs\label{(C)}\begin{cases}
		Y_1(x,\xi)=\vp'_{1,\xi}(x,N_{1}(x,\xi)) 
		\\
		Y_j(x,\xi)=\vp'_{j,\xi}(Y_{j-1}(x,\xi),N_{j}(x,\xi)), & j=2,\ldots,M
		\\
		N_j(x,\xi)=\vp'_{j+1, x}(Y_j(x,\xi),N_{j+1}(x,\xi)), & j=1,\ldots,M-1
		\\
		N_M(x,\xi)=\vp'_{M+1, x}(Y_M(x,\xi),\xi) .
	\end{cases}
	\eeqs
	In the sequel we will only refer to the system \eqref{(C)},
	tacitly meaning \eqref{(C1)} when $M=1$.

	\begin{definition}[Multi-product of SG phase functions]\label{mprod}
		If, for every fixed $(x,\xi)\in \R^{2n}$, the system \eqref{(C)}
		admits a unique solution $(Y,N)=(Y,N)(x,\xi)$, we define
		\beqs\label{3.4}
		\phi(x,\xi)=(\vp_1\ \sharp\ \cdots\ \sharp\ \vp_{M+1})(x,\xi):=\psi(x,Y(x,\xi),N(x,\xi),\xi).
		\eeqs 
		The function $\phi$ is called multi-product of the SG phase functions $\vp_1,\ldots,\vp_{M+1}.$
	\end{definition}
	
	\par
	
	The following properties of the multi-product SG phase functions can be found in \cite{AsCor}.
	
	\begin{proposition}\label{stime}
		Under the assumptions \eqref{hyp} and \eqref{1/4}, the system \eqref{(C)} 
		admits a unique solution $(Y,N)$, satisfying
		\beqs\label{324}
		\{(Y_j-Y_{j-1})/\tau_j\}_{j\in\N}\ is\ bounded\ in\ \SG^{1,0}(\R^{2n}),
		\\
		\{(N_j-N_{j+1})/\tau_{j+1}\}_{j\in\N}\ is\ bounded\ in\ \SG^{0,1}(\R^{2n}).
		\eeqs
	\end{proposition}
	
	\par
	
\begin{proposition}\label{properties}
Under the assumptions \eqref{hyp} and \eqref{1/4},
the multi-product $\phi(x,\xi)$ in Definition \ref{mprod} is well defined
for every $M\geq 1$ and has the following properties.
\begin{enumerate}
\item There exists $k\geq 1$ such that
$\phi(x,\xi)=(\varphi_1\ \sharp\ \cdots\ \sharp\ \varphi_{M+1})(x,\xi)\in\mathcal P_r(k\bar\tau_{M+1})$ and,
setting 
$$
J_{M+1}(x,\xi):=(\varphi_1\ \sharp\ \cdots\ \sharp\ \varphi_{M+1})(x,\xi)-x\cdot\xi,
$$
the sequence $\{J_{M+1}/\bar\tau_{M+1}\}_{M\geq1}$ is bounded in $\SG^{1,1}(\R^{2n})$.
\item The following relations hold:
\[\begin{cases}
\phi'_x(x,\xi)=\varphi'_{1, x}(x,N_1(x,\xi))
\\
\phi'_\xi(x,\xi)=\varphi'_{M+1, \xi}(Y_M(x,\xi),\xi),
\end{cases}
\]
	where $(Y,N)$ are the critical points \eqref{(C)}.
\item The associative law holds:
$\varphi_1\ \sharp\ (\varphi_2\ \sharp\ \cdots\ \sharp\ \varphi_{M+1})
=(\varphi_1\ \sharp\ \cdots\ \sharp\ \varphi_{M})\ \sharp\ \varphi_{M+1}$.
\item For any $\ell\geq 0$ there exist $0<\tau^\ast<1/4$ and $c^\ast\geq 1$ such that,
if $\varphi_j\in \mathcal P_r(\tau_j,\ell)$ for all $j$ and $\tau_0\leq \tau^\ast$,
then $\phi\in \mathcal P_r(c^\ast\bar\tau_{M+1},\ell)$.
\end{enumerate}
\end{proposition}
	
	\par
	
	Passing to regular SG Fourier integral operators, one can prove the following algebra properties.
	
	\begin{proposition}\label{thm:main}
		Let $\fy_j\in\Phr(\tau_j)$, $j=1,2, \dots, M$, $M\ge 2$, 
		be such that $\tau_1+\cdots+\tau_M\leq\tau\leq\frac{1}{4}$
		for some sufficiently small $\tau>0$, and set
		\begin{align*}
		\Phi_0(x,\xi)&=x\cdot\xi, 
		\\
		\Phi_1&=\fy_1, 
		\\
		\Phi_j&=\fy_1\sharp\cdots\sharp\fy_j, \; j= 2, \dots, M
		\\ 
		\Phi_{M,j}&=\fy_j\sharp\fy_{j+1}\sharp\cdots\sharp\fy_{M}, j=1,\dots,M-1,
		\\
		\Phi_{M,M}&=\fy_{M}, 
		\\
		\Phi_{M,M+1}(x,\xi)&=x\cdot\xi.
		\end{align*}
		Assume also $a_j\in \SG^{m_j,\mu_j}(\R^{2n})$, and set $A_j=\Op_{\fy_j}(a_j)$, $j=1,\dots,M$. 
		Then, the following properties hold true.
		\begin{enumerate}
			\item Given $q_j, q_{M,j}\in \SG^{0,0}(\R^{2n})$, $j=1,\dots, M$, such that
			\[
			\Op^*_{\Phi_j}(q_j)\circ I_{\Phi_j}=I, \quad I^*_{\Phi_{M,j}}\circ\Op_{\Phi_{M,j}}(q_{M,j})=I,
			\] 
			set $Q_j^*=\Op^*_{\Phi_j}(q_j)$, $Q_{M,j}=\Op_{\Phi_{M,j}}(q_{M,j})$, and
			\[
			R_j=I_{\Phi_{j-1}}\circ A_j\circ Q^*_j, \quad
			R_{M,j}=Q_{M,j}\circ A_j\circ I^*_{\Phi_{M,j+1}}, \quad j=1,\dots,M.
			\]
			Then, $R_j,R_{M,j}\in\Op(\SG^{0,0}(\R^{2n}))$, $j=1,\dots,M$, and
			\begin{equation}\label{eq:prodAj}
			A=A_1\circ\cdots\circ A_M=R_1\circ\cdots\circ R_{M}\circ I_{\Phi_M}
			=I^*_{\Phi_{M,1}}\circ R_{M,1}\circ\cdots\circ R_{M,M}.
			\end{equation}
			\item There exists $a\in \SG^{m,\mu}(\R^{2n})$, $m=m_1+\cdots+m_M$,
			$\mu=\mu_1+\cdots+\mu_M$ such that, setting $\phi=\fy_1\sharp\cdots\sharp\fy_M$,
			\[
			A=A_1\circ\cdots\circ A_M=\Op_{\phi}(a).
			\]
			\item For any $l\in\Z_+$ there exist $l^\prime\in\Z_+$, $C_l>0$ such that 
			\begin{equation}\label{eq:estsna}
			\vvvert a \vvvert_l^{m,\mu} \le C_l\prod_{j=1}^M \vvvert a_j \vvvert_{l^\prime}^{m_j,\mu_j}.
			\end{equation}
		\end{enumerate}
	\end{proposition}

	\par
	\subsection{Eikonal equations and Hamilton-Jacobi systems in SG classes}\label{subsec:eik}
	
	\par
	
	To study evolution equations within the SG environment,
	we need first to introduce parameter-dependent symbols,
	where the parameters give rise to bounded families in $\SG^{m,\mu}$.
	The proof of the next technical Lemma \ref{lem:invariantcomsymbol}
	can be found in \cite{Abdeljawadthesis}.

\par
	
\begin{definition}
Let $\Omega\subseteq\R^N$. We write
$f\in C^k\left(\Omega;\SG^{m,\mu}(\R^{2n})\right)$,
with $m,\,\mu\in \R$ and $k\in\Z_+$ or $k=\infty$, if
\begin{enumerate}[(i)]
\item $f=f(\omega;x,\xi)$, $\omega\in\Omega,\,x,\,\xi\in\R^n$;
\item for any $\omega\in \Omega$,
$\partial^\a_\omega f(\omega)\in \SG^{m,\mu}(\R^{2n})$,
for all $\a\in\Z_+^N,\,|\a|\leq k$;
\item $\{\partial^\a_\omega f(\omega)\}_{\omega\in\Omega}$ is
bounded in $\SG^{m,\mu}(\R^{2n})$,
for all $\a\in\Z_+^N,\,|\a|\leq k$.
\end{enumerate}
\end{definition}
	
\begin{lemma}\label{lem:invariantcomsymbol}
Let $\Omega\subseteq\R^N$,
$a\in C^k\left(\Omega;\SG^{m,\mu}(\R^{2n})\right)$
and $h\in C^k\left(\Omega;\SG^{0,0}(\R^{2n})\otimes\R^N\right)$ 
such that $k\in\Z_+$ or $k=\infty$.
Assume also that, for any $\omega\in \Omega$, $x,\xi\in\R^n$,
the function $h(\omega;x,\xi)$
takes value in $\Omega$.
Then, setting $b(\omega)=a(h(\omega))$,
that is, $b(\omega;x,\xi)=a(h(\omega;x,\xi);x,\xi)$,
we find $b\in C^k\left(\Omega;\SG^{m,\mu}(\R^{2n})\right)$.
\end{lemma}

\par
	
	Given a real-valued symbol $a\in C([0,T]; \SG^{1,1}(\R^{2n}))$,
	consider the so-called \textit{eikonal equation}
	\beqs\label{eik}
	\begin{cases}
		\partial_t\varphi(t,s;x,\xi)=a(t;x,\varphi'_x(t,s;x,\xi)),& t\in [0,T'],
		\\
		\varphi(s,s;x,\xi)=x\cdot\xi,& s\in [0,T'],
	\end{cases}
	\eeqs
	with $0<T'\leq T$. By an extension of the theory developed  in \cite{Coriasco:998.2}, it is possible to
	prove that the following Proposition \ref{trovala!} holds true.
\begin{proposition}\label{trovala!}
Let $a\in C([0,T]; \SG^{1,1}(\R^{2n}))$ be real-valued. Then, for a small enough $T'\in(0,T]$, 
equation \eqref{eik} admits a unique solution
$\varphi\in C^1([0,T']^2;\SG^{1,1}(\R^{2n}))$,
satisfying $J\in C^1([0,T']^2;\SG^{1,1}(\R^{2n}))$
and
\beqs
\partial_s\varphi(t,s;x,\xi)=-a(s;\varphi'_\xi(t,s;x,\xi),\xi),
\eeqs
for any $t,s\in[0,T']$. Moreover, 
for every $h\geq 0$ there exists $c_h\geq 1$ and $T_h\in[0,T']$
such that $\varphi(t,s;x,\xi)\in\mathcal P_r(c_h|t-s|)$, 
with $\| J\|_{2,h}\leq c_h |t-s|$ for all $0\leq s\leq t\leq T_h$.
\end{proposition}
	
	\noindent
	In the sequel we will sometimes write $\varphi_{ts}(x,\xi):=\varphi(t,s;x,\xi)$, for a solution $\varphi$ of \eqref{eik}.
	
	\par
	
	\begin{remark}
		Note that the eikonal equation \eqref{eik} appears in the so-called geometric optics
		approach to the solution of $\cL u=f,\, u(0)=u_0$ for the hyperbolic operator
		\begin{equation}\label{hypoprcortoeikeq}
		\cL=D_t-a(t;x,D_x) \qquad \text{ on } [0,T].
		\end{equation}
		%
	\end{remark}

	\par
	As a simple realization of a sequence of phase functions satisfying \eqref{hyp} and \eqref{1/4},
	we recall the following example, see \cite{AsCor} and \cite{Kumano-go:1}.
\begin{example}
	Let  $\varphi(t,s;x,\xi)$ be the solution of the eikonal equation \eqref{eik}.
	Choose the partition $\Delta_{M+1}(T_1)\equiv\Delta(T_1)$ of
	the interval $[s,t]$, $0\leq s\leq t\leq T_1$, given by
	\begin{equation}\label{partitionof[s,t]}
	s=t_{M+1}\leq t_M\leq\cdots\leq t_1\leq t_0=t,
	\end{equation}
	and define the sequence of phase functions
	\[\chi_j(x,\xi)=
	\begin{cases}
	\varphi(t_{j-1}, t_j;x,\xi) & 1\leq j\leq M+1
	\\
	x\cdot\xi & j\geq M+2.
	\end{cases}\]
	From Proposition \ref{mainprop} we know that $\chi_j\in\mathcal P_r(\tau_j)$
	with $\tau_j=c_0(t_{j-1}-t_j)$ for $1\leq j\leq M+1$ and with $\tau_j=0$ for $j\geq M+2$. 
	Condition \eqref{1/4} is fulfilled by the choice of a small enough positive constant $T_1$, since
	\[\ds\sum_{j=1}^\infty \tau_j=\ds\sum_{j=1}^{M+1} c_0(t_{j-1}-t_j)=c_0(t-s)\leq c_0T_1<\frac14\]
	if $T_1<(4c_0)^{-1}$. Moreover, again from Proposition \ref{mainprop},
	we know that $\| J_j\|_{2,0}\leq c_0 |t_j-t_{j-1}|=\tau_j$ for all $1\leq j\leq M+1$
	and $J_j=0$ for $j\geq M+2,$ so each one of the $J_j$ satisfies \eqref{hyp}. 
\end{example}
	\par
	We now focus on the Hamilton-Jacobi system corresponding to the real-valued Hamiltonian
	$a\in C([0,T];\SG^{1,1}(\R^{2n}))$, namely,
	\par
	\begin{eqnarray}\label{hameq}
	\left\{\begin{array}{ll}
	\partial_t{q}(t,s;y,\eta)&=- a^\prime_\xi\left(t;q(t,s;y,\eta),p(t,s;y,\eta)\right),
	\\
	\partial_t{p}(t,s;y,\eta)&=\phantom{-} a^\prime_x(t;q(t,s;y,\eta),p(t,s;y,\eta)),
	\end{array}
	\right.
	\end{eqnarray}
	where $t,s\in[0,T]$,  $T>0$, and the Cauchy data
	\begin{eqnarray}
	\label{hameqCauchyData}\left\{\begin{array}{cc}
	q(s,s;y,\eta)=y,\\
	p(s,s;y,\eta)=\eta.
	\end{array}\right.
	\end{eqnarray}

	\par

		%

\par
	
We will now recall how the solution of \eqref{hameq},
\eqref{hameqCauchyData} is related
with the solution of \eqref{eik} in the SG context.
We here mainly refer to known results from \cite[Ch. 6]{CO}
and \cite{Coriasco:998.2}.
\par
\begin{proposition}\label{prop:equizeroder}
Let $a\in C\left([0,T];\SG^{1,1}(\R^{2n})\right)$ be real-valued. 
Then, the solution $(q,p)(t,s;y,\eta)$ of the Hamilton-Jacobi
system \eqref{hameq} with the Cauchy data \eqref{hameqCauchyData} satisfies
\par
\begin{equation}\label{eq:equiv}
\jb{q(t,s;y,\eta)}\sim \jb{y},\qquad \jb{p(t,s;y,\eta)}\sim \jb{\eta}.
\end{equation}
\end{proposition}
\par
\begin{proposition}
Under the same hypotheses of Proposition \ref{prop:equizeroder},
the maximal solution of the Hamilton-Jacobi system \eqref{hameq} with
the Cauchy data \eqref{hameqCauchyData} is defined on the whole
product of intervals $[0,T]\times [0,T]$.
\end{proposition}
\par
\begin{proposition}\label{prop:classofqandp}
The solution  $(q,p)$ of the Hamilton-Jacobi system \eqref{hameq}
with $a\in C^\infty([0,T];\SG^{1,1}(\R^{2n}))$ real-valued, and the
Cauchy data \eqref{hameqCauchyData}, satisfies
\begin{description}
\item[i] {$q$ belongs to $C^\infty([0,T]^2;\SG^{1,0}(\R^{2n}))$,}
\item[ii] {$p$ belongs to $C^\infty([0,T]^2;\SG^{0,1}(\R^{2n}))$.}
\end{description}
\end{proposition}
	
\par
	
\begin{lemma}\label{lem:bndedlemhamsol}
The following statements hold true.
\begin{description}
\item[i\label{lem:classofpandqfirstpart}]
{Let $a\in C([0,T];\SG^{1,1}(\R^{2n}))$ be real-valued.
Then, the solution $(q,p)(t,s;y,\eta)$ of the Hamilton-Jacobi
system \eqref{hameq} with Cauchy data \eqref{hameqCauchyData} satisfies,
for a sufficently small $T'\in (0,T]$ and a fixed $t$ such that
$0\leq s,t\leq T',\ s\neq t$,}
\begin{eqnarray}\label{SGboundednessOfHamiltonSolution}
\left\{\begin{array}{ccc} 
{(q(t,s;y,\eta)-y)/(t-s)} &\text{ is bounded in }&\SG^{1,0}(\R^{2n})
\\[1ex]
{(p(t,s;y,\eta)-\eta)/(t-s)} &\text{ is bounded in }&\SG^{0,1}(\R^{2n})
\end{array}	\right.
\end{eqnarray}
and
\begin{eqnarray}\label{continuityofp-yandpinbanachspace}
\left\{\begin{array}{cc}
q(t,s;y,\eta),q(t,s;y,\eta)-y \in C^1(I(T');\SG^{1,0}(\R^{2n})),
\\[1ex]
p(t,s;y,\eta),p(t,s;y,\eta)-\eta \in C^1(I(T');\SG^{0,1}(\R^{2n})),
\end{array}	\right.
\end{eqnarray}
where, for $T>0$, $I(T)=\{(t,s)\colon 0\leq t,s\leq T\}$.
\item[ii\label{lem:classofpandqsecondpart}]
{
	Furthermore,
if, additionally, $a$ belongs to $C^\infty([0,T];\SG^{1,1}(\R^{2n}))$,
then, $q(t,s;y,\eta)-y\in C^\infty(I(T');\SG^{1,0}(\R^{2n}))$
and $p(t,s;y,\eta)-\eta\in C^\infty(I(T');\SG^{0,1}(\R^{2n}))$.
}
\end{description}
\end{lemma}
	\par
The proof of Lemma \ref{lem:bndedlemhamsol} combines techniques and results
similar to those used in \cite{CO}, \cite{Coriasco:998.2} and \cite{Kumano-go:1},
so we omit the details.
\par
Now, we observe that there exists a constant 
$T_1\in( 0, T']$ such that $q(t,s;y,\eta)$ is invertible
with respect to $y$ for any $(t,s)\in I(T_1)$ and any $\eta\in \R^n$.
Indeed, this holds by continuity and the fact that
$$
q(s,s;y,\eta)=y\Rightarrow\displaystyle\frac{\partial q}{\partial y}(s,s;y,\eta)=I_n.
$$
\par
We denote the inverse function by $\overline{q}$,
that is $$
y=\overline{q}(t,s)=\overline{q}(t,s;x,\eta) \Leftrightarrow x=q(t,s;y,\eta),
$$ 
which exists on $I(T_1)$. Moreover,  $\overline{q}\in C^\infty(I(T_1);\SG^{1,0}(\R^{2n}))$,
cf. \cite{CO,Coriasco:998.2}.
\par
Observe now that, in view of Lemma \ref{lem:bndedlemhamsol},
$\|\frac{\partial q}{\partial y}-I_n\|\rightarrow 0$ when $t\rightarrow s$,
uniformly on $I(T_1)$. Then, one can deduce the following result,
which is an extension to the analogous ones that can be found e.g., in \cite{CO,Kumano-go:1}.
\par
\begin{lemma}\label{lem:bndedleminversehamsol}
Assume that $a\in C([0,T];\SG^{1,1}(\R^{2n}))$ is real-valued,
and let $T_1\in (0,T']$, $\epsilon_1\in(0,1]$
be constants such that on $I(T_1)$ we have
\begin{equation}\label{contractioncont}
\left\|\frac{\partial q}{\partial y}-I_n\right\|\leq 1-\e_1.
\end{equation}
Then, the mapping
$x=q(t,s;y,\xi)\,:\,\R_y^n\ni y\longmapsto x\in\R_x^n$
with $(t,s,\xi)$
understood as parameter, has the inverse function
$y=\overline{q}(t,s;x,\xi)$ satisfying
\begin{align}\label{bndoftheinvfunc}
\left\{\begin{array}{ll}
\overline{q}(t,s;x,\xi)-x\; \text{belongs to } C^1(I(T_1);\SG^{1,0}(\R^{2n})),
\\
\{(\overline{q}(t,s;x,\xi)-x)/|t-s|\}\text{ is bounded in } \SG^{1,0}(\R^{2n}),
\text{ whenever } 0\leq s,t\leq T_1; s\neq t.
\end{array}
\right.
\end{align}
Moreover, if, additionally, $a\in C^\infty([0,T];\SG^{1,1}(\R^{2n}))$,
we also have $\overline{q},\overline{q}(t,s;x,\xi)-x\in C^\infty(I(T_1);\SG^{1,0}(\R^{2n}))$. 
\end{lemma}
	\par
	The proof of the next Proposition \ref{mainprop} follows by slight modifications of the classical
	arguments given, e.g., in \cite{Kumano-go:1} (see \cite{Abdeljawadthesis} for details).
	\begin{proposition}\label{mainprop}
		Let $a\in C([0,T];\SG^{1,1}(\R^{2n}))$ be real-valued, $q(t,s;y,\eta),\,p(t,s;y,\eta)$
		and $\overline{q}(t,s;x,\xi)$ be the symbols constructed in the previous
		Lemmas \ref{lem:bndedlemhamsol}  and \ref{lem:bndedleminversehamsol}. 
		We define $u(t,s;y,\eta)$ by
		\begin{multline}\label{integralsolform}
		u(t,s;y,\eta)=y\cdot\eta + 
		\\
		\int_s^t
		\left(a(\tau;q(\tau,s;y,\eta),p(\tau,s;y,\eta))-a'_\xi(\tau;q(\tau,s;y,\eta),p(\tau,s;y,\eta))\cdot p(\tau,s;y,\eta)\right)\,d\tau,
		\end{multline}
		and   
		\begin{equation}\label{eq:phfundefbyhamsyssol}
		\varphi(t,s;x,\xi)=u(t,s;\overline{q}(t,s;x,\xi),\xi).
		\end{equation}
		Then, $\varphi(t,s;x,\xi)$ is a solution of the eikonal equation \eqref{eik} and satisfies
		\begin{align}
		\label{xivarphiequal}\varphi'_\xi(t,s;x,\xi)&=\phantom{-}\overline{q}(t,s;x,\xi),
		\\
		\label{xvarphiequal} \varphi'_x(t,s;x,\xi)&=\phantom{-}p(t,s;\overline{q}(t,s;x,\xi),\xi),
		\\
		\label{svarphiequal}\partial_s\varphi(t,s;x,\xi)&=-a(s;\varphi'_\xi(t,s;x,\xi),\xi),
		\\
		\label{equivonderivofvarphi}\jb{\varphi'_x(t,s;x,\xi)}&\sim\jb{\xi}
		\text{ and }\jb{\varphi'_\xi(t,s;x,\xi)}\sim\jb{x}.
		\end{align}
		Moreover, for any $l\geq0$ there exists a constant
		$c_l\geq 1$ and $T_l\in(0,T_1]$ such that $c_lT_l<1,\, \varphi(t,s;x,\xi)$
		belongs to $\Phr(c_l|t-s|,l)$ and $\{J(t,s)/|t-s|\}$ is bounded in $\SG^{1,1}(\R^{2n})$
		for $0\leq t,s\leq T_l\leq T_1,\,t\neq s,$ where $J(t,s;x,\xi)=\varphi(t,s;x,\xi)-x\cdot\xi$.
	\end{proposition}
	\par

	\begin{corollary}\label{CinftyclassofJ}
		Let $a\in C([0,T];\SG^{1,1}(\R^{2n}))$ be real-valued, and let $q,\,p$ and $\overline{q}$ be the symbols
		constructed in Lemma \ref{lem:bndedlemhamsol}  and Lemma  \ref{lem:bndedleminversehamsol},
		respectively.
		Then, $J\in C^1(I(T_1);\SG^{1,1}(\R^{2n}))$. Moreover, if, additionally, $a\in C^\infty([0,T];\SG^{1,1}(\R^{2n}))$,
		we find $J\in C^\infty(I(T_1);\SG^{1,1}(\R^{2n}))$.
	\end{corollary}
	
	\subsection{Classical symbols of SG type}\label{subsec:zpower}
In the last part of Section \ref{sec:sginvcp} we will focus on the subclass of symbols and operators which are SG-classical, that is,
 $a\in \SG^{m,\mu}_\cl\subset \SG^{m,\mu}$. In this subsection
we recall their definition and summarize some of their main properties, using materials coming from \cite{BC10a} (see, e.g., \cite{EgorovSchulze} for additional details and proofs).

In the sequel of this subsection we denote  by $\widetilde{\mathscr{H}}_\xi^{m}(\R^n)$ the space of homogeneous functions of order $m$ with respect to the variable $\xi$, smooth with respect to the variable $x$,  and by $\widetilde{\mathscr{H}}_x^{\mu}(\R^n)$ the space of homogeneous functions of order $\mu$ with respect to the variable $x$, smooth with respect to the variable $\xi$.
\begin{definition}
\label{def:sgclass-a}
\begin{itemize}
\item[i)]A symbol $a(x, \xi)$ belongs to the class $\SG^{m,\mu}_{\cl(\xi)}(\R^{2n})$ if there exist $a_{m-i, \cdot} (x, \xi)\in \widetilde{\mathscr{H}}_\xi^{m-i}(\R^n)$, $i=0,1,\dots$, such that, for a $0$-excision function $\omega$,
\[
a(x, \xi) - \sum_{i=0}^{N-1}\omega(\xi) \, a_{m-i, \cdot} (x, \xi)\in \SG^{m-N, \mu}(\R^n), \quad N=1,2, \ldots;
\]
\item[ii)]A symbol $a(x, \xi) $ belongs to the class $\SG_{\cl(x)}^{m,\mu}(\R^{2n})$ if there exist $a_{\cdot, \mu-k}(x, \xi)\in \widetilde{\mathscr{H}}_x^{\mu-k}(\R^n)$, $k=0,\,\dots$, such that, for a $0$-excision function $\omega$,
\[
a(x, \xi)- \sum_{k=0}^{N-1}\omega(x) \, a_{\cdot, \mu-k}(x,\xi) \in \SG^{m, \mu-N}(\R^n), \quad N=1,2, \ldots
\]
\end{itemize}
\end{definition}
\begin{definition}
\label{def:sgclass-b}
A symbol $a(x,\xi)$ is SG-classical, and we write $a \in \SG_{\cl(x,\xi)}^{m,\mu}(\R^{2n})=\SG_{\cl}^{m,\mu}(\R^{2n})=\SG_{\cl}^{m,\mu}$, if
\begin{itemize}
\item[i)] there exist $a_{m-j, \cdot} (x, \xi)\in \widetilde{\mathscr{H}}_\xi^{m-j}(\R^n)$ such that, 
for a $0$-excision function $\omega$, $\omega(\xi) \, a_{m-j, \cdot} (x, \xi)\in \SG_{\cl(x)}^{m-j, \mu}(\R^n)$ and
\[
a(x, \xi)- \sum_{j=0}^{N-1} \omega(\xi) \, a_{m-j, \cdot}(x, \xi) \in \SG^{m-N, \mu}(\R^n), \quad N=1,2,\dots;
\]
\item[ii)] there exist $a_{\cdot, \mu-k}(x, \xi)\in \widetilde{\mathscr{H}}_x^{\mu-k}(\R^n)$ such that, 
for a $0$-excision function $\omega$, $\omega(x)\,a_{\cdot, \mu-k}(x, \xi)\in \SG_{\cl(\xi)}^{m, \mu-k}(\R^n)$ and
\[
a(x, \xi) - \sum_{k=0}^{N-1} \omega(x) \, a_{\cdot, \mu-k} \in \SG^{m, \mu-N}(\R^n), \quad N=1,2,\dots
\] 
\end{itemize}
We set $L_{\cl(x, \xi)}^{m,\mu}(\R^{2n})=L_{\cl}^{m,\mu}(\R^{2n})=\Op(\SG^{m,\mu}_{\cl}(\R^{2n}))$.
\end{definition}

\noindent
The next two results are especially useful when dealing with SG-classical symbols.

\begin{theorem}
	\label{thm4.6}
	Let $a_{k} \in \SG_{\cl}^{m-k,\mu-k}(\R^{2n})$, $k =0,1,\dots$,
	be a sequence of SG-classical symbols
	and $a \sim \sum_{k=0}^\infty a_{k}$
	its asymptotic sum in the general SG-calculus.
	Then, $a \in \SG_{\cl}^{m,\mu}(\R^{2n})$.
\end{theorem}

\begin{theorem}
	\label{thm4.6.1.1}
	Let $\mathbb{B}^n= \{ x \in \R^n: |x| \le 1 \}$  and let $\chi$ be a 
	diffeomorphism from the interior of $\mathbb{B}^n$ to $\R^{n}$ such that
	\[ 
	\chi(x) = \displaystyle \frac{x}{|x|(1-|x|)} \quad \mbox{for} \quad |x| > 2/3.
	\]
	Choosing a smooth function $[x]$ on $\R^n$ such that
	$1 - [x] \not= 0$ for all $x$ in the interior of $\mathbb{B}^n$
	and $|x| > 2/3 \Rightarrow [x] = |x|$,
	for any $a \in \SG^{m,\mu}_{\cl}(\R^{2n})$ denote by $(D^\m a)(y, \eta)$,
	$\m=(m,\mu)$, the function
	\begin{equation}
		\label{eq4.26.1}
		b(y,\eta) = (1-[\eta])^{m_{1}} (1-[y])^{m_{2}}
		a(\chi(y), \chi(\eta)).
	\end{equation}
	Then, $D^\m$ extends to a homeomorphism from $\SG^{m,\mu}_{\cl}(\R^{2n})$ to 
	$C^\infty(\mathbb{B}^n \times \mathbb{B}^n)$.
\end{theorem}

\noindent
Note that the definition of SG-classical symbol implies a condition of compatibility for the terms of the expansions with respect to $x$ and $\xi$. 
In fact, defining $\sigma_\psi^{m-j}$ and $\sigma_e^{\mu-i}$ on $\SG_{\cl(\xi)}^{m,\mu}$ and $\SG_{\cl(x)}^{m,\mu}$, respectively,  as
\begin{align*}
	\sigma_\psi^{m-j}(a)(x, \xi) &= a_{m-j, \cdot}(x, \xi),\quad j=0, 1, \ldots, 
	\\
	\sigma_e^{\mu-i}(a)(x, \xi) &= a_{\cdot, \mu-i}(x, \xi),\quad i=0, 1, \ldots,
\end{align*}
it possibile to prove that
\[
\begin{split}
a_{m-j,\mu-i}=\sigma_{\psi e}^{m-j,\mu-i}(a)=\sigma_\psi^{m-j}(\sigma_e^{\mu-i}(a))= \sigma_e^{\mu-i}(\sigma_\psi^{m-j}(a)), \\
j=0,1, \ldots, \; i=0,1, \ldots
\end{split}
\]
Moreover, the algebra property of SG-operators and Theorem \ref{thm4.6} imply that the composition of two SG-classical operators is still classical. 
For $A=\Op(a)\in L^{m,\mu}_\cl$ the triple $\sigma(A)=(\sigma_\psi(A),\sigma_e(A),\sigma_{\psi e}(A))=(a_{m,\cdot}\,,\,
a_{\cdot,\mu}\,,\, a_{m,\mu})$ is called the \textit{principal symbol of $A$}. This definition keeps the usual multiplicative behaviour, that is, for any $A\in L^{r,\varrho}_\cl$, $B\in L^{s,\s}_\cl$, 
$r,\varrho,s,\s\in\R$,
$\sigma(AB)=\sigma(A)\,\sigma(B)$, with componentwise product in the right-hand side. We also set 
\begin{align*}
	\Symp{A}(x,\xi) = &\;\,\Symp{a}(x,\xi) =\\
	=
	&\;\;\omega(\xi) a_{m,\cdot}(x,\xi) +
	\omega(x)(a_{\cdot,\mu}(x,\xi) - \omega(\xi) a_{m,\mu}(x,\xi))
\end{align*}
for a $0$-excision function $\omega$. Theorem \ref{thm:ellclass} below allows to express the ellipticity
of SG-classical operators in terms of their principal symbol.
\begin{theorem}
	 \label{thm:ellclass}
	An operator $A\in L^{m,\mu}_\cl(\R^{2n})$ is elliptic if and only if each element of the triple $\sigma(A)$ is 
	everywhere non-vanishing on its domain of definition.
\end{theorem}

\begin{remark}
	The composition results in the previous Subsection \ref{sec:fioprelim} have \textit{classical counterparts}. Namely,
	when all the involved starting elements are SG-classical, the resulting objects (multi-products, amplitudes, etc.) are
	SG-classical as well.
\end{remark}

	\section{Commutative law for multi-products of regular SG-phase functions}\label{sec:sgphfcomlaw}
	\setcounter{equation}{0}
	%
	%
	\par
	In this section we prove a commutative law for the multi-products of regular SG phase functions.
	Through this result, we further expand the theory of SG Fourier integral operators,
	which we will be able to apply to obtain the solution of Cauchy problems
	for weakly hyperbolic linear differential operators, with polynomially bounded coefficients,
	variable multiplicities and involutive characteristics.
	
	More precisely, we focus on the $\sharp$-product of regular SG phase functions
	obtained as solutions to eikonal equations. Namely, let $[\varphi_j(t,s)](x,\xi)=\varphi_j(t,s;x,\xi)$ be the phase functions
	defined by the eikonal equations \eqref{eik}, with $\varphi_j$
	in place of $\varphi$ and $a_j$ in place of $a$, 
	where $a_j \in C([0,T];\SG^{1,1})$, $a_j$ real-valued, $j\in\N$. 
	Moreover, let $I_{\varphi_j}(t,s)=I_{\varphi_j(t,s)}$ be the SG Fourier integral operator
	with phase function $\varphi_j(t,s)$ and symbol identically equal to $1$.
	\par

	Assume that $\{a_j\}_{j\in\N}$ is bounded in $C([0,T];\SG^{1,1})$.
	Then, by the above Proposition \ref{mainprop},
	there exists a constant $c$, independent of $j$, such that 
	\begin{equation*}
	\varphi_j(t,s)\in\Phr(c|t-s|),\, j\in\N.
	\end{equation*}
	We make a choice of $T_1$, once and for all, such that $cT_1\leq\tau_0$
	for the constant $\tau_0$ defined in \eqref{1/4}. 
	Moreover, for convenience below, we define, for $M\in\Z_+$,
\begin{equation}\label{eq:bft}
\begin{aligned}
	\bt_{M+1}&=(t_0,\dots,t_{M+1})\in\Delta(T_1),
	\\[1ex]
	\bt_{M+1,j}(\tau)&=(t_0,\dots,t_{j-1},\tau,t_{j+1},\dots,t_{M+1}), \; 0\le j \le M+1,
	\end{aligned}
\end{equation} 
where $\Delta(T_1) =\Delta_{M+1}(T_1)= \{(t_0,\dots,t_{M+1})\colon 0\leq t_{M+1}\leq t_{M}\leq\dots\leq t_0\leq T_1\}$.
	\par
\subsection{Parameter-dependent multi-products of regular SG phase functions}\label{subsec:parpdepphase}
	\par
	Let $M\geq 1$ be a fixed integer. We have the following well defined multi-product
	\begin{equation}\label{sharpprodofphi}
	\phi(\bt_{M+1};x,\xi)=[\varphi_1(t_0,t_1)\sharp\varphi_2(t_1,t_2)
	\sharp\dots\sharp\varphi_{M+1}(t_M,t_{M+1})](x,\xi),
	\end{equation}
	where we will often set 
	$t_0=t$, $t_{M+1}=s$, 
	for
	$\bt_{M+1}\in\Delta(T_1)$ from \eqref{eq:bft}.
	Explicitly, $\phi$ is defined as in \eqref{3.4},
	by means of the critical points $(Y,N)=(Y,N)(\bt_{M+1};x,\xi)$,
	obtained, when $M\ge2$, as solutions of the system
	\begin{equation}\label{Ct}\begin{cases}
	Y_1(\bt_{M+1};x,\xi)
	=\vp'_{1,\xi}(t_0,t_1;x,N_{1}(\bt_{M+1};x,\xi)) 
	\\
	Y_j(\bt_{M+1};x,\xi)
	=\vp'_{j,\xi}(t_{j-1},t_j;Y_{j-1}(\bt_{M+1};x,\xi),N_{j}(\bt_{M+1};x,\xi)), & j=2,\ldots,M
	\\
	N_j(\bt_{M+1};x,\xi)
	=\vp'_{j+1, x}(t_j,t_{j+1};Y_j(\bt_{M+1};x,\xi),N_{j+1}(\bt_{M+1};x,\xi)), & j=1,\ldots,M-1
	\\
	N_M(\bt_{M+1};x,\xi)
	=\vp'_{M+1, x}(t_M,t_{M+1}; Y_M(\bt_{M+1};x,\xi),\xi),
	\end{cases}
	\end{equation}
	namely,
	\begin{equation}\label{theM+1sharpproductphi}
	\begin{aligned}
	\phi&(\bt_{M+1};x,\xi):=\\
	&\sum_{j=1}^M\left[
	\vp_j(t_{j-1},t_j;Y_{j-1}(\bt_{M+1};x,\xi),N_j(\bt_{M+1};x,\xi))
		-Y_j(\bt_{M+1};x,\xi)\cdot N_j(\bt_{M+1};x,\xi)
	\right]
	\\
	+&\phantom{\sum_{j=1}^M}\vp_{M+1}(t_M,t_{M+1};Y_M(\bt_{M+1};x,\xi),\xi).
	\end{aligned}
	\end{equation}
	Next, we give some properties of the multi-product $\phi$, whose proofs can be found in \cite{Abdeljawadthesis}.
	\begin{proposition}\label{prop:tderivofmultiph}
		Let $\phi$ be the multi-product \eqref{theM+1sharpproductphi}, with real-valued $a_j\in C([0,T];\SG^{1,1}(\R^{2n}))$, $j=1,\dots,M+1$.
		Then, the following properties hold true.
		\begin{description}
			\item[i\label{i.propofsharprod}] 
			\begin{align}\label{tderofmultprodphasefunc}
			\begin{split}
			\left\{\begin{array}{cl}
			\partial_{t_0}\phi(\bt_{M+1};x,\xi)&=\phantom{-}a_1(t_0;x,\phi'_x(\bt_{M+1};x,\xi))\\
			\partial_{t_j}\phi(\bt_{M+1};x,\xi)&=\phantom{-}a_{j+1}(t_j;Y_j(\bt_{M+1};x,\xi),N_j(\bt_{M+1};x,\xi))
			\\
			& \phantom{=}-a_j(t_j;Y_j(\bt_{M+1};x,\xi),N_j(\bt_{M+1};x,\xi)),\; j=1,\dots,M\\
			\partial_{t_{M+1}}\phi(\bt_{M+1};x,\xi)&=-a_{M+1}(t_{M+1},\phi'_\xi(\bt_{M+1};x,\xi),\xi).
			\end{array}\right.
			\end{split}
			\end{align}
			\item[ii\label{ii.propofsharprod}] For any $\varphi_j$ solution
			to the eikonal equation associated with the Hamiltonian $a_j$, we have
			\begin{align*}
			\begin{split}
			\left\{\begin{array}{cc}				
			\varphi_i(t,s)\sharp\varphi_j(s,s)&=\varphi_i(t,s)\\
			\varphi_i(s,s)\sharp\varphi_j(t,s)&=\varphi_j(t,s),
			\end{array}\right.
			\end{split}
			\end{align*}
			for all $i,j=1,\dots,M+1$.
		\end{description}
	\end{proposition}
\begin{proposition}\label{propcritpntbound}
Let $\{a_j\}_{j\in\N}$ be a family of parameter-dependent, real-valued symbols, bounded in $C^\infty([0,T];\SG^{1,1}(\R^{2n}))$,
and $(Y,N)$ be the solution of \eqref{Ct}.
Then, for $\g_k\in\Z_+$, $k=0,1\dots,M+1$, the following properties hold true.
\begin{description}
\item[i\label{bndoncritpnts}] 	
\begin{align}\label{eq:bndcritpnts}
\begin{cases}
\{\partial_{t_0}^{\g_0}\dots\partial_{t_{M+1}}^{\g_{M+1}}
\left(Y_j-Y_{j-1}\right)\}_{j\in\N}\text{  is bounded in }\SG^{1,0}(\R^{2n}),
\\
\{\partial_{t_0}^{\g_0}\dots\partial_{t_{M+1}}^{\g_{M+1}}
\left(N_j-N_{j+1}\right)\}_{j\in\N}\text{  is bounded in }\SG^{0,1}(\R^{2n}).
\end{cases}
\end{align}
\item[ii\label{bndonJfunc}] For $J_{M+1}(\bt_{M+1};x,\xi)=\phi(\bt_{M+1};x,\xi)-x\cdot\xi$,
we have
\begin{equation*}
\{\partial_{t_0}^{\g_0}\dots\partial_{t_{M+1}}^{\g_{M+1}}J_{M+1}\}
\text{ is bounded in }\SG^{1,1}(\R^{2n}),
\end{equation*}
\end{description}
\end{proposition}
	
	\par
	
\begin{corollary}
Under the hypotheses of Proposition \ref{propcritpntbound}, we have, for some $T_1\in(0,T]$ as in Proposition \ref{mainprop}, and $j=0,\dots,M+1$,
\begin{align*}
\begin{cases}
Y_j\text{ belongs to }C^\infty\left(\Delta(T_1);\SG^{1,0}(\R^{2n})\right),
\\
N_j\text{ belongs to } C^\infty\big(\Delta(T_1);\SG^{0,1}(\R^{2n})\big).
\end{cases}
\end{align*}

\end{corollary}
	
\par
	
For any parameter-dependent, real-valued family $\{a_j\}_{j\in\N}\subset C^\infty([0,T];\SG^{1,1})$, we consider
the solutions $(q_j,p_j)(t,s;y,\eta)$ of the Hamilton-Jacobi system \eqref{hameq},
with the Hamiltonian $a_j$ in place of $a$, $j\in\N$.
We define the trajectory
$(\tilde{q}_{j},\tilde{p}_{j})(\bt_{j-1},\s;y,\eta)$,
for $(y,\eta)\in\R^n\times\R^n$, $\bt_{j-1}\in\Delta_{j-1}(T_1)$, $\s\in[t_j, t_{j-1}]$, $j\in\N$, by
\begin{equation}\label{trajectory}
\begin{alignedat}[t]{3}
(\tilde{q}_1,\tilde{p}_1)(t_0,\s;y,\eta)&=(q_1,p_1)(\s,t_0;y,\eta),
	&& t_1\leq\s\leq t_0,
\\
(\tilde{q}_{j},\tilde{p}_{j})(\bt_{j-1},\s;y,\eta)&=(q_j,p_j)(\s,t_{j-1};(\tilde{q}_{j-1},\tilde{p}_{j-1})(\bt_{j-1};y,\eta)), \;
&& t_j\leq\s\leq t_{j-1}, j\geq 2.
\end{alignedat}
\end{equation}

\par

\begin{proposition}\label{prop:trajec}
Let $(Y,N)=(Y,N)(\bt_{M+1};x,\xi)$ be the solution of \eqref{Ct} under the hypotheses of Proposition \ref{propcritpntbound}.
Then, we have
\beqs\label{trajequal1}
\begin{cases}
	q_1(t_1,t_0;x,\phi'_x(\bt_{M+1};x,\xi))=Y_1(\bt_{M+1};x,\xi)\\
	p_1(t_1,t_0;x,\phi'_x(\bt_{M+1};x,\xi))=N_1(\bt_{M+1};x,\xi),
\end{cases}
\eeqs
\begin{equation}\label{trajequal2}
\begin{cases}
q_j(t_j,t_{j-1};Y_{j-1}(\bt_{M+1};x,\xi),N_{j-1}
(\bt_{M+1};x,\xi))=Y_j(\bt_{M+1};x,\xi)
\\[1ex]
p_j(t_j,t_{j-1};Y_{j-1}(\bt_{M+1};x,\xi),N_{j-1}
(\bt_{M+1};x,\xi))=N_j(\bt_{M+1};x,\xi),\qquad 2\leq j\leq M,
\end{cases}
\end{equation}
and, for any $j\leq M$,
\begin{equation}\label{subtaged}
(\tilde{q}_{j},\tilde{p}_{j})
		(\bt_{j};x,\phi'_x(\bt_{M+1};x,\xi))
	\tag{\theequation$_j$}
=(Y_j,N_j)(\bt_{M+1};x,\xi).
\end{equation}
\end{proposition}
	
\par
		
Notice that, from the Propositions \ref{propcritpntbound} and \ref{prop:trajec},
we obtain the following one.
	
\par
	
\begin{proposition}\label{trajclass}
Let $\{a_j\}_{j\in\N}$ be a family of parameter-dependent, real-valued symbols, bounded in $C^\infty([0,T];\SG^{1,1}(\R^{2n}))$.
Then we have, for the trajectory
$\left(\tilde{q}_{j},\tilde{p}_{j}\right)
(\bt_{j-1},\s;y,\eta)$
defined in \eqref{trajectory},
\begin{align*}
\begin{cases}
\{\partial_{t_0}^{\g_0}\dots\partial_{t_{j-1}}^{\g_{j-1}}
\partial_\s^{\g_{j}}\tilde{q}_{j}\}_{j\in\N,\,\bt_{j-1}\in\Delta_{j-1}(T_1),\, \s\in[t_j, t_{j-1}]}
\text{ is bounded in }\SG^{1,0}(\R^{2n}),
\\[1ex]
\{\partial_{t_0}^{\g_0}\dots\partial_{t_{j-1}}^{\g_{j-1}}
\partial_\s^{\g_{j}}\tilde{p}_{j}\}_{j\in\N,\,\bt_{j-1}\in\Delta_{j-1}(T_1),\, \s\in[t_j, t_{j-1}]}
\text{ is bounded in }\SG^{0,1}(\R^{2n}),
\end{cases}
\end{align*}
where $\g_k\in\Z_+$ for $0\leq k\leq j$.
\end{proposition}
	
\par
	
\subsection{A useful quasilinear auxiliary PDE}\label{subs:auxpde}
Consider the following quasi-linear partial differential equation
	\begin{align}\label{quasi}
	\begin{cases}
	\partial_{t_{j-1}} \Upsilon(\bt_{M+1})-H(\Upsilon(\bt_{M+1}),\bt_{M+1})
		\cdot \Upsilon^\prime_x(\bt_{M+1})-G(\bt_{M+1,j}(\Upsilon(\bt_{M+1})))=0,
		\\[1ex]
		\Upsilon|_{t_{j-1}=t_j}=t_{j+1}
	\end{cases}	
\end{align}
where, for $s\in \R$, $\bt_{M+1,j}(s)$ is defined in \eqref{eq:bft}, 
$\Upsilon(\bt_{M+1})=\Upsilon(\textbf{t}_{M+1};x,\xi)\in C^\infty(\Delta(T_1);\SG^{0,0})$,
and $H(\tau,\bt_{M+1})=L(\tau,\textbf{t}_{M+1};x,\xi)$ is a vector-valued family of symbols
of order $(1,0)$ such that
$H\in C^\infty([t_{j+1},t_{j-1}]\times\Delta({T_1});\SG^{1,0}\otimes\R^n)$.
We also assume that 
$G(\bt_{M+1,j}(\tau))=G(\textbf{t}_{M+1,j}(\tau);x,\xi)$
belongs to $C^\infty(\Delta(T_1);\SG^{0,0})$, and satisfies 
\[
	G(\bt_{M+1};x,\xi)>0,	\quad G(\bt_{M+1};x,\xi)|_{t_{j}=t_{j-1}}\equiv 1,
\]
for any $\bt_{M+1}\in\Delta(T_1)$, $(x,\xi)\in\R^{2n}$.
	
\par
	
The following Lemma \ref{lem:exitOfSolForOrdin} (cf. \cite{Taniguchi02}) is the key step to prove
our first main result, Theorem \ref{cruthm}.
In fact, it gives the solution of the characteristics system
\begin{align}\label{ordin}
\left\{\begin{array}{ll}
&\partial_{t_{j-1}} R(\bt_{M+1})=
	-H(K(\bt_{M+1}),\bt_{M+1};R(\bt_{M+1}),\xi)
\\[1ex]
&\partial_{t_{j-1}}K(\bt_{M+1})=\phantom{-}G(\bt_{M+1,j}
	(K(\bt_{M+1}));R(\bt_{M+1}),\xi)
\\[1ex]
&R|_{t_{j-1}=t_j}=y,\,K|_{t_{j-1}=t_j}=t_{j+1},
\end{array}
\right.
\end{align}
which then easily provides the solution to the quasi-linear equation \eqref{quasi}. The latter,
in turn, is useful to simplify the computations in the proof ot Theorem \ref{cruthm}. In view of this,
we then give a mostly complete proof of Lemma \ref{lem:exitOfSolForOrdin}.

\par
	
\begin{lemma}\label{lem:exitOfSolForOrdin}
There exists a constant $T_2\in(0, T_1]$ such that \eqref{ordin} admits a unique solution
$(R,K)=(R,K)(\bt_{M+1};y,\xi)\in C^\infty(\Delta(T_2);(\SG^{1,0}(\R^{2n})\otimes\R^n)\times \SG^{0,0}(\R^{2n}))$,
$t_{j-1}\in[t_j,T_2]$, which satisfies, for any $\bt_{M+1}\in\Delta(T_2)$, $(y,\xi)\in\R^{2n}$,
\begin{equation}\label{estODE}
	\left\|\frac{\partial R}{\partial y}(\bt_{M+1};y,\xi)-I \right\|
	\leq C(t_{j-1}-t_{j+1}),
\end{equation} 
for a suitable constant $C>0$ independent of $M$, and 
\begin{align}\label{equivODE}
	\left\{\begin{array}{ll}
		&t_{j+1}\leq K(\bt_{M+1};y,\xi) \leq t_{j-1}
	\\[1ex]	
		&K_{|t_{j-1}=t_{j}}=t_{j+1}.
	\end{array}
	\right.
\end{align}
\end{lemma}
	
\par
	
\begin{proof}
First, we notice that, as a consequence of Lemma \ref{lem:invariantcomsymbol},
the compositions in the right-hand side of \eqref{ordin} are well defined, and
produce symbols of order $(1,0)$ and $(0,0)$, respectively, provided that 
$(R,K)\in C^\infty(\Delta(T_2);\SG^{1,0}(\R^{2n})\times \SG^{0,0}(\R^{2n}))$
and $K(\bt_{M+1})\in[t_{j+1},t_{j-1}]$ for any $\bt_{M+1}\in\Delta(T_2)$.
\par
We focus only on the variables $(t_{j-1},t_j,t_{j+1};y,\xi)$,
since all the others here play the role of (fixed) parameters, on which the solution clearly
depends smoothly. We then omit them in the next computations. We will also write, to shorten some of the formulae,
$(R,K)(s)=(R,K)(\bt_{{M+1},{j-1}}(s);y,\xi)$, $s\in[t_j,T_2]$, $T_2\in(0,T_1]$ sufficiently
small, to be determined below, $\bt_{M+1}\in\Delta(T_2)$, $(y,\xi)\in\R^{2n}$.
		
We rewrite \eqref{ordin} in integral form, namely
\begin{align}\label{equiordin}
\left\{\begin{array}{ll}
& R(s)=y-{\displaystyle \int_{t_j}^{s}}H(K(\s);\s,t_j,t_{j+1};R(\s),\xi)\,d\s
\\[1ex]
&K(s)=t_{j+1}+{\displaystyle\int_{t_j}^{s}}G(\s,K(\s),t_{j+1};R(\s),\xi)\,d\s,
\\[1ex]
\end{array}
\right.
\end{align}
 $s\in[t_j,T_2]$, $\bt_{M+1}\in\Delta(T_2)$, $(y,\xi)\in\R^{2n}$, and solve \eqref{equiordin} by the customary Picard method of successive approximations.
That is, we define the sequences 
\begin{align}\label{PicSequence}
\begin{cases}
&R_{l+1}(s)=y-{\displaystyle \int _{t_j}^{s}}H(K_l(\s);\s,t_j,t_{j+1};R_l(\s),\xi)\,d\s
\\[1ex]
&K_{l+1}(s)=t_{j+1}+{\displaystyle \int _{t_j}^{s}}G(\s,K_l(\s),t_{j+1};R_l(\s),\xi)\,d\s,
\end{cases}
\end{align}
for $l=1,2,\dots$,  $s\in[t_j,T_2]$, $\bt_{M+1}\in\Delta(T_2)$, $(y,\xi)\in\R^{2n}$, with
\begin{align*}
R_0(s)=y,\qquad K_0(s)=s-t_{j}+t_{j+1}.
\end{align*}
		
\par

We start by showing that $\{R_{l}\}_l$ and $\{K_l\}_l$
are bounded in $ C^\infty(\Delta(T_2);\SG^{1,0})$
and $ C^\infty(\Delta(T_2);\SG^{0,0})$, respectively, 
and that \eqref{estODE} and \eqref{equivODE} hold true for $R_l$ and $K_l$ in place of $R$ and $K$, respectively,
for any $l\in\Z_+$, uniformly with respect to $l,j,M$. This follows by induction. 

Namely, notice that all the stated properties are
true for $l=0$. Indeed, it is clear that 
$R_0\in C^\infty(\Delta(T_2);\SG^{1,0})$ and $K_0\in C^\infty(\Delta(T_2);\SG^{0,0})$, with seminorms
bounded by $\max\{1,2T_2\}$. \eqref{estODE} with $R_0$ in place of $R$
is trivial, while \eqref{equivODE} with $K_0$ in place of $K$ follows immediatly, by inserting $s=t_{j-1}$ in $K_0(s)$ and recalling that
$\bt_{M+1}\in\Delta(T_2)\Rightarrow t_{j-1}-t_{j}+t_{j+1}\in[t_{j+1},t_{j-1}]$. 

Assume now that \eqref{estODE} and \eqref{equivODE} hold true for $(R_\ell,K_\ell)$ for all the values of the index $\ell$ up to $l\ge0$.
We then find, by the same composition argument mentioned above, 
$R_{l+1}\in C^\infty(\Delta(T_2);\SG^{1,0})$ and $K_{l+1}\in C^\infty(\Delta(T_2);\SG^{0,0})$, with seminorms 
uniformly bounded with respect to $l$, since they depend only on the seminorms of $L$, $H$, $R_l$, $K_l$, 
and $T_2$. It follows that \eqref{estODE} holds true also for $R_{l+1}$ in place of $R$, since
\begin{align*}
	\left\|\frac{\partial R_{l+1}}{\partial y}(\bt_{M+1};y,\xi)-I \right\|
	&=
	\left\| 
		\int _{t_j}^{t_{j-1}}\frac{\partial}{\partial y}
		\left[H(K_l(\s);\s,t_j,t_{j+1};R_l(\s),\xi)\right]d\s 
	\right\|
	\\
	&\le C(t_{j-1}-t_{j}) \le C(t_{j-1}-t_{j+1}),
\end{align*}
for a suitable constant $C>0$ independent of $l$. 
Of course, $K_{l+1}(\bt_{M+1};y,\xi)|_{t_{j-1}=t_{j}}=t_{j+1}$ is clear, by the definition of $K_{l+1}$, $l\ge0$. 
It is also immediate that the hypotheses and the definition of $K_{l+1}$, $l\ge0$, imply that $K_{l+1}(\bt_{M+1};y,\xi)$
is, for any fixed $(y,\xi)\in\R^{2n}$, $0\le t_{M+1}\le \cdots t_{j+1}\le t_{j-1}\le \cdots t_0\le T_2\le T_1$,
a monotonically decreasing function with respect to $t_j\in[t_{j+1},t_{j-1}]$. \eqref{equivODE} with $K_{l+1}$ in place of $K$,
$l\ge0$, follows by such property and the hypotheses on $G$, observing that, for any $l\ge0$, $\bt_{M+1}\in\Delta(T_2)$,
$(y,\xi)\in\R^{2n}$, $s\in[t_{j+1},t_{j-1}]$,
\begin{equation}\label{eq:extrKl}
	K_{l+1}(\bt_{M+1,j-1}(s);y,\xi)|_{t_j=t_{j+1}}=s, \quad K_{l+1}(\bt_{M+1,j-1}(s);y,\xi)|_{t_j=s}=t_{j+1},
\end{equation}
which, in particular, also shows
\begin{equation}\label{eq:extrKlbis}
	K_{l+1}(\bt_{M+1};y,\xi)|_{t_j=t_{j+1}}=t_{j-1}, \quad K_{l+1}(\bt_{M+1};y,\xi)|_{t_j=t_{j-1}}=t_{j+1}.
\end{equation}
Notice that \eqref{eq:extrKlbis}, together with the monotonicity property of $K_{l+1}(\bt_{M+1};y,\xi)$
with respect to $t_j\in[t_{j+1},t_{j-1}]$, $l\ge0$, $\bt_{M+1}\in\Delta(T_2)$, $(y,\xi)\in\R^{2n}$, 
completes the proof of \eqref{equivODE} with $K_{l+1}$ in place of $K$ and the argument. Then, it just remains to prove \eqref{eq:extrKl}.

We again proceed by induction. \eqref{eq:extrKl} is manifestly true for $K_0$. Assume then that it holds
true for $K_\ell$ for all the values of the index $\ell$ up to $l\ge0$. We find, in view of the hypotheses on $G$ and the
inductive hypothesis,
\begin{align*}
	K_{l+1}(\bt_{M+1,j-1}(s);&y,\xi)|_{t_j=t_{j+1}}
	\\
	&=t_{j+1}+\int_{t_{j+1}}^s G(\sigma, K_l(\bt_{M+1,j-1}(\s),t_{j+1},t_{j+1};y,\xi),R_l(\sigma),\xi)\,d\sigma
	\\
	&=t_{j+1}+\int_{t_{j+1}}^s G(\sigma,\sigma,t_{j+1};R_l(\sigma),\xi)\,d\sigma=t_{j+1}+\int_{t_{j+1}}^s d\sigma=s,
	\\
	K_{l+1}(\bt_{M+1,j-1}(s);&y,\xi)|_{t_j=s}
	\\
	&=t_{j+1}+\int_{s}^s G(\sigma, K_l(\bt_{M+1,j-1}(\s),s,t_{j+1};y,\xi),R_l(\sigma),\xi)\,d\sigma=t_{j+1},
\end{align*}
which completes the proof of \eqref{eq:extrKl}.

\par

In order to show that $\{R_l\}$ and $\{K_l\}$ converge, we employ
Taylor formula with respect to the variable $t_{j-1}$.
For an arbitrary $N\in\Z_+$ we can write
\begin{align}\label{tayforKl}
K_{l+1}(t_{j-1})-K_l(t_{j-1})&=
  \sum\limits_{k<N}\frac{\left(\left(\partial_{t_{j-1}}^kK_{l+1}\right)(t_j)-
	\left(\partial_{t_{j-1}}^kK_{l}\right)(t_j)\right)(t_{j-1}-t_j)^k}{k!}
\\[1ex]
&+\frac{1}{N!}\int_{t_j}^{t_{j-1}}(t_{j-1}-\s)^N
 \left(\left(\partial_{t_{j-1}}^{N+1}K_{l+1}\right)(\s)-
	\left(\partial_{t_{j-1}}^{N+1}K_{l}\right)(\s)\right)d\s\nonumber
\end{align}
and 
\begin{align}\label{tayforRl}
R_{l+1}(t_{j-1})-R_l(t_{j-1})&=
  \sum\limits_{k<N}\frac{\left(\left(\partial_{t_{j-1}}^kR_{l+1}\right)(t_j)-
	\left(\partial_{t_{j-1}}^kR_{l}\right)(t_j)\right)(t_{j-1}-t_j)^k}{k!}
\\[1ex]
&+\frac{1}{N!}\int_{t_j}^{t_{j-1}}(t_{j-1}-\s)^N
 \left(\left(\partial_{t_{j-1}}^{N+1}R_{l+1}\right)(\s)-
  \left(\partial_{t_{j-1}}^{N+1}R_{l}\right)(\s)\right)d\s,\nonumber
\end{align}
respectively.
The summations in the above equalities \eqref{tayforKl}
and \eqref{tayforRl} are actually identically vanishing.
To prove this assertion, one proceeds by induction on $N$, employing a
standard argument involving the Fa\'a di Bruno formula for the derivatives of composed functions
(see \cite{Taniguchi02} and the references quoted therein for more details).
		
		\par
Now, using standard inequalities for
the remainder, together with the fact that
$\{R_l\}_l$ is bounded in $C^\infty(\Delta(T_2);\SG^{1,0}\otimes\R^n)$, while
$\{K_l\}_l$ is bounded in $C^\infty(\Delta(T_2);\SG^{0,0})$,
from \eqref{tayforKl}
and \eqref{tayforRl} we get, for any $\a,\b\in \Z_+$,
\begin{multline}\label{K_lEstim}
\sup_{(y,\xi)\in\R^{2n}}\left|\partial_y^\a \partial_\xi^\b
\left(K_{l+1}(\bt_{{M+1},{j-1}}(t_{j-1});y,\xi)-K_l(\bt_{{M+1},{j-1}}
(t_{j-1});y,\xi)\right)
\jb{y}^{|\a|}\jb{\xi}^{|\b|}\right|
\\[1ex]
\leq C_{\alpha\beta}\frac{\left(t_{j-1}-t_j\right)^{N+1}}{(N+1)!},
\end{multline}
with $C_{\alpha\beta}$ independent of $j,\,N$, and, similarly,
\begin{multline}\label{R_lEstim}
\sup_{(y,\xi)\in\R^{2n}}\left|\partial_y^\a \partial_\xi^\b
\left(R_{l+1}(\bt_{{M+1},{j-1}}(t_{j-1});y,\xi)-R_l(\bt_{{M+1},{j-1}}(t_{j-1});y,\xi)\right)
\jb{y}^{-1+|\a|}\jb{\xi}^{|\b|}\right|
\\[1ex]
\leq \widetilde{C}_{\alpha\beta}\frac{\left(t_{j-1}-t_j\right)^{N+1}}{(N+1)!}.
\end{multline}
where $\widetilde{C}_{\alpha\beta}$ is independent of $j,\,N$.
		
Writing $l$ in place of $N$ in the right-hand side of 
\eqref{K_lEstim} and \eqref{R_lEstim}, it easily follows that $(R_l,K_l)$ converges, for $l\to+\infty$, to a unique fixed point $(R,K)$, which satisfies the
stated symbol estimates.
Since, as we showed above, the properties \eqref{estODE} and \eqref{equivODE} hold true for $(R_l,K_l)$ in place of $(R,K)$, $l\ge0$, uniformly with respect to
$M,j,l$, they also hold true for 
the limit $(R,K)$. The proof is complete. 
\end{proof}
	
The next Corollary \ref{cor:qlcpsol} is a standard result in the theory of Cauchy problems for quasilinear PDEs of the form \eqref{quasi}.
Its proof is based on the hypotheses on $L$ and $H$, and the properties of the solution of \eqref{ordin}.

\begin{corollary}\label{cor:qlcpsol}
	Under the same hypotheses of Lemma \ref{lem:exitOfSolForOrdin}, denoting by $\bar{R}(\bt_{M+1};x,\xi)$ the solution of the equation
	\[
		R(\bt_{M+1};y,\xi)=x, \quad \bt_{M+1}\in\Delta(T_2),x,\xi\in\R^n,
	\]
	the function
	\[
		\Upsilon(\bt_{M+1}; x,\xi)=K(\bt_{M+1};\bar{R}(\bt_{M+1};x,\xi),\xi)
	\]
	solves the Cauchy problem \eqref{quasi} for $x,\xi\in\R^n$, $\bt_{M+1}\in\Delta(T_2)$, for a sufficiently small $T_2\in(0,T_1]$.
\end{corollary}
\begin{remark}\label{rem:invJR}
	Notice that \eqref{estODE} implies that, for all $y,\xi\in\R^n$, $\bt_{M+1}\in\Delta(T_2)$, $T_2\in(0,T_1]$ suitably small,
	the Jacobian matrix $\dfrac{\partial R}{\partial y}(\bt_{M+1};y,\xi)$ belongs to a suitably small open neighbourhood of the
	identity matrix, so it is invertible, with norm in an interval of the form $[1-\varepsilon,1+\varepsilon]$, for positive, arbitrarily 
	small $\varepsilon$. A standard argument in the SG symbol theory (see, e.g., \cite{CO, Coriasco:998.1, Cothesis}), 
	then shows $\bar{R}\in C^\infty(\Delta(T_2),\SG^{1,0}(\R^{2n})\otimes\R^n)$ and 
	\[
		\jb{{R}(\bt_{M+1};y,\xi)}\sim\jb{y}, \quad \jb{\bar{R}(\bt_{M+1};x,\xi)}\sim\x,
	\]
	with constants independent of $\bt_{M+1}\in\Delta(T_2)$, $\xi\in\R^n$. This also implies 
	that $\Upsilon$ satisfies $t_{j+1}\le \Upsilon(\bt_{M+1}; x,\xi)\le t_{j-1}$, $\bt_{M+1}\in\Delta(T_2)$, $x,\xi\in\R^n$ 
	and $\Upsilon\in C^\infty(\Delta(T_2); \SG^{0,0}(\R^{2n}))$.
\end{remark}

\subsection{Commutative law for multi-products of SG phase functions given by solutions of eikonal equations}\label{subsec:commlaw}
Let $\{a_j\}_{j\in\N}$  be a bounded family of parameter-dependent, real-valued symbols in
$C^\infty([0,T];\SG^{1,1})$ and let $\{\varphi_j\}_{j\in\N}$ be the corresponding
family of phase functions in $\Phr(c|t-s|)$, obtained as solutions to the eikonal equations associated
with $a_j$, $j\in\N$.
In the aforementioned multi-product \eqref{sharpprodofphi},
we commute $\varphi_j$ and $\varphi_{j+1}$, defining a new multi-product $\phi_j$,
namely
\begin{equation}\label{comsharpmultprod}
 \begin{aligned}
 \phi_j(\bt_{M+1};x,\xi)=\left(\varphi_1(t_0,t_1)\right.&\!\!\left.\sharp\varphi_2(t_1,t_2)
  \sharp\dots\sharp\varphi_{j-1}(t_{j-2},t_{j-1})\sharp\right.
\\[1ex]
&\!\!\left.\sharp\varphi_{j+1}(t_{j-1},t_j)\sharp\varphi_j(t_j,t_{j+1})\sharp\right.
\\
&\!\!\left.
  \sharp\varphi_{j+2}(t_{j+1},t_{j+2})\sharp\dots\sharp
    \varphi_{M+1}(t_M,t_{M+1})\right)(x,\xi),
    \end{aligned}
\end{equation}
where $\bt_{M+1}=(t_0,t_1,\dots,\dots t_{M+1})\in \Delta(T_1)$.

\par

\begin{assumption}[Involutiveness of symbol families]\label{assumpt:invo}
Given the family of parameter-dependent, real-valued symbols $\{a_j\}_{j\in\N}\subset C^\infty([0,T];\SG^{1,1}(\R^{2n})$,
there exist families of parameter-dependent, real-valued symbols $\{b_{jk}\}_{j,k\in\N}$ and $\{d_{jk}\}_{j,k\in\N}$, such that
$b_{jk},d_{jk}\in C^\infty([0,T];\SG^{0,0}(\R^{2n}))$, $j,k\in\N$, and the Poisson brackets
\begin{align*}
\{\tau-a_j(t;x,\xi),\tau-a_{k}(t;x,\xi)\}
&:=\partial_ta_j(t;x,\xi)-\partial_ta_{k}(t;x,\xi)
\\
&\phantom{:}+ a'_{j,\xi}(t;x,\xi)
\cdot a'_{k,x}(t;x,\xi)-a'_{j,x}(t;x,\xi)\cdot  a'_{k,\xi}(t;x,\xi)
\end{align*}
satisfy
\begin{equation}\label{pbrack}
\{\tau-a_j(t;x,\xi),\tau-a_{k}(t;x,\xi)\}
=
b_{jk}(t;x,\xi)\cdot(a_j-a_{k})(t;x,\xi)+d_{jk}(t;x,\xi),
\end{equation}
for all $j,k\in\N$, $t\in[0,T]$, $x,\xi\in\R^n$.
\end{assumption}
We can now state our first main theorem.

\par

\begin{theorem}\label{cruthm}
Let $\{a_j\}_{j\in\N}$ be a family of parameter-dependent, real-valued symbols, bounded in
$C^\infty\left([0,T];\SG^{1,1}(\R^{2n})\right)$, and let $\varphi_j\in\Phr(c|t-s|)$, for some $c>0$, be the phase functions
obtained as solutions to \eqref{eik} with $a_j$ in place of $a$, $j\in\N$.
Consider $\phi(\bt_{M+1})$ and $\phi_j(\bt_{M+1})$ defined
in \eqref{sharpprodofphi} and \eqref{comsharpmultprod}, respectively,
for any $M\ge2$ and $j\leq M$.
Then, Assumption \ref{assumpt:invo} implies that there exists $T^\prime\in(0,T_2]$,
independent of $M$, such that we can find a symbol family
$Z_j\in C^\infty(\Delta(T');\SG^{0,0}(\R^{2n}))$ satisfying, for all $\bt_{M+1}\in\Delta(T^\prime)$, $x,\xi\in\R^n$,
\begin{equation}\label{eq:Zprop1}
\begin{aligned}
&t_{j+1}\leq Z_j(\bt_{M+1};x,\xi)\leq t_{j-1},
\\[1ex]
&Z_j|_{{t_j=t_{j-1}}}=t_{j+1},\text{ and } Z_j|_{{t_j=t_{j+1}}}=t_{j-1},
\\[1ex]
&\partial_{t_j}Z_j(\bt_{M+1};x,\xi)\asymp -1.
\end{aligned}
\end{equation}	
Moreover, we have
\begin{equation*}
\phi_j(\bt_{M+1};x,\xi)=\phi(\bt_{M+1,j}(Z_j(\bt_{M+1};x,\xi));x,\xi)
+\Psi_j(\bt_{M+1};x,\xi), \bt_{M+1}\in\Delta(T^\prime),x,\xi\in\R^n,
\end{equation*}
where $\Psi_j \in C^\infty(\Delta(T');\SG^{0,0}(\R^{2n}))$ satisfies
\begin{equation*}
\Psi_j\equiv 0\quad\text{ if }\quad d_j\equiv 0.
\end{equation*}
\end{theorem}
	
\par

The argument originally given in \cite{Taniguchi02} extends to the SG setting, in view of Lemma \ref{lem:exitOfSolForOrdin} above, and allows
to prove Theorem \ref{cruthm}. For the sake of completeness, we include  
it here (see \cite{Abdeljawadthesis} for additional details).

\par
	
\begin{proof}[Proof of Theorem \ref{cruthm}]
Let $\{(Y_1,\dots,Y_M,N_1,\dots,N_M)\}(\bt_{M+1};x,\xi)$
be the solution of the critical point system 
\begin{align*}\begin{cases}
x_j&=\varphi'_{j,\xi}(t_{j-1},t_j;x_{j-1},\xi_j)
\\[1ex]
\xi_j&=\varphi'_{j+1,x}(t_j,t_{j+1};x_j,\xi_{j+1}),
\end{cases}
\end{align*}
such that $x_j,\xi_j\in\R^n$,  $x_0=x$, and $\xi_{M+1}=\xi$
(cf. \cite{Abdeljawadthesis,AsCor,Kumano-go:1}).
\par
In view of \eqref{comsharpmultprod},
let $\left(\tilde{Y}_1,
\dots,\tilde{Y}_M,\tilde{N}_1,\dots,\tilde{N}_M\right)(\bt_{M+1};x,\xi)$
be the solution to the critical points problem,
for the phase functions in modified order, namely,
\begin{align*}
\begin{cases}
x_k&=\varphi'_{k,\xi}(t_{k-1},t_k;x_{k-1},\xi_k)\qquad\text{ if }k\in\{1,\dots,j-1,j+2,\dots,M\}
\\[1ex]
x_j&=\varphi'_{j+1,\xi}(t_{j-1},t_j;x_j,\xi_j),
\\[1ex]
x_{j+1}&=\varphi'_{j,\xi}(t_j,t_{j+1};x_j,\xi_{j+1}),
\end{cases}
\end{align*}
\par
\begin{align*}
\begin{cases}
\xi_k&=\varphi'_{k+1,x}(t_k,t_{k+1};x_k,\xi_{k+1})\qquad\text{ if }k\in\{1,\dots,j-2,j+1,\dots,M\}
\\[1ex]
\xi_{j-1}&=\varphi'_{j+1,x}(t_{j-1},t_j;x_{j-1},\xi_j),
\\[1ex]
\xi_j&=\varphi'_{j,x}(t_j,t_{j+1};x_j,\xi_{j+1}),
\end{cases}
\end{align*}
where $x_0=x$ and $\xi_{M+1}=\xi$. For convenience below, we also set
\begin{align*}
\begin{cases}
a_0(t;x,\xi)&\equiv 0,
\\[1ex]
Y_0=\tilde{Y}_0&=x,
\\[1ex]
N_0=\phi'_x,\;&\tilde{N}_0=\phi'_{j,x}.
\end{cases}
\end{align*}
\par
Let $\Psi_j$ be defined as 
\begin{equation}\label{Psidef}
\Psi_j(\bt_{M+1};x,\xi)=\phi_j(\bt_{M+1};x,\xi)
  -\phi(\bt_{M+1,j}(Z_j(\bt_{M+1};x,\xi));x,\xi).
\end{equation}
Here we look for a symbol $Z_j = Z_j(\bt_{M+1};x,\xi)$ satisfying \eqref{eq:Zprop1}
such that $\Psi_j\in C^\infty(\Delta(T^\prime);\SG^{0,0})$, $T^\prime\in(0,T_2]$.
In view of Proposition \ref{prop:tderivofmultiph} and \eqref{comsharpmultprod}, we find
\begin{equation}\label{PsiTDer}
	\begin{aligned}
		\partial_{t_{j-1}}\Psi_j(\bt_{M+1};x,\xi)&=(\partial_{t_{j-1}}\phi_j)(\bt_{M+1};x,\xi)
		-(\partial_{t_{j-1}}\phi)(\bt_{M+1,j}(Z_j(\bt_{M+1};x,\xi));x,\xi)
		\\[1ex]
		&-(\partial_{t_j}\phi)(\bt_{M+1,j}(Z_j(\bt_{M+1};x,\xi));x,\xi)\cdot\partial_{t_{j-1}}Z_j(\bt_{M+1};x,\xi)
		\\[1ex]
		&=a_{j+1}(t_{j-1};Y_{j-1}(\bt_{M+1};x,\xi),N_{j-1}(\bt_{M+1};x,\xi))
		\\[1ex]
		&-a_{j-1}(t_{j-1};Y_{j-1}(\bt_{M+1};x,\xi),N_{j-1}(\bt_{M+1};x,\xi))
		\\[1ex]
		&-\left[a_j(t_{j-1};Y_{j-1}(\bt_{M+1};x,\xi),N_{j-1}(\bt_{M+1};x,\xi))\right]_{t_{j}=Z_j(\bt_{M+1};x,\xi)}
		\\[1ex]
		&+\left[a_{j-1}(t_{j-1};Y_{j-1}(\bt_{M+1};x,\xi),N_{j-1}(\bt_{M+1};x,\xi))\right]_{t_{j}=Z_j(\bt_{M+1};x,\xi)}
		\\[1ex]		
		&-(\partial_{t_j}\phi)(\bt_{M+1,j}(Z_j(\bt_{M+1};x,\xi));x,\xi)\cdot\partial_{t_{j-1}}Z_j(\bt_{M+1};x,\xi).
		\end{aligned}
		\end{equation}
		\par
		When $j\geq 2$, we use the trajectory $\left(\tilde{q}_{j-1},\tilde{p}_{j-1}\right)(\bt_{j-2},\s;y,\eta)$ defined in  \eqref{trajectory}.
		Then (cf. Proposition \ref{prop:trajec}), we have, for $\s=t_{j-1}$, the equalities
		\begin{align*}
		\begin{cases}
		(\tilde{q}_{j-1},\tilde{p}_{j-1})(\bt_{j-1};x,\phi'_x(\bt_{M+1};x,\xi))&=(Y_{j-1},N_{j-1})(\bt_{M+1};x,\xi),
		\\[1ex]
		(\tilde{q}_{j-1},\tilde{p}_{j-1})(\bt_{j-1};x,\phi'_{j,x}(\bt_{M+1};x,\xi))&=(\tilde{Y}_{j-1},\tilde{N}_{j-1})(\bt_{M+1};x,\xi).
		\end{cases}
		\end{align*}
		\par
		Next, we set
		\begin{align}\label{eq:alphaTrajectory}
		\left\{\begin{array}{rl}
		\a_1(\s;z,\zeta)&=a_2(\s;z,\zeta),
		\\[1ex]
		\a_j(\s;\bt_{j-2};z,\zeta)&=\left(a_{j+1}-a_{j-1}\right)
		\big(\s;\left(\tilde{q}_{j-1},\tilde{p}_{j-1}\right)
		(\bt_{j-2},\s;z,\zeta)\big),\; j\geq 2.
		\end{array}
		\right.
		\end{align}
		In \eqref{eq:alphaTrajectory} the compositions are well defined,
		in view of the properties of the symbols  $(\tilde{q}_{j-1},\tilde{p}_{j-1})$
		in Proposition \ref{trajclass}, which imply that the conditions
		of Lemma \ref{lem:invariantcomsymbol} are satisfied.
		Thus, $\a_j\in C^\infty(\Delta_{j-1}(T_1);\SG^{1,1})$, $j=1,\dots,M$. 
		Moreover, $\a_j$ satisfies
		\begin{equation}\label{eq:alphaDef}
		\begin{cases}
		\a_j(\bt_{j-2},\s;x,\phi'_x(\bt_{M+1};x,\xi))\!\!\!\!\!&=(a_{j+1}-a_{j-1})\left(\s;Y_{j-1}(\bt_{M+1};x,\xi),N_{j-1}(\bt_{M+1};x,\xi)\right),
		\\[1ex]
		\a_j(\bt_{j-2},\s;x,\phi'_{j,x}(\bt_{M+1};x,\xi))\!\!\!\!\!&=(a_{j+1}-a_{j-1})(\s;\tilde{Y}_{j-1}(\bt_{M+1};x,\xi),\tilde{N}_{j-1}(\bt_{M+1};x,\xi)),
		\end{cases}
		\end{equation}
		and when $j=1$ the variables $(\bt_{j-2},\s)$  reduce to $\s$.
		\par
		Finally, let $[T_j(\tau)](\bt_{M+1};x,\xi)= T_j(\tau,\bt_{M+1};x,\xi)$ be defined as
\[
T_j(\tau)=\int_{0}^{1}\a'_{j,\xi}\bigg(\bt_{j-1};x,\rho\phi'_{j,x}
(\bt_{M+1};x,\xi)+(1-\rho) \phi'_x(\bt_{M+1,j}
(\tau);x,\xi)\bigg)\,d\rho.
\]
		Notice that, again by Lemma \ref{lem:invariantcomsymbol} and the properties of the involved symbols,
		we find $T_j\in C^\infty([t_{j+1},t_{j-1}]\times\Delta(T_1);\SG^{1,0}\otimes\R^n)$. Indeed, in view of the fact that both $\phi$ and $\phi_j$
		are regular SG phase functions, for all $\rho\in[0,1]$, $\tau\in[t_{j+1},t_{j-1}]$,
		\[
			\jb{ \rho\phi'_{j,x}
			(\bt_{M+1};x,\xi)+(1-\rho) \phi'_x(\bt_{M+1,j}
			(\tau);x,\xi)}\sim\csi,
		\]
		uniformly with respect to all the involved parameters.
		\par
		We now show that $\phi_j$ satisfies a certain PDE, whose form we will simplify using the results in Subsection \ref{subs:auxpde}.
		First of all, we observe that 
		\begin{equation}\label{eq:alphaDifference}
		\begin{aligned}
		\a_j(\mathbf t_{j-1};&x,\phi'_{j,x}(\mathbf t_{M+1};x,\xi))-	
			\a_j(\mathbf t_{j-1};x,\phi'_x(\mathbf t_{M+1,j}(\tau);x,\xi))
		\\[1ex]
		&=T_j(\tau)\cdot\bigg(\phi'_{j,x}(\mathbf t_{M+1};x,\xi)-
		  \phi'_x(\mathbf t_{M+1,j}(\tau);x,\xi)\bigg).
		\end{aligned}
		\end{equation}
		From \eqref{Psidef} it follows
		\begin{equation}\label{PsiXDer}
		\Psi'_{j,x}(\bt_{M+1};x,\xi)=\phi'_{j,x}(\mathbf t_{M+1};x,\xi)- 
		 (\partial_{t_j}\phi)(\mathbf t_{M+1,j}(Z_j);x,\xi)\cdot Z'_{j,x}(\bt_{M+1};x,\xi).
		\end{equation}
		\par
		Now, we rewrite $\partial_{t_{j-1}}\Psi_j$ from \eqref{PsiTDer},
		using \eqref{eq:alphaDef}, \eqref{eq:alphaDifference} and \eqref{PsiXDer}:
		\begin{equation}\label{newPsiTDer}
		\begin{aligned}
		\partial_{t_{j-1}}\Psi_j(\bt_{M+1};x,\xi)&=\a_j\left(\mathbf t_{j-1};x,\phi'_{j,x}(\mathbf t_{M+1};x,\xi)\right)
		\\[1ex]
		&-\a_j\left(\mathbf t_{j-1};x,\phi'_x(\mathbf t_{M+1,j}(Z_j(\bt_{M+1};x,\xi));x,\xi)\right)
		\\[1ex]
		&-(\partial_{t_{j}}\phi)(\mathbf t_{M+1,j}(Z_j(\bt_{M+1};x,\xi));x,\xi)\cdot(\partial_{t_{j-1}}Z_j)(\bt_{M+1};x,\xi)
		\\[1ex]
		&-\left[\left(a_j-a_{j+1}\right)(t_{j-1};Y_{j-1}(\bt_{M+1};x,\xi),N_{j-1}(\bt_{M+1};x,\xi))\right]_{t_j=Z_j(\bt_{M+1};x,\xi)}
		\\[1ex]
		&=T_j(Z_j(\bt_{M+1};x,\xi))\cdot\Psi'_{j,x}(\bt_{M+1};x,\xi)
		\\[1ex]
		&-(\partial_{t_j}\phi)(\mathbf t_{M+1,j}(Z_j(\bt_{M+1};x,\xi));x,\xi)\cdot
		\\[1ex]
		&\cdot\left(\partial_{t_{j-1}}Z_j(\bt_{M+1};x,\xi)-T_j(Z_j(\bt_{M+1};x,\xi))\cdot Z'_{j,x}(\bt_{M+1};x,\xi))\right) 
		\\[1ex]
		&- \left[\left(a_j-a_{j+1}\right)(t_{j-1};Y_{j-1}(\bt_{M+1};x,\xi),N_{j-1}(\bt_{M+1};x,\xi))\right]_{t_j=Z_j(\bt_{M+1};x,\xi)}.
		\end{aligned}
		\end{equation}
		\par
		Once more, we use the solution $\left(q_j,p_j\right)(t,s;y,\eta)$ of the Hamilton-Jacobi system \eqref{hameq}, 
		with $a$ replaced by $a_j$, and define $\tilde{\a}_j$ as
		\[
		\tilde{\a}_j(\s,t;y,\eta)=\left(a_j-a_{j+1}\right)(\s;\left(q_j,p_j\right)(\s,t;y,\eta)).
		\]
		Then, after a differentiation with respect to $\s$, we have
		\begin{align*}
		\partial_\s\tilde{\a}_j&(\s,t;y,\eta)=
		\\[1ex]
		&=\partial_\s\left(\left(a_j-a_{j+1}\right)
		(\s;\left(q_j,p_j\right)(\s,t;y,\eta)\right)
		\\[1ex]
		&=(\partial_ta_j)\left(\s;\left(q_j,p_j\right)(\s,t;y,\eta)\right)-
		(\partial_ta_{j+1})(\s;\left(q_j,p_j\right)(\s,t;y,\eta))
		\\[1ex]
		&+\bigg	(a'_{j,x}(\s;\left(q_j,p_j\right)(\s,t;y,\eta))-a'_{j+1,x}
		(\s;\left(q_j,p_j\right)(\s,t;y,\eta))\bigg)\cdot\partial_tq_j(\s,t;y,\eta)
		\\[1ex]
		&+\bigg( a'_{j,\xi}(\s;\left(q_j,p_j\right)(\s,t;y,\eta))- a'_{j+1,\xi}
		(\s;\left(q_j,p_j\right)(\s,t;y,\eta))\bigg)\cdot\partial_tp_j(\s,t;y,\eta),
		\end{align*}
		and we use \eqref{hameq} to write
		\begin{align*}
		\partial_\s\tilde{\a}_j(\s,t;y,\eta)&=\left[\partial_\s(a_j-a_{j+1})\right]
		(\s;\left(q_j,p_j\right)(\s,t;y,\eta))
		\\[1ex]
		&+a'_{j,\xi}(\s;\left(q_j,p_j\right)(\s,t;y,\eta))\cdot a'_{j+1,x}
		(\s;\left(q_j,p_j\right)(\s,t;y,\eta))
		\\[1ex]
		&- a'_{j,x}(\s;\left(q_j,p_j\right)(\s,t;y,\eta))\cdot a'_{j+1,\xi}
		(\s;\left(q_j,p_j\right)(\s,t;y,\eta))
		\\[1ex]
		&=\{\tau-a_j,\tau-a_{j+1}\}(\s;\left(q_j,p_j\right)(\s,t;y,\eta)).
		\end{align*}
		Assumption \ref{assumpt:invo} then implies
		\begin{equation}\label{tildeAlphaEqu}
			\partial_\s\tilde{\a}_j(\s,t;y,\eta)=b_{j,j+1}(\s;\left(q_j,p_j\right)(\s,t;y,\eta))\cdot\tilde{\a}_j(\s,t;y,\eta)+d_{j,j+1}(\s;\left(q_j,p_j\right)(\s,t;y,\eta)).
		\end{equation}
		Solving \eqref{tildeAlphaEqu} as a first order linear ODE in $\s$ with unknown $\tilde{\a}_j(\s,t;y,\eta)$, and writing $b_j$ in place of $b_{j,j+1}$,
		$d_j$ in place of $d_{j,j+1}$, respectively, we see that
		\begin{align*}
		\tilde{\a}_j(\s,t;y,\eta)&=\exp\left(\int_{t_j}^{\s} b_j
		(\tau;\left(q_j,p_j\right)(\tau,t;y,\eta))d\tau\right)
		\cdot\left[\tilde{\a}_j(t_{j},t;y,\eta)+
		\phantom{
		\int_{t_j}^\s
		d_j\left(\nu;\left(q_j,p_j\right)(\nu,t;y,\eta)\right)
		\cdot \exp\left(\int_{\nu}^{\s}b_j\left(\varsigma;\left(q_j,p_j\right)
		(\varsigma,t;y,\eta)\right)d\varsigma\right)d\nu
		}
		\right.
		\\[1ex]
		&\left.+{\displaystyle{\int_{t_j}^\s}}d_j\left(\nu;\left(q_j,p_j\right)(\nu,t;y,\eta)\right)
		\cdot \exp\left(-\int_{\nu}^{\s}b_j\left(\varsigma;\left(q_j,p_j\right)
		(\varsigma,t;y,\eta)\right)d\varsigma\right)d\nu.\right]
		\end{align*}
		Once again, notice that all the composition performed so far are well defined, and produce SG symbols, in view of Lemma \ref{lem:invariantcomsymbol},
		\eqref{eq:equiv}, and recalling that $h\in\SG^{0,0}\Rightarrow \exp(h)\in\SG^{0,0}$ (see, e.g., \cite{Coriasco:998.1,Cothesis}).
		\par
		As stated in Proposition \ref{prop:trajec}, we can write $\tilde{\a}_j$ in terms of the solution to the critical points problem \eqref{critpntprob}.
		Indeed, by \eqref{trajequal1}, \eqref{trajequal2} we get
		\begin{align*}
		\tilde{\a}_j(t_j,t_{j-1};&\left(q_j,p_j\right)(t_j,t_{j-1};Y_{j-1}(\bt_{M+1};x,\xi),N_{j-1}(\bt_{M+1};x,\xi)))=
		\\[1ex]
		&=\tilde{\a}_j(t_j,t_{j-1};Y_{j}(\bt_{M+1};x,\xi),N_{j}(\bt_{M+1};x,\xi))	
		\\[1ex]
		&=\left(a_j-a_{j+1}\right)(t_j;Y_j(\bt_{M+1};x,\xi),N_j(\bt_{M+1};x,\xi)).
		\end{align*}
		Moreover, using Proposition \ref{prop:tderivofmultiph}, we obtain the equality
		\begin{align}\label{aphaTildaPartialPhi}
		\tilde{\a}_j(t_j,t_{j-1};\left(q_j,p_j\right)(t_j,t_{j-1};Y_{j-1}(\bt_{M+1};x,\xi),N_{j-1}(\bt_{M+1};x,\xi)))=
			-\partial_{t_{j}}\phi(\bt_{M+1};x,\xi).
		\end{align}
		Define also
		\begin{align*}\label{GDef}
		G_j(t_j)&\equiv G_j(\bt_{M+1};x,\xi)=
		\\[1ex]		
		&=\exp\left[\int_{t_j}^{t_{j-1}} b_j(\tau;\left(q_j,p_j\right)
			(\tau,t_{j-1};Y_{j-1}(\bt_{M+1};x,\xi),N_{j-1}(\bt_{M+1};x,\xi)))d\tau\right]
		\end{align*}
		and
		\begin{align*}
		F_j(t_j)&\equiv F_j(\bt_{M+1};x,\xi)=
		\\[1ex]
		&=\left[\int_{t_j}^{t_{j-1}}d_j\left(\nu;\left(q_j,p_j\right)
		(\nu,t_{j-1};y,\eta)\right) \cdot \right.
		\\[1ex]
		&\phantom{=(\int_{t_j}^{t_{j-1}}}\cdot\left.\exp\left(-\int_{\nu}^{t_{j-1}}b_j
		\left(\varsigma;\left(q_j,p_j\right)(\varsigma,t_{j-1};y,\eta)\right)
		d\varsigma\right)d\nu\right]_{(y,\eta)=(Y_{j-1},N_{j-1})(\bt_{M+1};x,\xi)},
		\end{align*}
		where both $G_j$ and $F_j$, as a consequence of Lemma \ref{lem:invariantcomsymbol} and the properties of
		$b_j$, $d_j$, $(q_j,p_j)$ and $(Y_{j-1},N_{j-1})$, are symbols belonging to $C^\infty(\Delta(T_1);\SG^{0,0})$.
		\par
		Then, using the formulae \eqref{newPsiTDer} and \eqref{aphaTildaPartialPhi} above, we find that $\Psi_j$ must fullfill
		\begin{equation}\label{psiComp}
		\partial_{t_{j-1}}\Psi_j=T_j(Z_j)\cdot \Psi'_{j,x}-F_j(Z_j)
		  -(\partial_{t_j}\phi)(Z_j)\bigg(\partial_{t_{j-1}}Z_j-T_j(Z_j)
		     \cdot Z'_{j,x}-G_j(Z_j)\bigg),
		\end{equation}
		where we omitted everywhere the dependence on $(\bt_{M+1};x,\xi)$,
		$(\partial_{t_j}\phi)(Z_j)$ stands for $(\partial_{t_j}\phi)(\bt_{M+1,j}(Z_j(\bt_{M+1};x,\xi));x,\xi)$,
		and
		\[
			F_j(Z_j)=F_j(\bt_{M+1,j}(Z_j(\bt_{M+1};x,\xi));x,\xi), G_j(Z_j)=G_j(\bt_{M+1,j}(Z_j(\bt_{M+1};x,\xi));x,\xi).
		\]
		\par
		Now, in order to simplify \eqref{psiComp}, we choose $Z_j$ as solution to the quasilinear Cauchy problem,
		\begin{align}\label{quasiZj}
		\begin{cases}
		\partial_{t_{j-1}}Z_j&=T_j(Z_j)\cdot Z'_{j,x}+G_j(Z_j)
		\\[1ex]
		Z|_{t_j={t_{j-1}}}&=t_{j+1}.
		\end{cases}
		\end{align}
		\par
		It is easy to see that \eqref{quasiZj} is a quasilinear Cauhcy problem of the type considered in
		Subsection \ref{subs:auxpde}. In view of Lemma \ref{lem:exitOfSolForOrdin},
		we can solve \eqref{quasiZj} through its characteristic system \eqref{ordin}, with $T_j$ in place of $L$ and $G_j$ in place of $H$,
		choosing a sufficiently small parameter interval $[0,T^\prime]$, $T^\prime\in(0,T_2]$. Indeed, 
		by Corollary \ref{cor:qlcpsol} and Remark \ref{rem:invJR} (cfr. \cite{Taniguchi02}), defining 
		\[
			Z_j(\bt_{M+1};x,\xi)=K(\bt_{M+1};\bar{R}(\bt_{M+1};x,\xi),\xi), \quad \bt_{M+1}\in\Delta(T^\prime), x,\xi\in\R^n,
		\]
		gives a solution of \eqref{quasiZj} with all the desired properties. 
		
		With such choice of $Z_j$, \eqref{psiComp} is then reduced to the linear, non-homohgeneous PDE 
		\begin{align}\label{newPsiComp}
			\partial_{t_{j-1}}\Psi_j=T_j(Z_j)\cdot\Psi'_{j,x}-F_j(Z_j),
		\end{align}
		with the initial condition
		\begin{align}\label{inConNewPsi}
			{\Psi_j}|_{t_{j-1}=t_j}=0.
		\end{align}
		Notice that \eqref{inConNewPsi} holds true since we have
		$Z_j|_{{t_j=t_{j-1}}}=t_{j+1}$, that, together with \eqref{ii.propofsharprod}
		in Proposition \ref{prop:tderivofmultiph}, gives
		\begin{multline*}
		\Psi_j(\bt_{M+1,j-1}(t_j);x,\xi)=
		\\[1ex]
		\phi_j(\bt_{M+1,j-1}(t_j);x,\xi)
		-\phi(t_0,\dots,t_{j-2},t_j,{Z_j(\bt_{M+1};x,\xi)}|_{t_j=t_{j-1}},t_{j+1},\dots,t_{M+1};x,\xi)
		\\[1ex]
		=\left[\varphi_1(t_0,t_1)\sharp\dots\sharp\varphi_{j-1}(t_{j-2},t_j)
		\sharp\right.\underbrace{(\varphi_{j+1}(t_j,t_j)\sharp\varphi_j(t_j,t_{j+1}))}_{=\varphi_j(t_j,t_{j+1})}\sharp
		\\[1ex]
		\left.\sharp\varphi_{j+2}(t_{j+1},t_{j+2})\sharp\dots\varphi_{M+1}(t_M,t_{M+1})\right](x,\xi)
		\\[1ex]
		-\left[\varphi_1(t_0,t_1)\sharp\dots\sharp\varphi_{j-1}(t_{j-2},t_{j})
		\sharp\right.\underbrace{(\varphi_j(t_j,t_{j+1})\sharp\varphi_{j+1}(t_{j+1},t_{j+1}))}_{=\varphi_j(t_j,t_{j+1})}\sharp
		\\[1ex]
		\left.\sharp\varphi_{j+2}(t_{j+1},t_{j+2})\sharp\dots\sharp\varphi_{M+1}(t_M,t_{M+1})\right](x,\xi)
		=0.
		\end{multline*}
		\par
		Then, the method of characteristics, applied to the linear, non-homogeneous PDE \eqref{newPsiComp}, shows that we can write $\Psi_j$ in the form
		\begin{equation}\label{finalFormOfPsi_j}
			\Psi_j(\bt_{M+1};x,\xi)=\int_{t_j}^{t_{j-1}}\widetilde{F}_j(\bt_{M+1,j-1}(\tau);
		\theta(\tau;\tilde{\theta}(\tau;x,\xi),\xi),\xi)d\tau,
		\end{equation}
		where 
		\begin{align*}
			\widetilde{F}_j(\bt_{M+1};x,\xi)&=-F_j(\bt_{M+1,j}(Z_j(\bt_{M+1};x,\xi));x,\xi),
			\\[1ex]
			\theta(\tau;y,\xi)&=\phantom{-}\theta(\bt_{M+1,j}(\tau);y,\xi),
			\\[1ex]
			\tilde{\theta}(\tau;x,\xi)&=\phantom{-}\tilde{\theta}(\bt_{M+1,j-1}(\tau);x,\xi),
		\end{align*}
		for suitable vector-valued functions $\theta,\tilde{\theta}$.
		By arguments similar to those in Subsection \ref{subs:auxpde} (cf. \cite{Taniguchi02}), both $\theta$ and $\tilde{\theta}$
		turn out to be elements of $C^\infty(\Delta(T^\prime);\SG^{1,0}\otimes\R^n)$, satisfying
		\[
			\jb{\theta(\bt_{M+1,j}(\tau);y,\xi)}\sim\jb{y},\quad
			\jb{\tilde{\theta}(\bt_{M+1,j-1}(\tau);x,\xi)}\sim\x,
		\]
		with constants independent of $\bt_{M+1}\in\Delta(T^\prime)$, $x,\xi\in\R^n$.
		Such result, together with the properties of $Z_j$ and another application of Lemma \ref{lem:invariantcomsymbol}, allows to conclude that
		$\Psi_j\in C^\infty(\Delta(T');\SG^{0,0})$, and it is identically zero when $d_j\equiv0$, as claimed.
		The proof is complete.
\end{proof}
	\par

\begin{remark}\label{rem:commiterint}
	As we will see in the next Section \ref{sec:fundsol}, the fundamental solution of a system of the form \eqref{sysint}
	involves iterated integrals (on subintervals of $[0,T^\prime]$) of multi-products
	of SG FIOs with matrix-valued symbols of order $(0,0)$, and regular,
	parameter-dependent SG phase functions $\varphi_{N_j}$, $N_j=1,\dots,N$, $j=1,\dots, \nu$, $\nu\in\N$, obtained as solutions to eikonal equations,
	associated with an involutive family of Hamiltonians $\{\lambda_j\}_{j=1}^N$. 
	In view of \cite{AsCor}, such multi-product is a SG FIO, with symbol of order $(0,0)$ and regular phase fuction
	given by the multi-product $\phi=\varphi_{N_1}\#\cdots\#\varphi_{N_\nu}$.
	Theorem \ref{cruthm}, together with some standard manipulation of multiple integrals and change of time-parameters of the form
	$t_j=Z_j(\bt_{\nu+1,j}(\widetilde{t}_j);x,\xi)$, implies that it is possible to change
	the order of two (or more) factors in the multi-product $\varphi_{N_1}\#\cdots\#\varphi_{N_\nu}$ within the iterated integral.
	The result is still an iterated integral, on the same subintervals, of a SG FIOs with regular phase function $\widetilde{\phi}$,
	given by a differently sorted multi-product of the same phase functions $\varphi_{N_j}$, $N_j=1,\dots,N$, $j=1,\dots, \nu$, $\nu\in\N$,
	and symbol of order $(0,0)$. Indeed, the symbol is altered only by the change of the $t_j$ parameter
	and (products of) factors of the form $\partial_{t_j}Z_j\cdot\exp(-i\Psi_j)$,
	which are of order $(0,0)$ and elliptic, by Theorem \ref{cruthm} and Lemma \ref{lem:invariantcomsymbol},
	again taking into account that $h\in\SG^{0,0}\Rightarrow \exp(h)\in\SG^{0,0}$ (see \cite{Abdeljawadthesis, Coriasco:998.1, Morimoto, Taniguchi02}
	for details). In particular, if the original symbol is elliptic, the same property holds true for the transformed symbol.
\end{remark}

	\section{Cauchy problems for weakly SG-hyperbolic systems}\label{sec:fundsol}
	\setcounter{equation}{0}
	
	In the present section we deal with the Cauchy problem 
	\begin{equation}\label{sys}
	\begin{cases}
	\bL U(t,s) = F(t), & (t,s)\in\Delta_T,
	\\[1ex]
	U(s,s)  = G, & s\in[0,T), 
	\end{cases}
	\end{equation}
	
	on the simplex $\Delta_{T}:=\{(t,s)\vert\ 0\leq s\leq t\leq T\}$, where
	\begin{equation}\label{L}
	\bL(t,D_t;x,D_x) = D_t + \Lambda(t;x,D_x) +R(t;x,D_x),
	\end{equation}
	 $\Lambda$ is a ($N\times N$)-dimensional, diagonal operator matrix, whose entries
	 $\lambda_j(t;x,D_x)$, $j=1,\dots, N$, are pseudo-differential operators
	 with real-valued, parameter-dependent symbols $\lambda_j(t;x,\xi)\in C^\infty([0,T]; \SG^{1,1})$,
	 $R$ is a parameter-dependent, ($N\times N$)-dimensional operator matrix of pseudo-differential operators
	 with symbols in $C^\infty([0,T];\SG^{0,0})$,
	 $F\in C^\infty([0,T],H^{r,\varrho}\otimes\R^N)$, $G\in H^{r,\varrho}\otimes\R^N$, $r,\varrho\in\R$.
	
	The system \eqref{L} is then of hyperbolic type, since
	the principal symbol part $\diag(\lambda_j(t;x,\xi))_{j=1,\dots,N}$ of the coefficient matrix is
	diagonal and real-valued.
	Then, its fundamental solution $E(t,s)$ exists (see \cite{CO}), and can be obtained as an
	infinte sum of matrices of Fourier integral operatos (see \cite{Kumano-go:1,Taniguchi02} and Section 5 of \cite{AsCor}
	for the SG case).
	Here we are going to show that if \eqref{sys} is of involutive type,
	then its fundamental solution $E(t,s)$ can be reduced to a finite sum expression,
	modulo a smoothing remainder, in the same spirit of \cite{Kumano-go:1, Taniguchi02},
	by applying the results from Section \ref{sec:sgphfcomlaw} above.	
		
The fundamental solution of 
	\eqref{sys} is a family $\{E(t,s)\vert (t,s)\in \Delta_{T'}\}$ of operators satisfying 
	\begin{equation}\label{tocheck}
	\begin{cases}
	\bL E(t,s) = 0, 	& (t,s)\in \Delta_{T'},
	\\[1ex]
	E(s,s)=I,						& s\in[0,T'),
	\end{cases}
	\end{equation}
	for $0< T'\leq T$.
	For $T'$ small enough, see Section 5 of \cite{AsCor},
	it is possible to express $\{E(t,s)\}$ in the form
	\begin{equation}\label{eq:E}
	E(t,s) = I_\varphi(t,s) + \int_s^t I_\varphi(t,\theta)\sum_{\nu=1}^\infty W_\nu(\theta,s)\,d\theta,
	\end{equation}
	where $I_\varphi(t,s)$ is the operator matrix defined by 
	 \[
	 I_\varphi(t,s) = \begin{pmatrix} I_{\varphi_1}(t,s) & & 0 
	 \\[1ex]
	  & \ddots & 
	  \\[1ex]
	   0 & & I_{\varphi_N}(t,s)\end{pmatrix}
	 \]
	and $I_{\varphi_j}:=Op_{\varphi_j}(1)$, $1\leq j\leq N$.
	The phase functions $\varphi_j = \varphi_j(t,s;x,\xi),\ 1\leq j\leq N$,
	defined on $\Delta_{T'}\times\R^{2n}$,
	are solutions to the eikonal equations \eqref{eik} with $-\lambda_j$ in place of $a$.
	The sequence of ($N\times N$)-dimensional matrices of SG \fios\
	$\{W_\nu(t,s);(t,s)\in\Delta({T'})\}_{\nu\in\N}$ is defined recursively as 
	\begin{equation}\label{eq:Wn+1}
	W_{\nu+1}(t,s;x,D_x) = \int_s^t W_1(t,\theta;x,D_x)W_\nu(\theta,s;x,D_x)d\theta,
	\end{equation}
	starting with $W_1$ defined as
	\begin{equation}\label{heart}
	\bL I_{\varphi}(t,s) = i W_1(t,s).
	\end{equation}
	We also set
		\begin{equation}\label{omega_j}
				w_{j}(t,s;x,\xi)=\sigma(W_{j}(t,s;x,D_x)),\quad j=1,\dots,\nu+1,\ldots .
		\end{equation}
		the (matrix-valued) symbol of $W_{j}$,

	\par
	
The following result about existence and uniqueness of a solution $U(t,s)$ 
to the Cauchy problem \eqref{sys} is a SG variant of the classical Duhamel formula, see \cite{AsCor,CO,Coriasco:998.2}.

\par

\begin{proposition}\label{wp}
For $F\in C^\infty([0,T]; H^{r,\varrho}(\R^n)\otimes\R^N)$ and $G\in H^{r,\varrho}(\R^n)\otimes\R^N$,
the solution $U(t,s)$ of the Cauchy problem \eqref{sys}, under the SG-hyperbolicity assumptions
explained above, exists uniquely for $(t,s)\in\Delta_{T'}$, $T'\in(0,T]$ suitably small,
it belongs to the class $\displaystyle\bigcap_{k\in\Z_+}C^k(\Delta_{T'}; H^{r-k,\varrho-k}(\R^n)\otimes\R^N)$, and is given by
$$
U(t,s)=E(t,s)G+i\ds\int_s^t E(t,\s)F(\s)d\s,\quad (t,s)\in\Delta_{T'}, s\in[0,T').
$$	
\end{proposition}

		 Notice that, since the phase functions $\varphi_j$ are solutions of eikonal equations
	\eqref{eik} associated with the Hamiltonians $-\lambda_j$,  we have the relation
	
	\beqsn
	D_t \Op_{\varphi_j(t,s)}(1) + \Op(\lambda_j(t))\Op_{\varphi_j(t,s)}(1)
	=
	\Op_{\varphi_j(t,s)}(b_{0,j}(t,s)),\quad b_{0,j}(t,s)\in  \SG^{0,0}(\R^{2n}),
	\eeqsn
	$j=1,\dots,N$. Then, 
	\beqs\label{eq:W_1}
	W_1(t,s):= -i\left( 
	\begin{pmatrix} B_{0,1}(t,s) & & 0 
		\\[1ex]
		 & \ddots & 
		 \\[1ex]
		  0 & & B_{0,N}(t,s)
	  \end{pmatrix} + R(t) I_{\varphi}(t,s)\right),
	\eeqs
	with $B_{0,j}(t,s)=\Op_{\varphi_j(t,s)}(b_{0,j}(t,s))$ and $b_{0,j}(t,s)\in S^{0,0}$, $j=1,\dots, N$. 

	\par
	
	By \eqref{eq:W_1} and Theorem \ref{thm:compi}, one can rewrite equation \eqref{heart} as
	\begin{equation}\label{heartbis}
	LI_{\varphi}(t,s) =  \sum_{j=1}^{N} \widetilde{W}_{\varphi_j}(t,s),
	\end{equation}
	where $ \widetilde{W}_{\varphi_j}(t,s)$ are ($N\times N$)-dimensional matrices,
	 with entries given by Fourier integral operators with parameter-dependent phase function $\varphi_j$ and 
	 symbol in $S^{0,0}$, $1\leq j\leq N$.
	Thus, if we set $M_\nu=\{\mu=(N_1,\dots,N_\nu):N_k=1,\dots,N, \ k=1,\dots,\nu\}$ for $\nu \geq 2$,
	the operator matrix $W_\nu(t,s)$ can be written in the form of iterated integrals, namely
	\begin{align}
	\int_s^t\int_s^{\theta_1}\ldots\int_s^{\theta_{\nu-2}}  \sum\limits_{\mu\in M_\nu}
	W ^{(\mu)}(t,\theta_1,\dots,\theta_{\nu-1},s)d\theta_{\nu-1}\ldots d\theta_1,
	\end{align}
	where 
	$$
	W ^{(\mu)}(t,\theta_1,\dots,\theta_{\nu-1},s)
	=
	W_{\varphi_{N_1}}(t,\theta_1) W_{\varphi_{N_2}}(\theta_1,\theta_2)
		\dots W_{\varphi_{N_\nu}}(\theta_{\nu-1},s)
	$$
	is the product of $\nu$ Fourier integral operators matrices with regular phase functions $\varphi_{N_j}$
	and elliptic symbols $\s(W_{\varphi_{N_j}}(\theta_{j-1},\theta_{j}))
	=-i\s(\widetilde{W}_{\varphi_{N_j}}(\theta_{j-1},\theta_{j}))\in S^{0,0}$.
	By (2) in Proposition \ref{thm:main}, $W ^{(\mu)}(t,\theta_1,\dots,\theta_{\nu-1},s)$
	is a matrix of Fourier integral operators with phase function
	$\phi ^{(\mu)}=\varphi_{N_1}\sharp\dots\sharp\varphi_{N_\nu}$ and parameter-dependent elliptic symbol
	$\omega ^{(\mu)}(t,\theta_1,\dots,\theta_{\nu-1},s)$ of order $(0,0)$. Consequently, we can write
	\begin{equation}\label{eq:fundsol}
	\begin{aligned}
	E(t,s)&= I_\varphi(t,s) + \int_s^t I_\varphi(t,\theta)\bigg\{\sum_{j=1}^{N} W_{\varphi_j}(\theta,s)
	\\[1ex]
	&+\sum\limits_{\nu=2}^\infty \sum\limits_{\mu\in M_\nu}\int_s^\theta\int_s^{\theta_1}
	\!\!\!\!\ldots\int_s^{\theta_{\nu-2}}  
	\!\!\!W ^{(\mu)}(\theta,\theta_1,\dots,\theta_{\nu-1},s)d\theta_{\nu-1}\ldots d\theta_1\bigg\}d\theta.
	\end{aligned}
	\end{equation}
	
	Given the commutative properties of the product of the Fourier integral operators appearing in the
	espression of $E(t,s)$ under Assumption \ref{assumpt:invo}, which follow from the results proved
	in Section \ref{sec:sgphfcomlaw}, our second main result, the next Theorem \ref{thm:mainthm}, can be proved by an argument
	analogous to the one illustrated in \cite{Taniguchi02} (the details of the proof in the SG symbol classes can be found in \cite{Abdeljawadthesis}).
	
	\begin{theorem}\label{thm:mainthm}
		Let \eqref{sys} be an involutive SG-hyperbolic system, that is, Assumption \ref{assumpt:invo} is fulfilled by the family
		$\{\lambda_j\}_{j=1}^N$. 
		Then, the fundamental solution \eqref{eq:fundsol} can be reduced, modulo smoothing terms, to
		\begin{equation}\label{eq:reducedfundsol}
		\begin{aligned}
		E(t,s)=&I_\varphi(t,s)+\sum\limits_{j=1}^{N}W_{\varphi_j}^\nmid(t,s)
		\\[1ex]
		&+\sum\limits_{j=2}^N \sum\limits_{\mu\in                    
			M_j^\nmid}\int_s^t\int_s^{\theta_1}\ldots\int_s^{\theta_{j-2}}  
		{W}^{(\mu^\nmid)}(t,\theta_1,\dots,\theta_{j-1},s)d\theta_{j-1}\ldots d\theta_1,
		\end{aligned}
		\end{equation}
		where the symbol of $W_{\varphi_j}^\nmid(t,s)$ is $\displaystyle{\int_s^t} w_{j}(\theta,s)\,d\theta$,
		with $w_j$ in  \eqref{omega_j},
		 $\mu^\nmid=(N_1,\dots,N_j)\in M_j^\nmid:=\{N_1<\dots<N_j:\,N_k=1,\dots,N,\ k=1,\dots,j\}$,
		and ${W}^{(\mu^\nmid)}(t,\theta_1,\dots,\theta_{j-1},s)$ is a ($N\times N$)-dimensional matrix of SG \fios\ 
		with regular phase function ${\phi}^{(\mu^\nmid)}=\varphi_{N_1}\sharp\dots\sharp\varphi_{N_j}$ 
		and matrix-valued, parameter-dependent symbol ${\omega}^{(\mu^\nmid)}(t,\theta_1,\dots,\theta_{j-1},s)\in \SG^{0,0}(\R^{2n})$.
	\end{theorem}
	
	\begin{remark}\label{rem:fundsol}
		Theorem \ref{thm:mainthm} can clearly be applied to the case of a $N\times N$ system such that $\Lambda$ is diagonal and its symbol entries
		$\lambda_j$, $j=1,\dots, N$, coincide with the (repeated) elements of a family of real-valued, parameter-dependent 
		symbols $\{\tau_j\}_{j=1}^m$, $1 < m < N$, satisfying Assumption \ref{assumpt:invo}.
		In such situation, working initially ``block by block''
		of coinciding elements, and then performing the reduction of \eqref{eq:fundsol} to \eqref{eq:reducedfundsol}, 
		through further applications of the commutative properties proved above, as well as of the associative properties of the multi-products 
		mentioned in Proposition \ref{prop:tderivofmultiph}, we see that \eqref{eq:reducedfundsol} can be further reduced to
		\begin{equation}\label{eq:reducedfundsolbis}
		\begin{aligned}
		E(t,s)&=I_\varphi(t,s)+\sum\limits_{j=1}^{m}W_{\varphi_j}^\nmid(t,s)
		\\[1ex]
		&+\sum\limits_{j=2}^m \sum\limits_{\mu\in                    
			M_j^\nmid}\int_s^t\int_s^{\theta_1}\ldots\int_s^{\theta_{j-2}}  
		{W}^{(\mu^\nmid)}(t,\theta_1,\dots,\theta_{j-1},s)d\theta_{j-1}\ldots d\theta_1,
		\end{aligned}
		\end{equation}
		with $\mu^\nmid=(m_1,\dots,m_j)\in M_j^\nmid:=\{m_1<\dots<m_j:\,m_k=1,\dots,m,\ k=1,\dots,j\}$,
		and  ($N\times N$)-dimensional matrices of \fios\ 
		with phase function ${\phi}^{(\mu^\nmid)}=\varphi_{m_1}\sharp\dots\sharp\varphi_{m_j}$ 
		and matrix-valued, parameter-dependent symbol as above. 
	\end{remark}

\section{Cauchy problems for weakly SG-hyperbolic linear operators}
\label{sec:sginvcp}\setcounter{equation}{0}

Here we employ the results from the previous section to the study of Cauchy problems associated with linear hyperbolic
differential operators of SG type. After obtaining the fundamental solution, we study the propagation of singularities
in the case of SG-classical coefficients.
We recall here just a few basic definitions, see \cite{CO, Coriasco:998.1, Coriasco:998.2, Cothesis, CoPa} for more details.

\begin{definition}\label{def:4.13}
Let $m\in \N$,  $T>0$,
and $L$ be a differential operator of order $m$, that is
\begin{equation}\label{eq:4.69}
L(t,D_t;x,D_x)= D_{t}^m + \sum _{j=1}^{m}p_j(t;x,D_x)D_t^{m-j}
= D_{t}^m + \sum _{j=1}^{m}\sum_{|\alpha|\le j}c_{j\alpha}(t;x)D_x^\alpha D_t^{m-j}
\end{equation}
with symbol $p_{j}(t;x,\xi)=\displaystyle\sum_{|\alpha|\le j} c_{j\alpha}(t;x)\xi^\alpha$
such that $p_j\in C^\infty([0,T];\SG^{j,j}(\R^{2n}))$, that is
\[
	|\partial_t^k\partial^\beta_x c_{j\alpha}(t;x)|\lesssim \x^{j-|\beta|}, \quad \beta\in\Z_+^{n}, j=1,\dots,m,\alpha\in\Z_+^n, |\alpha|\le j .
\]
We denote by
\[
L(t,\tau;x,\xi)  = \tau^m + \sum _{j=1}^{m}p_j(t;x,\xi)\tau ^{m-j}
\] the symbol of the operator $L$, and by
\[
L_m(t,\tau;x,\xi)  = \tau^m + \sum _{j=1}^{m}q_j(t;x,\xi)\tau ^{m-j}
\] 
its principal symbol, where, for $j=1,\dots,m$, $q_{j}(t;x,\xi)=\displaystyle \sum_{|\alpha|=j}\widetilde{c}_{j\alpha}(t;x) \xi^\alpha$ belongs to $C^\infty([0,T];\SG^{j,j}(\R^{2n}))$ and is such that 
\begin{equation*}\label{eq:4.72}
(L-L_m)(t,\tau;x,\xi) = \sum_{j=1}^m r_{j}(t;x,\xi) \tau^{m-j},
\quad  r_{j} \in C^\infty([0,T];\SG^{j-1,j-1}(\R^{2n})).
\end{equation*}
\end{definition}

\begin{definition}\label{def:sghyp}
An operator $L$ of the type introduced in Definition \ref{def:4.13} is called hyperbolic if 
\begin{equation}\label{roots}
	L_m(t,\tau;x,\xi)=\prod_{j=1}^m\left(\tau-\tau_j(t,x,\xi)\right),
\end{equation}
with real-valued, smooth roots $\tau_j$, $j=1,\dots,m$. The roots $\tau_j$ are usually called \textit{characteristics}.
More precisely, $L$ is called:

\begin{enumerate}
\item {\em strictly SG-hyperbolic}, if $L_m$ satisfies
\eqref{roots} with real-valued, distinct and separated roots $\tau_j$, $j=1,\dots,m$, in the sense that there exists a constant $C>0$ such that 
\begin{equation}\label{def:strict}
	|\tau_j(t;x,\xi)-\tau_k(t;x,\xi)|\geq C\x\csi,\quad \forall j\neq k,\ (t;x,\xi)\in[0,T]\times\R^{2n};
\end{equation}
\item {\em (weakly) SG-hyperbolic with (roots of) constant multiplicities}, if $L_m$ satisfies
\eqref{roots} and the real-valued, characteristic roots can be divided into $\mu$ groups ($1\leq \mu\leq m$) of distinct and separated roots, in the sense that, possibly
after a reordering of the $\tau_j$, $j=1,\dots, m$, there exist 
$l_1,\ldots l_\mu\in\N$ with $l_1+\ldots+l_\mu=m$ and $n$ sets
\[ G_1=\{\tau_1=\cdots=\tau_{l_1}\}, G_2=\{\tau_{l_1+1}=\cdots=\tau_{l_1+l_2}\}, \ldots  
G_\mu=\{\tau_{m-l_\mu+1}=\cdots=\tau_{m}\},\]
satisfying, for a constant $C>0$,
\begin{equation}\label{def:constmult}\tau_j\in G_p,\tau_k\in G_q ,\ p\neq q,\ 1\leq p,q\leq \mu\Rightarrow |\tau_j(t,x,\xi)-\tau_k(t,x,\xi)|\geq C\x\csi,\end{equation}
for all $(t,x,\xi)\in[0,T]\times\R^{2n}$; notice that, in the case $n=1$, we have only one group of $m$ coinciding roots, that is, $L_m$ 
admits a single real root of multiplicity $m$,
while for $n=m$ we recover the strictly hyperbolic case; the number $l=\max_{j=1,\dots,n}l_j$ is the \textit{maximum multiplicity of the roots of $L_m$};
\item {\em (weakly) SG-hyperbolic with involutive roots (or SG-involutive)}, if $L_m$ satisfies \eqref{roots} with 
real-valued characteristic roots such that the family $\{\tau_j\}_{j=1}^m$ satisfies Assumption \ref{assumpt:invo}. 
\end{enumerate} 
\end{definition}

\par

\begin{remark}\label{rem:4.14} 
	Roots of constant multiplicities are involutive. The converse statement is not true in general, see e.g., \cite{AsCoSu:1,Morimoto}.
\end{remark}

\par

\subsection{Fundamental solution for linear SG-involutive operators of order $m\in\N$}%

We will focus here on SG-involutive operators, see the references quoted above for the known results about SG-hyperbolic operators with constant
multiplicities. In particular, we assume that there is no couple of characteristic roots $\tau_j$, $\tau_k$, $j,k=1,\dots,m$, $k\not=j$, satisfying \eqref{def:constmult}. 
It is possible to translate the Cauchy problem
\begin{equation}
\label{eq:4.74}
\left\{
\begin{array}{ll}
L u(t,s) = f(t),          & (t,s)\in \Delta_T,
\\[1ex]
D_{t}^k u(s,s) = g_{k}, & k=0,\dots,m-1, s\in[0,T),
\end{array}
\right.
\end{equation}
for a SG-involutive operator $L$ in the sense of Definition \ref{def:sghyp},
into a Cauchy problem for an involutive system \eqref{sys} with suitable initial 
conditions, under an appropriate factorization condition, see below.

\par

We write $\Theta _{j}=\Op(\tau_j)$, and also set, for convenience below,
$\Gamma_{j} = D_{t} - \Theta _{j}$, $j=1,\dots,m$.
Moreover, with the permutations $\perm_{k}$ of $k$ elements of the set $\{1, \dots, m\}$ from Section \ref{sec:fundsol}, and their sorted counterparts
$\perms_k$, $1 \leq k\leq m$, we introduce the notation
\[
\perm_{0}  =   \{ \emptyset  \}, \quad
\perm      = \bigcup_{k=0}^{m-1} \perm_{k}, \quad
\perms    =  \bigcup_{k=1}^{m} \perms_{k}.
\]
For $\alpha \in \perm_{k}$, $0 \leq k\leq m$, we define 
$\card{\alpha} = k$ and
\[
\tD _{\emptyset}    =  I, \qquad
\tD_{\alpha}          =  \tD _{\alpha_{1}}\dots \tD _{\alpha_{k}},\,
\alpha =(\alpha _{1}, \dots, \alpha _{k} )\in M_k, k\ge1.
\]

\par

The proof of the following Lemma \ref{lemma:4.16.2.1} can be found in \cite{Cothesis}. 
Analogous results are used in \cite{Kumano-go:1} and \cite{Morimoto}.

\par

\begin{lemma}
	\label{lemma:4.16.2.1}
	When $\{ \lambda _{j} \}$ is an involutive system, for all
	$\alpha \in \perm_{m}$ we have
	\begin{equation}
	\label{eq:4.112.7}
	\Gamma _{\alpha} = \Gamma _{1} \dots \Gamma _{m} + 
	\sum_{\beta \in \perm}
	\Op(q^{\alpha}_{\beta}(t)) \Gamma _{\beta},
	\end{equation}
	where $q^{\alpha}_{\beta} \in C^\infty([0,T];\SG^{0,0}(\R^{2n}))$.
\end{lemma}

\par

A systemization and well-posedness (with loss of decay and regularity)
theorem can be stated for the 
Cauchy problem \eqref{eq:4.74} under a suitable condition for the 
operator $L$. This result is due, in its original local form,
to Morimoto \cite{Morimoto} and it has been extended to the SG case
in \cite{Coriasco:998.2}, where the proof of the next result, based on Lemma \ref{lemma:4.16.2.1}, can be found. 

\par

\begin{proposition}\label{thm:4.20}
Assume the SG-hyperbolic operator $L$ to be of the form
\begin{equation}\label{eq:4.117}
L = \Gamma _{1} \cdots \Gamma _{m} +
\sum_{\alpha \in \perms}
\Op(p_{\alpha}(t)) \Gamma _{\alpha} \; \mod \op (C^\infty([0,T]; \SG^{-\infty,-\infty}(\R^{2n}))),
\end{equation}
with $p_{\alpha} \in C^\infty([0,T];\SG^{0,0}(\R^{2n}))$. 
Moreover, assume that the family of its characteristic roots $\{\tau_j\}_{j=1}^m$ sastisfies Assumption \ref{assumpt:invo}.
Then, the Cauchy problem \eqref{eq:4.74} for $L$ is equivalent to
a Cauchy problem for a suitable first order system \eqref{sys}
with diagonal principal part, of the form
\begin{equation}
\label{eq:4.119}
\left\{
\begin{array}{l}
\ds\left(D_t +K(t)\right) U(t,s)  = F(t),
\mbox{  } (t,s) \in \Delta_T,
\\[1ex]
U(s,s) = G, \hspace*{2.5cm} s\in[0,T),
\end{array}
\right.
\end{equation}
where $U$, $F$ and $G$ are $N$-dimensional vector, $K$ a ($N\times N$)-dimensional matrix,
with $N$ given by \eqref{eq:4.120.1}. $U$ is
defined in \eqref{eq:4.120}, \eqref{eq:4.121}, and \eqref{eq:4.122}. Namely,
\begin{equation}
\label{eq:4.120.1}
N = \sum_{j=0}^{m-1} \frac{m!}{(m-j)!},
\end{equation}
\begin{equation}\label{eq:4.120}
U = \ ^t\left(u_{\emptyset}\equiv u, u_{(1)}, \dots, u_{(m)},
u_{(1,2)}, u_{(2,1)}, \dots, u_{\alpha}, \dots\right),
\end{equation}
with $\alpha \in \perm$, and
\begin{itemize}
	\item[-] for $\alpha \in \perm_{k}$, $0\le k \le m - 2$ and
	$j=\max \{ 1, \dots, m \} \setminus \alpha$, we set
	\begin{equation}
	\label{eq:4.121}
	\tD_{j} u_{\alpha} = u_{\alpha_{j}}
	\end{equation}
	with $\alpha_{j}= (j, \alpha_{1}, \dots, \alpha_{k})
	\in \perm_{k+1}$;
	\item[-] for $\alpha \in \perm_{m-1}$ and $j \notin \{ \alpha \}$,
	we set
	\begin{equation}
	\label{eq:4.122}
	\tD_{j} u_{\alpha} =   f
	- \sum_{\beta \in \perms}
	\Op(p_{\beta}(t)) u_{\beta}
	+ \sum_{\beta \in \perm}
	\Op(q^{\alpha_{j}}_{\beta}(t))
	u_{\beta},
	\end{equation}
	with $\alpha_{j}= (j, \alpha_{1}, \dots, \alpha_{k})
	\in \perm_{m}$ and the symbols $p_\beta$, $q^{\alpha}_\beta$  from \eqref{eq:4.112.7} and
	\eqref{eq:4.117}.
\end{itemize}
\end{proposition}

\par

\begin{remark}\label{rem:Levi}
	We call the SG-hyperbolic operators $L$ satisfying the factorization condition \eqref{eq:4.117}  ``operators of Levi type''.
\end{remark}

\par

\begin{remark}\label{rem:incond}
	Since, for $\alpha \in \perm_{k}$, $k\ge1$, we have
\[
\tD_{\alpha} = D^k_{t} + \sum_{j=0}^{k-1} \Op(\Upsilon^j_{\alpha}(s)) D^j_{t},
\mbox{     }
\Upsilon^j_{\alpha} \in C^\infty([0,T];\SG^{k-j,k-j}(\R^{2n})),
\]
the initial conditions $G$ for $U$ can be expressed as
\begin{equation}
\label{eq:4.124}
\left\{
\begin{array}{l}
G_{\emptyset}(s,s) = g_{0},
\\
G_{\alpha}(s,s) =   g_{\card{\alpha}}
+ \ssum{j=0}{\card{\alpha}-1}
\Op(\Upsilon^j_{\alpha}(s)) g_{j}, \quad \alpha\in M, \card{\alpha}>0.
\end{array}
\right.
\end{equation}
Notice that, in view of the continuity properties of the SG pseudo-differential operators and of the orders of the $\Upsilon^i_\alpha$,
\eqref{eq:4.124} implies 
\begin{equation}\label{eq:Gaord}
	G_{\alpha}\in H^{m-1-\card{\alpha},\mu-1-\card{\alpha}}(\R^n), \quad \alpha\in M.
\end{equation}
\end{remark}

\par

The next Theorem \ref{thm:mainLastSecThm} is our third main result, namely, a well-posedness result,
with decay and regularity loss, for SG-involutive operators of the form \eqref{eq:4.119}. It is a consequence of
Proposition \ref{thm:4.20} in combination with the main results of Section \ref{sec:fundsol}.

\par

\begin{theorem}\label{thm:mainLastSecThm}
Let the operator $L$ in \eqref{eq:4.74} be SG-involutive, of the form considered in Proposition \ref{thm:4.20}.
Let $f\in C^\infty([0,T]; H^{r,\varrho}(\R^n))$ and
$g_k\in H^{r+m-1-k,\varrho+m-1-k}(\R^n)$, $k=0,\dots,m-1$.
Then, for a suitable $T^\prime\in(0,T]$,
the Cauchy problem \eqref{eq:4.74}  admits a unique solution
$u(t,s)$,
belonging to $\displaystyle\bigcap_{k\in\Z_+}C^k(\Delta_{T^\prime}; H^{r-k,\varrho-k}(\R^n))$, given, modulo
elements in $C^\infty(\Delta_{T^\prime};\cS(\R^n))$, by
    \begin{equation}\label{eq:soleqordm}
	u(t,s)  = \sum_{\alpha\in M}W_{\alpha}(t,s) G_{\alpha}+
	\sum_{\alpha\in M_{m-1}}\int_s^t W_{\alpha}(t,\s)  f(\s)\,d\s,    \quad (t,s)\in\Delta_{T^\prime}, s\in[0,T^\prime),
    \end{equation}
    for suitable linear combinations of parameter-dependent families of (iterated integrals of) 
    regular SG Fourier integral operators $W_{\alpha}(t,s)$, $\alpha\in M$, $(t,s)\in\Delta_{T^\prime}$,
    with phase functions and matrix-valued symbols determined through the characteristic roots of $L$.
\end{theorem}
\begin{proof}
	By the procedure explained in Proposition \ref{thm:4.20} and Remark \ref{rem:incond}, we can switch from the Cauchy problem
	\eqref{eq:4.74} to an equivalent Cauchy problem \eqref{eq:4.119}, with $u\equiv U_\emptyset$. 
	The uniqueness of the solution is then a consequence of known results about symmetric SG-hyperbolic systems, see \cite{CO},
	of which \eqref{eq:4.119} is a special case.
 
	The fundamental solution of \eqref{eq:4.119} is given by \eqref{eq:reducedfundsolbis}, in view of Theorem \ref{thm:mainthm}
	and Remark \ref{rem:fundsol}.
	It is a matrix-valued, parameter-dependent operator family $E(t,s)=(E_{\mu\mu^\prime})_{\mu,\mu^\prime\in M}(t,s)$, whose elements 
	$E_{\mu\mu^\prime}(t,s)$, $\mu,\mu^\prime\in M$, are, modulo elements with kernels in $C^\infty(\Delta_{T^\prime};\cS)$, linear combinations of
	parameter-dependent families of (iterated integrals of) regular SG Fourier operators, with phase functions of the type
	\begin{align*}
		\phi^{(\mu^\nmid)}&=\varphi_{m_1}, && \mu^\nmid=(m_1)\in M^\nmid_1,
		\\
		\phi^{(\mu^\nmid)}&=\varphi_{m_1}\sharp\dots\sharp\varphi_{m_j}, &&\mu^\nmid=(m_1,\dots,m_j)\in M^\nmid_j,j\ge2, 
	\end{align*}
	$\varphi_k$ solution of the eikonal equation
	associated with the characteristic root $\tau_k$ of $L$, $k=1,\dots,m$, and parameter-dependent, matrix-valued symbols 
	of the type
	\[
		{\omega}^{(\mu^\nmid)}(t,\theta_1,\dots,\theta_{j-1},s)\in \SG^{0,0}, \quad \mu \in M^\nmid_j,
	\]
	$j=1,\dots,m$. 
	Then,  the component $U_\emptyset\equiv u$ of the solution $U$ of \eqref{sys} has the form
	\eqref{eq:soleqordm}, with $W_\alpha=E_{\varnothing\alpha}$, taking into account \eqref{eq:4.122} and \eqref{eq:4.124}. 
	
	We observe that the $k$-th order
	$t$-derivatives of the operators $W_\alpha$, $\alpha\in M$, map continuously $H^{r,\varrho}$ to $H^{r-k,\varrho-k}$, $k\in\Z_+$,
	in view of Theorem \ref{thm:fiocont} and of the fact that, of course,
	\[
		\partial_t[\Op_{\phi^{(\mu^\nmid)}(t,s)}(w^{(\mu^\nmid)}(t,s))] = \Op_{\phi^{(\mu^\nmid)}(t,s)}(i(\partial_t\phi^{(\mu^\nmid)})(t,s)\cdot w^{(\mu^\nmid)}(t,s)
		+\partial_t w^{(\mu^\nmid)}(t,s)),
	\]
	obtaining a symbol of orders $1$-unit higher in both components at any $t$-derivative step.
	This fact, together with the hypothesis on $f$, implies that the second sum 
	in \eqref{eq:soleqordm} belongs to $\displaystyle\bigcap_{k\in\Z_+}C^k(\Delta_{T^\prime}; H^{r-k,\varrho-k}(\R^n))$.
	
	The same is true for the elements of the first sum. In fact, recalling the embedding among the Sobolev-Kato spaces and \eqref{eq:Gaord}, 
	since $\alpha\in M\Rightarrow 0\le\card{\alpha}\le m-1$, we find
	\begin{align*}
		W_\alpha(t,s)G_\alpha \in& \bigcap_{k\in\Z_+} C^k(\Delta_{T^\prime}; H^{r+m-1-\card{\alpha}-k,\varrho+m-1-\card{\alpha}-k})
		\\
		\hookrightarrow &\bigcap_{k\in\Z_+} C^k(\Delta_{T^\prime}; H^{r-k,\varrho-k}), \quad \alpha\in M,
	\end{align*}
	and this concludes the proof.
\end{proof}

\subsection{Propagation of singularities for classical SG-involutive operators}
Theorem \ref{thm:mainLastSecThm}, together with the propagation results proved in \cite{CJT4},
implies our fourth main result, Theorem \ref{thm:propwfs} below, about the global wave-front set
of the solution of the Cauchy problem \eqref{eq:4.74}, in the case of a classical SG-involutive operator $L$
of Levi type.
We first recall the necessary definitions, adapting some materials appeared in \cite{CJT2,CJT4,CoMa}.

\begin{definition}\label{def:admspace}
Let $\cB$ be a topological vector space
of distributions on $\R^n$ such that
$$
\cS(\R^n)\subseteq \cB  \subseteq \cS^\prime(\R^n)
$$
with continuous embeddings. Then $\cB$ is called SG-admissible
when $\Op _t(a)$ maps $\cB$ continuously
into itself, for every $a\in \SG ^{0,0}(\R^{2n})$.
If $\cB$ and $\cC$ are SG-admissible, then the pair $(\cB ,\cC )$ is
called SG-ordered (with respect to $(m,\mu)\in\R^2$), when the mappings
$$
\Op _t(a)\, :\, \cB \to \cC \quad \text{and}\quad \Op _t(b)\, :\, \cC \to \cB
$$
are continuous for every $a\in \SG ^{m,\mu}(\R^{2n})$ and
$b\in \SG ^{-m,-\mu}(\R^{2n})$.
\end{definition}

\begin{remark}
	$\cS(\R^n)$, $H^{r,\varrho}(\R^n)$, $r,\varrho\in\R$, and $\cS^\prime(\R^n)$ are SG-admissible.
	$(\cS(\R^n),\cS(\R^n))$, $(H^{r,\varrho}(\R^n),H^{r-m,\varrho-\mu}(\R^n))$, $r,\varrho\in\R$, $(\cS^\prime(\R^n),\cS^\prime(\R^n))$ are 
	SG-ordered (with respect to any $(m,\mu)\in\R^2$). The same holds true for (suitable couples of) modulation spaces, see \cite{CJT2}.
\end{remark}

\begin{definition}\label{admspacesdef}
Let $\fy \in \Phr$ be a regular phase function,
 $\cB$, $\cB_1$, $\cB_2$, $\cC$, $\cC_1$, $\cC_2$, be
SG-admissible and $\Omega \subseteq \R^n$ be open. Then
the pair $(\cB ,\cC )$ is called weakly-I SG-ordered (with respect to
$(m,\mu,\varphi,\Omega )$), when the mapping
$$
\op_\varphi(a)\, :\, \cB \to \cC
$$
is continuous for every $a\in \SG ^{m,\mu}(\R^{2n})$ with support
such that the projection on the $\xi$-axis does not intersect $\R^n\setminus \Omega$. Similarly, the pair $(\cB ,\cC )$ is called 
weakly-II SG-ordered (with respect to
$(m,\mu,\varphi,\Omega )$), when the mapping
$$
\op_\varphi^*(b)\, :\, \cC \to \cB
$$
is continuous for every $b\in \SG ^{m,\mu}(\R^{2n})$ with support
such that the projection on the $x$-axis does not intersect $\R^n\setminus \Omega$. Furthermore,
$(\cB_1, \cC_1, \cB_2, \cC_2)$ are called SG-ordered
(with respect to $m_1,\mu_1, m_2,\mu_2$, $\varphi$,
and $\Omega$), when $(\cB_1,\cC_1)$ is a weakly-I 
SG-ordered pair with respect to
$(m_1,\mu_1,\varphi ,\Omega )$, and
$(\cB_2,\cC_2)$ is a weakly-II 
SG-ordered pair with respect to
$(m_2,\mu_2,\varphi ,\Omega )$.
\end{definition}

\begin{remark}
	$(\cS(\R^n),\cS(\R^n))$, $(H^{r,\varrho}(\R^n),H^{r-m,\varrho-\mu}(\R^n))$, $r,\varrho\in\R$, $(\cS^\prime(\R^n),\cS^\prime(\R^n))$ are 
	weakly-I and weakly-II SG-ordered pairs (with respect to any $(m,\mu)\in\R^2$, $\varphi\in\Phr$, and $\Omega=\emptyset$). The situation is more delicate
	in the case of modulation spaces, even just on Sobolev-Kato spaces modeled on $L^p(\R^n)$, $p\in[1,\infty)$, $p\not=2$, see \cite{CJT2} and the references
	quoted therein.
\end{remark}

Now we recall the definition given in \cite{CJT2} of global
wave-front sets for temperate distributions with respect to Banach
or Fr\'echet spaces and state some of their properties. First of all, we recall the definitions of set of
characteristic points that we use in this setting.

We need to deal with the situations where \eqref{eq:ellcond} holds only in certain
(conic-shaped) subset of $\R^n \times \R^n$. Here we let $\Omega _m$,
$m=1,2,3$, be the sets
\begin{equation}\label{omegasets}
\begin{aligned}
\Omega _1= \R^n\times (&\R^n\setminus \{0\}),\qquad
\Omega _2 = (\R^n \setminus \{0\})\times \R^n,
\\[1ex]
\Omega _3 &= (\R^n \setminus \{0\})\times (\R^n \setminus \{0\}),
\end{aligned}
\end{equation}

\par

\begin{definition}\label{defchar}
Let $\Omega _k$, $k=1,2,3$ be as in \eqref{omegasets},
and let $a\in \SG ^{m,\mu}(\R^{2n})$.

\par

\begin{enumerate}

\item $a$ is called \emph{locally} or \emph{type-$1$ invertible} 
with respect to $m,\mu$ at the
point $(x_0,\xi_0)\in \Omega _1$, if
there exist a neighbourhood $X$ of $x_0$, an open conical
neighbourhood $\Gamma$ of $\xi _0$  and a positive constant $R$
such that \eqref{eq:ellcond} holds for $x\in X$, $\xi\in \Gamma$ and
$|\xi|\ge R$.

\par

\item  $a$ is called
\emph{Fourier-locally} or \emph {type-$2$ invertible} with respect to
$m,\mu$ at the point $(x_0,\xi_0)\in \Omega _2$, if
there exist an open conical neighbourhood $\Gamma$  of $x_0$, a
neighbourhood $X$ of $\xi _0$ and a positive constant $R$ such
that \eqref{eq:ellcond} holds for $x\in \Gamma$, $|x|\ge R$ and  $\xi\in
X$.

\par

\item  $a$ is called
\emph{oscillating} or \emph{type-$3$ invertible} with respect to
$m,\mu$ at the point $(x_0,\xi_0)\in \Omega _3$, if
there exist open conical neighbourhoods $\Gamma _1$  of $x_0$ and
$\Gamma _2$ of $\xi _0$, and a positive constant $R$
such that \eqref{eq:ellcond} holds for $x\in \Gamma _1$, $|x|\ge R$,
$\xi \in \Gamma _2$ and $|\xi |\ge R$.

\end{enumerate}

\par

If $k\in \{ 1,2,3\}$ and  $a$ is \emph{not} type-$k$ invertible
with respect to $m,\mu$ at $(x_0,\xi_0)\in \Omega _k$,
then $(x_0,\xi_0)$ is called \emph{type-$k$ characteristic} for $a$ with
respect to $m,\mu$. The set of type-$k$
characteristic points for $a$ with respect to $m,\mu$ is denoted by
$\Char _{m,\mu}^k(a)$.

\par

The \emph{(global) set of characteristic points} (the characteristic set), for a
symbol $a\in \SG^{m,\mu}(\R^{2n})$ with respect to $m,\mu$ is defined as
$$
\Char (a)=\Char _{m,\mu}(a)=\Char ^1_{m,\mu}(a)\bigcup\Char ^2 _{m,\mu}(a)\bigcup\Char ^3_{m,\mu}(a).
$$
\end{definition}

\par

In the next Definition \ref{cuttdef} we introduce different classes of cutoff
functions (see also Definition 1.9 in \cite{CJT1}).

\par

\begin{definition}\label{cuttdef}
Let $X\subseteq \R^n$ be open, $\Gamma \subseteq \R^n
\setminus \{0\}$ be an open cone, $x_0\in X$ and $\xi _0\in \Gamma$.

\begin{enumerate}
\item A smooth function $\fy$ on $\R^n$ is called a cutoff
(function) with respect to $x_0$ and $X$, if $0\le \fy \le 1$, $\fy \in
C_0^\infty (X)$ and $\fy =1$ in an open neighbourhood of $x_0$. The
set of cutoffs with respect to $x_0$ and $X$ is denoted by
$\mathscr C_{x_0}(X)$ or $\mathscr C_{x_0}$.

\par

\item  A smooth function $\psi$ on $\R^n$ is called a
directional cutoff (function) with respect to $\xi_0$ and
$\Gamma$, if there is a constant $R>0$ and open conical neighbourhood
$\Gamma _1\subseteq \Gamma$ of $\xi _0$ such that the following is
true:
\begin{itemize}
\item $0\le \psi \le 1$ and $\supp \psi \subseteq \Gamma$;

\par

\item  $\psi (t\xi )=\psi (\xi )$ when $t\ge 1$ and $|\xi |\ge R$;

\par

\item $\psi (\xi )=1$ when $\xi \in \Gamma _1$ and $|\xi |\ge R$.
\end{itemize}

\par
\noindent 
The set of directional cutoffs with respect to $\xi _0$ and
$\Gamma$ is denoted by $\mathscr C^\direz  _{\xi _0}(\Gamma )$ or
$\mathscr C^\direz  _{\xi _0}$.
\end{enumerate}
\end{definition}

\par

\begin{remark}\label{psiinvremark}
Let $X\subseteq \R^n$ be open and $\Gamma ,\Gamma _1,\Gamma _2
\subseteq \R^n\back 0$
be open cones. Then the following is true.
\begin{enumerate}
\item if $x_0\in X$, $\xi _0\in \Gamma$, $\fy \in \mathscr
C _{x _0}(X)$ and $\psi \in \mathscr C ^{\direz} _{\xi _0}(\Gamma )$,
then $c_1=\fy \otimes \psi$ belongs to $\SG^{0,0}(\R^{2n})$,
and is type-$1$ invertible at $(x_0,\xi _0)$;

\par

\item if $x_0\in \Gamma$, $\xi _0\in X$, $\psi \in \mathscr C ^{\direz}
_{x_0}(\Gamma )$ and $\fy \in \mathscr C _{\xi  _0}(X)$,
then $c_2=\fy \otimes \psi$ belongs to $\SG^{0,0}(\R^{2n})$,
and is type-$2$ invertible at $(x_0,\xi _0)$;

\par

\item if $x_0\in \Gamma _1$, $\xi _0\in \Gamma _2$, $\psi _1\in
\mathscr C ^{\direz} _{x_0}(\Gamma _1)$ and $\psi _2\in
\mathscr C ^{\direz} _{\xi _0}(\Gamma _2)$, then $c_3=\psi _1 \otimes
\psi _2$ belongs to $\SG^{0,0}(\R^{2n})$, and is type-$3$
invertible at $(x_0,\xi _0)$.
\end{enumerate}
\end{remark}

\par

The next Proposition \ref{charequiv} shows that $\op _t(a)$ for
$t\in \mathbb R$ satisfies
convenient invertibility properties of the form
\begin{equation}\label{locinvop}
\op _t(a)\op _t(b) = \op _t(c) + \op _t(h),
\end{equation}
outside the set of characteristic
points for a symbol $a$. Here $\op _t(b)$, $\op _t(c)$ and $\op _t(h)$ have
the roles of ``local inverse'', ``local identity'' and smoothing operators
respectively. From these statements it also follows that our set of
characteristic points in Definition \ref{defchar} are related to those
in \cite{CoMa,Ho1}. 

\par

\begin{proposition}\label{charequiv}
Let $k\in \{1,2,3\}$, $m,\mu\in\R$,
 and let $a\in \SG ^{m,\mu}(\R^{2n})$. Also let $\Omega _k$ be as in \eqref{omegasets}, $(x_0,\xi _0)\in
\Omega _k$, when $k$ is equal to $1$, $2$ and $3$, respectively.
Then the following conditions are equivalent, $k=1,2,3$:

\medspace

\begin{enumerate}

\item $(x_0,\xi _0)\notin \Char _{m,\mu}^k (a)$;

\par

\item there is an element $c\in \SG^{0,0}(\R^{2n})$ which is
type-$k$ invertible at $(x_0,\xi _0)$, and an element $b\in
\SG^{-m,-\mu}(\R^{2n})$ such that $ab=c$;

\par

\item \eqref{locinvop} holds for some $c\in \SG^{0,0}(\R^{2n})$ which is
type-$k$ invertible at $(x_0,\xi _0)$, and some elements  $h\in
\SG^{-1,-1}(\R^{2n})$ and $b\in \SG^{-m,-\mu}(\R^{2n})$;

\par

\item \eqref{locinvop} holds for some $c_k\in \SG^{0,0}(\R^{2n})$
in Remark \ref{psiinvremark} which is
type-$k$ invertible at $(x_0,\xi _0)$, and some elements  $h$ and $b\in
\SG^{-m,-\mu}(\R^{2n})$, where $h\in \cS$ when $k\in
\{ 1,3\}$ and $h\in \SG ^{-\infty ,0}(\R^{2n})$ when $k=2$.

\par

Furthermore, if $t=0$, then the supports of $b$ and $h$ can be chosen to be
contained in $X\times \R^n$ when $k=1$, in $\Gamma \times \R^n$ when
$k=2$, and in $\Gamma _1\times \R^n$ when $k=3$.
\end{enumerate}
\end{proposition}

\par

We can now introduce the complements of the wave-front sets.
More precisely, let $\Omega _k$, $k\in \{ 1,2,3\}$, be given by
\eqref{omegasets}, $\cB$ be a Banach or Fr{\'e}chet space such
that $\cS(\R^n)\subseteq \cB \subseteq \cS^\prime(\R^n)$,
and let $f\in \cS^\prime(\R^n)$. Then the point $(x_0,\xi _0)\in
\Omega _k$ is called \emph{type-$k$} regular for $f$ with
respect to $\cB$, if
\begin{equation}\label{ImageOpcm}
\Op (c_k)f\in \cB ,
\end{equation}
for some $c_k$ in Remark \ref{psiinvremark}, $k=1,2,3$. The set of all type-$k$
regular points for $f$ with respect to $\cB$, is denoted by
$\Theta ^k_{\cB}(f)$.

\par

\begin{definition}\label{def:wfsMB}
Let $k\in \{ 1,2,3\}$, $\Omega _k$ be as in \eqref{omegasets}, and
let $\cB$ be a Banach or Fr\'echet space such that
$\cS(\R^n)\subseteq \cB \subset \cS^\prime(\R^n)$.
\begin{enumerate}
\item The \emph{type-$k$ wave-front set} of $f\in \cS^\prime(\R^n)$ with
respect to $\cB$ is the complement of $\Theta ^k_{\cB}(f)$ in $\Omega
_k$, and is denoted by $\WFF ^k_{\cB}(f)$;

\par

\item The \emph{global wave-front set} $\WFgB(f)\subseteq (\R^n\times\R^n)\back 0$ is the set
\begin{equation*}
\WFgB(f) \equiv \WFF ^1_\cB (f) \bigcup \WFF ^2_\cB (f) \bigcup \WFF ^3_\cB (f).
\end{equation*}
\end{enumerate}
\end{definition}

\par

The sets $\WFF ^1_{\cB}(f)$, $\WFF ^2_{\cB}(f)$ and $\WFF ^3_{\cB}(f)$
in Definition \ref{def:wfsMB}, are also called the \emph{local},
\emph{Fourier-local} and \emph{oscillating} wave-front set of $f$ with
respect to $\cB$.

\par

\begin{remark}
	In the special case when $\cB=H^{r,\varrho}(\R^n)$, $r,\varrho\in\R$, we write $\WFF ^k_{r,\varrho}(f)$,
	$k=1,2,3$. In this situation, $\WFF_{r,\varrho}(f) \equiv \WFF ^1_{r,\varrho} (f) \bigcup \WFF ^2_{r,\varrho} (f) \bigcup \WFF ^3_{r,\varrho} (f)$
	coincides with the scattering wave front set of $f\in\cS^\prime(\R^n)$ introduced by Melrose \cite{Me}. In the case when $\cB=\cS(\R^n)$,
	$\WFgB(f)$ coincides with the $\cS$-wave-front set considered in \cite{CoMa} (see also \cite{Schulz}).
\end{remark}

\par

\begin{remark}\label{rem:wfcon}
        Let $\Omega _m$, $m=1,2,3$ be the same as in \eqref{omegasets}. 
	\begin{enumerate}
		\item If $\Omega \subseteq \Omega _1$, and
		$(x_0,\xi_0)\in \Omega \ \Longleftrightarrow \ 
			(x_0,\sigma \xi_0)\in \Omega$ for $\sigma \ge 1$,
			then $\Omega$ is called \emph{$1$-conical};

		\item If $\Omega \subseteq \Omega _2$, and $(x_0,\xi_0)\in
		\Omega \ \Longleftrightarrow \ 
			(sx_0,\xi_0)\in\Omega$ for $s \ge 1$,
			then $\Omega$ is called \emph{$2$-conical};

		\item If $\Omega \subseteq \Omega _3$, and $(x_0,\xi_0)
		\in \Omega\ \Longleftrightarrow \ 
			(sx_0,\sigma\xi_0)\in \Omega$ for $s,\sigma \ge 1$,
		then $\Omega$ is called \emph{$3$-conical}.
	\end{enumerate}
        By \eqref{ImageOpcm} and the paragraph before Definition \ref{def:wfsMB},
	it follows that if $m=1,2,3$, then $\Theta^m_\cB (f)$ is $m$-conical.
	The same holds for $\WFF^m_\cB(f)$, $m=1,2,3$, 
	by Definition \ref{def:wfsMB}, noticing that, for any $x_0\in
	\R^{r} \setminus \{0\}$, any open cone $\Gamma\ni x_0$,
	and any $s>0$, $\mathscr C ^{\direz} _{x_0}(\Gamma )
	=\mathscr C ^{\direz} _{s x_0}(\Gamma )$.
	For any $R>0$ and $m\in \{1,2,3\}$, we set
\begin{gather*}
\Omega_{1,R} \equiv \sets {(x,\xi )\in \Omega _1}{|\xi |\geq R},
\quad
\Omega_{2,R} \equiv \sets {(x,\xi )\in \Omega _2}{|x |\geq R},
\\[1ex]
\Omega_{3,R} \equiv \sets {(x,\xi )\in \Omega _3}{|x|, |\xi |\geq R}
\end{gather*}
	Evidently, $\Omega_m^R$ is $m$-conical for every $m\in \{ 1,2,3\}$.
\end{remark}

\par

From now on we assume that $\cB$ in Definition \ref{def:wfsMB} is
SG-admissible, and recall
that Sobolev-Kato spaces and, more generally, modulation spaces, and
$\mathscr S(\R^d)$ are SG-admissible, see \cite{CJT2, CJT4}.

\par

The next result describes the relation between ``regularity with respect to
$\cB$\,'' of temperate distributions and global wave-front sets.

\par

\begin{proposition}\label{mainthm4}
Let $\cB$ be SG-admissible, and let $f\in\cS^\prime (\R^ d)$. Then
\begin{equation*}
f\in\cB \quad \Longleftrightarrow \quad
\WFgB(f)=\emptyset.
\end{equation*}
\end{proposition}

\par

Theorem \ref{thm:propwfs} extends the analogous result in \cite{CJT4} to the more general case of 
a classical, SG-hyperbolic involutive operator $L$ of Levi type, and the one in \cite{Taniguchi02}
to the global wave-front sets introduced above. It is a consequence of Theorem \ref{thm:mainLastSecThm}
and of Theorem 5.17 in \cite{CJT4}.

\begin{theorem}\label{thm:propwfs}
		Let $L$ in \eqref{eq:4.74} be a classical, SG-hyperbolic, involutive operator of Levi type, that is,
		of the type considered in Proposition \ref{thm:4.20} with SG-classical coefficients, of the form \eqref{eq:4.117}.
		Let $g_k\in\cS^\prime(\R^n)$, $k=0,\dots,m-1$, and assume that the
		$m$-tuple 
    		$(\cB_0,\dots,\cB_{m-1})$ consists of SG-admissible spaces.
    		Also assume that the SG-admissible space $\cC$ is such that $(\cB_k,\cC)$, $k=0,\dots,m-1$, 
    		are weakly-I SG-ordered pairs with respect to
    		\[
    			 k-j, k-j, k=0,\dots,m-1, j=0,\dots, k, \phi^{(\alpha)}, \alpha\in M, \text{ and } \emptyset.
    		\]
    Then, for the solution $u(t,s)$ of the Cauchy problem \eqref{eq:4.74} with $f\equiv0$, $(t,s)\in\Delta_{T^\prime}$, $s\in[0,T^\prime)$, we find
    \begin{equation}\label{eq:wfprop}
    	\WFF^k_\cC( u(t,s) ) \subseteq  \bigcup_{j=1}^m
    	\bigcup_{\alpha\in M^\nmid_j}
	   \bigcup_{\substack{\bt_j\in\Delta_j(T^\prime)
	   \\ t_0=t,t_j=s}}
	   \bigcup_{\ell=0}^{m-1}              
	                 	(\Phi_{\alpha}(\bt_j)(\WFF^k_{\cB_\ell} 
	                 	(g_\ell)))^{\mathrm{con}_k} , \quad k=1,2,3,
     \end{equation}
    where $V^{\mathrm{con}_k}$ for $V\subseteq \Omega_k$,
is the smallest $k$-conical subset of $\Omega_k$ which includes $V$, $k\in\{1,2,3\}$
and  $\Phi_\alpha(\bt_j)$ is the canonical transformation of $T^\star \R^{n}$ into itself generated 
by the parameter-dependent SG-classical phase functions $\phi^{(\alpha)}(\bt_j)\in\Phr$, $\alpha\in M_j^\nmid$, $\bt_j\in\Delta_j(T^\prime)$,
$t_0=t$, $t_j=s$, $j=1,\dots,m$, appearing in 
\eqref{eq:soleqordm}. 
\end{theorem} 

\begin{proof}
We prove \eqref{eq:wfprop} only for the case $k=3$, since the arguments for the cases $k=1$ and $k=2$ are analogous. Let 
\[
	(x_0,\xi_0)\in \bigcap_{j=1}^m
    	\bigcap_{\alpha\in M^\nmid_j}
	   \bigcap_{\substack{\bt_j\in\Delta_j(T^\prime)
	   \\ t_0=t,t_j=s}}
	   \bigcap_{\ell=0}^{m-1}              
	                 	[\Omega_3\setminus(\Phi_{\alpha}(\bt_j)(\WFF^3_{\cB_\ell} 
	                 	(g_\ell)))^{\mathrm{con}_3}].
\]
Then, by \eqref{eq:soleqordm} and \cite[Theorem 5.17]{CJT4}, in view of \eqref{eq:4.124} and the hypotheses 
on $\cC$, $\cB_k$, $k=0,\dots,m-1$, there exists $c_3\in \SG^{0,0}$, type-$3$ invertible at $(x_0,\xi_0)$,
such that $\Op(c_3)(W_\alpha(t,s)G_\alpha)\in\cC$, $\alpha\in M^\nmid_j$, $j=1,\dots,m$. 

In fact, for the terms such that $\alpha\in M^\nmid_1$, this follows by a direct application of Theorem 5.17 in \cite{CJT4}.
For the terms with $\alpha\in M^\nmid_j$, $j=2,\dots, m$, we first observe that we can bring $\Op(c_3)$ within the
iterated integrals analogous to those appearing in \eqref{eq:reducedfundsol}. Then, in view of the hypothesis on $(x_0,\xi_0)$,
again by Theorem 5.17 in \cite{CJT4} and the properties of the operators $W_\alpha$ in \eqref{eq:soleqordm},
the integrand belongs to $\cC$ for any $\bt_j$, $j=2,\dots,m$, in the integration domain. In view of the smooth dependence
on $\bt_j$, the iterated integral belongs to $\cC$ as well. Since the right-hand side of \eqref{eq:soleqordm} is a finite sum,
modulo a term in $\cS$, we conclude that $\Op(c_3)(u(t,s))\in \cC$, which implies $(x_0,\xi_0)\notin\WFF^3_\cC( u(t,s) )$ and proves the claim.
\end{proof}

\begin{remark}
	\begin{enumerate}
		\item We recall that the canonical transformation generated by an arbitrary regular phase function $\varphi\in\Phr$ is defined by the relations
	\[
		(x,\xi )=\Phi_\varphi (y,\eta )
		\quad \Longleftrightarrow \quad
  		\begin{cases}
     			y  = \varphi^\prime_{\xi} (x,\eta) =  \varphi^\prime_{\eta} (x,\eta),
   			\\[1ex]
     			\xi = \varphi^\prime_x (x, \eta).
 		 \end{cases}
	\]

		\item \label{point:wfsbichar} Assume that the hypotheses of Theorem \ref{thm:propwfs} hold true. Then
			$\WFF^k_\cC(u(t,s))$, $(t,s)\in \Delta_{T^\prime}$, $k=1,2,3$, consists of unions of arcs of bicharacteristics, 
			generated by the phase functions appearing in \eqref{eq:soleqordm} and emanating from points 
			belonging to  $\WFF^m_{\cB_k} (g_k)$, $k=0,\dots,m-1$, cf. \cite{CoMa, Morimoto, Taniguchi02}.

		\item The hypotheses on the spaces $\cB_k$, $k=0,\dots,m-1$, $\cC$, 
		  	 are authomatically fulfilled for 
			$\cB_k=H^{r+m-1-k,\varrho+m-1-k}(\R^n)$, $\cC=H^{r,\varrho}(\R^n)$, $r,\varrho\in\R$, $k=0,\dots,m-1$.
			That is, the results in Theorem \ref{thm:propwfs} and in point \ref{point:wfsbichar} above hold true, in particular,
			for the $\WFF^k_{r,\varrho}(u(t,s))$ wave-front sets, $r,\varrho\in\R$, $k=1,2,3$. 
		\item A result similar to Theorem \ref{thm:propwfs} holds true for the solution $U(t,s)$ of the system \eqref{sys} when $F\equiv0$
		        and $\Lambda$ and $R$ are matrices of SG-classical operators. 
	\end{enumerate}
\end{remark}

\section{Stochastic Cauchy problems for weakly SG-hyperbolic linear operators}\label{sec:sghypspde}%
\setcounter{equation}{0}

This section is devoted to the proof of the existence of random-field solutions of a stochastic PDE of the form

\begin{equation}\label{eq:SPDE}
  L(t,D_t;x,D_x)u(t,x) = \gamma(t,x) + \sigma(t,x)\dot{\Xi}(t,x),
\end{equation}
associated with the initial conditions $D_t^ku(0,x)=g_k(x)$, $k=0,\dots,m-1$, for a SG-involutive operator $L$, where $\gamma$ and $\sigma$ are suitable real-valued
functions, $\dot\Xi$ is a random noise, described below, and $u$ is an unknown stochastic process called \emph{solution} of the SPDE. 

Since the sample paths of the solution $u$ are, in general, not in the domain of the operator $L$, in view of the singularity of the random noise, we rewrite \eqref{eq:SPDE} in its corresponding integral (i.e., \textit{weak}) form and look for \emph{mild solutions of \eqref{eq:SPDE}}, that is, stochastic processes $u(t,x)$ satisfying
\begin{equation}\label{eq:mildsolutionSPDE}
  u(t,x) = v_0(t,x)+\int_0^t\int_{\R^n} \Lambda(t,s,x,y)\gamma(s,y)dyds +\int_0^t\int_{\R^n} \Lambda(t,s,x,y)\sigma(s,y)\dot\Xi(s,y)dyds,
\end{equation}
where $\Lambda$ is a suitable distribution, associated with the fundamental solution of the operator $L$, and in \eqref{eq:mildsolutionSPDE} we adopted
the usual abuse of notation involving \textit{distributional integrals}.

Based on the results of the previous Section \ref{sec:sginvcp}, and on the analysis in \cite{AsCoSu:1}, we can show that \eqref{eq:mildsolutionSPDE}
has a meaning, and we call it the solution of \eqref{eq:SPDE} with the associated initial conditions.

\subsection{Stochastic integration with respect to a martingale measure.}\label{subs:stochastics}
We recall here the definition of stochastic integral with respect to a martingale measure, using material coming from \cite{AsCoSu:1,AsCoSu:2}, to which we refer
the reader for further details. Let us consider a distribution-valued 
Gaussian process $\{\Xi(\phi);\; \phi\in\mathcal{C}_0^\infty(\mathbb{R}_+\times\R^n)\}$ on a complete probability space $(\Omega, \scrF, \P)$,
with mean zero and covariance functional given by
\begin{equation}
	\E[\Xi(\phi)\Xi(\psi)] = \int_0^\infty\int_{\R^n} \big(\phi(t)\ast\tilde{\psi}(t)\big)(x)\,\Gamma(dx) dt,
	\label{eq:correlation}
\end{equation}
where $\widetilde{\psi}(t,x) := \psi(t,-x)$, $\ast$ is the convolution operator and $\Gamma$ is a nonnegative, nonnegative definite, tempered measure on $\R^n$.
Then \cite[Chapter VII, Th\'{e}or\`{e}me XVIII]{schwartz} implies that there exists a nonnegative tempered measure $\mu$ on $\R^n$ such that $\caF\mu = \widehat{\mu}=\Gamma$. 
By Parseval's identity, the right-hand side of \eqref{eq:correlation} can be rewritten as
\begin{equation*}
	\E[\Xi(\phi)\Xi(\psi)] = \int_0^{\infty}\int_{\R^n}[\caF\phi(t)](\xi)\
	\cdot
	\overline{[\caF\psi(t)](\xi)}\,\mu(d\xi) dt.
\end{equation*}
The tempered measure $\Gamma$ is usually called \emph{correlation measure}. The tempered measure $\mu$ such that $\Gamma=\widehat\mu$ is usually called \emph{spectral measure}.


In this section we consider the SPDE \eqref{eq:SPDE} and its mild solution \eqref{eq:mildsolutionSPDE}: this is the way in which we understand \eqref{eq:SPDE}; we provide conditions to show that each term on the right-hand side of \eqref{eq:mildsolutionSPDE} is meaningful. 
\\
In fact, we call \textit{(mild) random-field solution to \eqref{eq:SPDE}} an $L^2(\Omega)$-family of random variables $u(t,x)$, $(t,x)\in[0,T]\times\R^n$, jointly measurable, satisfying the stochastic integral equation \eqref{eq:mildsolutionSPDE}. 
\\

To give a precise meaning to the stochastic integral in \eqref{eq:mildsolutionSPDE} we define
\beqs\label{intm}
\ds\int_0^t\int_{\R^n}\Lambda(t,s,x,y)\sigma(s,y)\dot\Xi(s,y)dsdy := \int_0^t\int_\R^n \Lambda(t,s,x,y)\sigma(s,y)M(ds,dy),
\eeqs
where, on the right-hand side, we have a stochastic integral with respect to the martingale measure $M$ related to $\Xi$. As explained in 
\cite{dalangfrangos}, by approximating indicator functions with $C^\infty_0$-functions, the process $\Xi$ can indeed be extended to a worthy martingale measure 
$M=(M_t(A);\; t\in\R_+, A\in\scrB_b(\R^n))$, where $\scrB_b(\R^n)$ denotes the bounded Borel subsets of $\R^n$. The natural filtration generated by this martingale measure will be denoted in the sequel by $(\scrF_t)_{t\geq 0}$.
The stochastic integral with respect to the martingale measure $M$ of stochastic processes $f$ and $g$, indexed by $(t,x)\in[0,T]\times\R^n$ and satisfying suitable conditions, is constructed by steps  (see \cite{conusdalang,dalang,walsh}), starting from the class $\caE$ of simple processes, and making use of the
pre-inner product (defined for suitable $f,g$) 
\begin{equation}\label{eq:norm02}
\begin{aligned}
\langle f,g\rangle_{0}&= \E\bigg[\int_0^T\int_\R^n \big(f(s)\ast\tilde{g}(s)\big)(x)\,\Gamma(dx) ds \bigg]
\\
& = \E\bigg[\int_0^T\int_\R^n [\caF f(s)](\xi)\cdot\overline{[\caF g(s)](\xi)}\,\mu(d\xi) ds\bigg], 
 \end{aligned}
 \end{equation}
with corresponding semi-norm $\|\cdot\|_0$, as follows.
\begin{enumerate}
\item For a {\it simple process} 
\[ g(t,x;\omega) = \sum_{j=1}^m 1_{(a_j,b_j]}(t)1_{A_j}(x)X_j(\omega)\in\mathcal E, \]
(with $m\in\N$, $0\leq a_j < b_j\leq T$, $A_j\in\scrB_b(\R^n)$, $X_j$ bounded, and $\scrF_{A_j}$-measurable random variable for all $1\leq j\leq n$) the stochastic integral with respect to $M$ is given by
\[ (g\cdot M)_t := \sum_{j=1}^m \big(M_{t\wedge b_j}(A_j) - M_{t\wedge a_j}(A_j)\big)X_j, \]
where $x\wedge y := \min\{x,y\}$. One can show, by applying the definition, that the fundamental isometry
\begin{equation}\label{eq:isometry}
  \E\big[(g\cdot M)_t^2\big] = \|g\|^2_0
\end{equation}
holds true for all $g\in\caE$. 
\item Since the pre-inner product \eqref{eq:norm02} is well-defined on elements of $\mathcal E$, if now we define $\caP_0$ as the completion of $\caE$ with respect to $\langle\cdot,\cdot\rangle_0$, then, for all the elements $g$ of the Hilbert space $\caP_0$, we can construct the stochastic integral with respect to $M$ as an $L^2(\Omega)$-limit of simple processes via the isometry \eqref{eq:isometry}. 
So, $\caP_0$ turns out to be the space of all integrable processes (with respect to $M$). Moreover, by Lemma 2.2 in \cite{sanzsuess1} we know that $\caP_0=L^2_{p}([0,T]\times\Omega,\caH)$,
where here $L^2_p(\ldots)$ stands for the predictable stochastic processes in $L^2(\ldots)$ and $\caH$ is the Hilbert space which is obtained by completing the Schwartz functions with respect to the inner product $\langle\cdot,\cdot\rangle_0$. Thus, $\caP_0$ consists of predictable processes which may contain tempered distributions in the $x$-argument (whose Fourier transforms are functions, $\P$-almost surely).  
\end{enumerate}

Now, to give a meaning  to the integral \eqref{intm}, we need to impose conditions on the 
distribution $\Lambda$ and on the coefficient $\sigma$ such that 
$\Lambda\sigma\in \mathcal P_0$.
In \cite{alessiandre}, sufficient conditions for the existence of the integral on the right-hand side of \eqref{intm} have been given, in the case that $\sigma$ depends on the spatial argument $y$, assuming that the spatial Fourier transform of the function $\sigma$ is a complex-valued measure with finite total variation. Namely, we assume that, for all $s\in[0,T]$,
\[ |\caF\sigma(\cdot,s)| = |\caF\sigma(\cdot,s)|(\R^n) = \sup_\pi \sum_{A\in\pi} |\caF\sigma(\cdot,s)|(A) < \infty, \]
where $\pi$ is any partition on $\R^n$ into measurable sets $A$, and the supremum is taken over all such partitions. Let, in the sequel,
$\nu_s:=\caF\sigma(\cdot,s)$, and let $|\nu_s|_\tv$ denote its total variation.
We summarize all these conditions in the following theorem; for its proof, see \cite[Theorem 2.6]{alessiandre}.

\begin{theorem}\label{thm:existenceanduniquenessconstant}
Let $\Delta_T$ be the simplex given by $0\leq t\leq T$ and $0\le s \le t$. Let, for $(t,s,x)\in\Delta_T\times\R^n$, $\Lambda(t,s,x)$ be a deterministic function with values in $\SX'(\R^n)_\infty$, the 
space of rapidly decreasing temperate distributions, and let $\sigma$ be a function in $L^2([0,T],C_b(\R^n))$. Assume that:
 \begin{enumerate}
    \item[(A1)]\label{ass:A1} the function $(t,s,x,\xi)\mapsto[\caF\Lambda(t,s,x)](\xi)$ is measurable, the function $s\mapsto\caF\sigma(s)=\nu_s \in L^2([0,T],\mathcal M_b(\R^n))$, and, moreover,
		
		\begin{equation}\label{eq:condition1}
			\int_0^T\bigg(\sup_{\eta\in\R^n} \int_{\R^n} |[\caF\Lambda(t,s,x)](\xi+\eta)|^2 \mu(d\xi)\bigg) |\nu_s|^2_\tv \, ds < \infty;		\end{equation}
    \item[(A2)]\label{ass:A2} $\Lambda$ and $\sigma$ are as in (A1) and
			\begin{align*}
				\lim_{h\downarrow0} \int_0^T		& \bigg(\sup_{\eta\in\R^n} \int_{\R^n} \sup_{r\in(s,s+h)} |[\caF(\Lambda(t,s,x)-\Lambda(t,r,x))](\xi+\eta)|^2 \mu(d\xi)\bigg) |\nu_s|^2_\tv\,ds = 0.
				\end{align*}
  \end{enumerate}
Then $\Lambda\sigma\in\caP_0$, so the stochastic integral on the right-hand side of \eqref{intm} is well-defined and
 \begin{align*}
	\E\big[((\Lambda(t,\cdot,x,\ast)\sigma(\cdot,\ast))\cdot M)_t^2\big] \leq \int_0^t\bigg(\sup_{\eta\in\R^n} \int_{\R^n} |[\caF\Lambda(t,s,x)](\xi+\eta)|^2 \mu(d\xi)\bigg) |\nu_s|^2_\tv\,ds.
\end{align*}
\end{theorem}

The reason for the assumption that $\Lambda(t)\in\SX'(\R^n)_\infty$ is that, in this case, the Fourier transform in the second spatial argument is a smooth function of slow growth and the convolution of such a distribution with any other distribution in $\mathcal S'(\R^n)$ is well-defined, see \cite[Chapter VII, Section 5]{schwartz}. A necessary and sufficient condition for $T\in\SX'(\R^n)_\infty$ is that each regularization of $T$ with a $C^\infty_0$-function is a Schwartz function. This is true in our application, due to next Proposition \ref{prop:ftofschwartzkernel} and the fact that the Fourier transform is a bijection on the Schwartz functions.

\begin{proposition}\label{prop:ftofschwartzkernel}
Let $A=\Op_\varphi(a)$ be a SG Fourier integral operator, with symbol $a\in S^{m,\mu}(\R^{2n})$, $(m,\mu)\in\R^2$,
and phase function $\varphi$. Let $K_A$ denote its Schwartz kernel. Then, the Fourier transform with respect to the second argument of $K_A$,  $\caF_{\cdot\mapsto\eta}K_{A}(x,\cdot)$, is given by
  \begin{equation}\label{eq:ftofschwartzkernel}
		\caF_{\cdot\mapsto\eta}K_{A}(x,\cdot) = \e^{i \varphi(x,-\eta)}a(x,-\eta).
  \end{equation}
\end{proposition}


%
 In the next subsection we will apply Theorem \ref{thm:existenceanduniquenessconstant}, Proposition \ref{prop:ftofschwartzkernel} and the theory developed in the previous sections to prove the existence of
random-field solutions for stochastic PDEs associated with a SG-involutive operator. 

\subsection{Random field solutions of stochastic linear SG-involutive equations}\label{subs:inv}
We conclude the paper with our fifth main result. 
\begin{theorem}\label{thm:linearin} 
Let us consider the Cauchy problem 
\begin{equation}\label{eq:cpstoch}
	\left\{
\begin{array}{ll}
L(t,D_t;x,\partial_x)u(t,x) = \gamma(t,x) + \sigma(t,x)\dot{\Xi}(t,x),          & t\in[0,T], x\in\R^n,
\\[1ex]
D_{t}^k u(0,x) = g_{k}(x), & k=0,\dots,m-1, x\in\R^n,
\end{array}
\right.
\end{equation}
for a SPDE associated with a SG-involutive operator $L$ of the type \eqref{eq:4.69}, $m\in\N$, of Levi type, that is,
satisfying \eqref{eq:4.117}.
Assume also, for the initial conditions, that $g_k\in H^{z+m-k-1, \zeta+m-k-1}(\R^n)$, $0\leq k\leq m-1$, with $z\in\R$ and $\zeta>n/2$. 
Furthermore, assume the Gaussian noise $\dot{\Xi}$ to be of the type described in Subsection \ref{subs:stochastics}, with the associated spectral measure such that 
\begin{equation}\label{eq:measinv}
  \int_{\R^n} \mu(d\xi) < \infty.
\end{equation}
Finally, assume that $\gamma\in C([0,T]; H^{z,\zeta}(\R^n))$, $\sigma \in C([0,T], H^{0,\zeta}(\R^n))$, 
and $s\mapsto\mathcal F\sigma(s)=\nu_s \in L^2([0,T],\mathcal M_b(\R^n))$.

Then, for some $T^\prime\in(0,T]$, there 
exists a random-field solution 
$u$ of \eqref{eq:cpstoch}. Moreover, 
$\displaystyle\E[u]\in C([0,T^\prime], H^{z,\zeta}(\R^n))$.
\end{theorem}
\begin{proof} The operator $L$ is of the form considered in Section \ref{sec:sginvcp}. In particular, by Theorem \ref{thm:mainLastSecThm} we know that the solution to \eqref{eq:cpstoch} is formally given by \eqref{eq:soleqordm}, that is
\begin{equation}\label{theu}
	u(t)  = \sum_{\alpha\in M}E_{\varnothing \alpha}(t,0) G_{\alpha}+
	\sum_{\alpha\in M_{m-1}}\int_0^t E_{\varnothing \alpha}(t,s)  f(s)\,ds,    \quad t\in[0,T^\prime],
\end{equation}
where $M=\bigcup_{k=0}^m M_k$ with $M_k$ the permutations of $k$ elements of the set $\{1,\ldots,m\}$, and where $E_{\varnothing \alpha}(t,s)$, $\alpha\in M$, $(t,s)\in\Delta_{T^\prime}$ are (modulo elements with kernels in $C^\infty(\Delta_{T'};\mathcal S)$) suitable linear combinations of parameter-dependent families of iterated integrals of regular SG Fourier integral operators, with phase functions given by (sorted) sharp products of solutions to the eikonal equations associated with the characteristic roots of $L$, and matrix-valued symbols in $S^{(0,0)}$ determined again through the characteristic roots of $L$.
Let us insert now $f(t,x)=\gamma(t,x) + \sigma(t,x)\dot{\Xi}(t,x)$ in \eqref{theu}, so that, formally, 
\begin{equation}\label{eq:sollin1}
	\begin{aligned}
	u(t,x)&=v_0(t,x)+\int_0^t\int_{\R^d}\Lambda(t,s,x,y)\gamma(s,y)\,dyds +\int_0^t\int_{\R^d}\Lambda(t,s,x,y)\sigma(s,y)\dot{\Xi}(s,y)\,dyds
	\\
	&=v_0(t,x)+v_1(t,x)+v_2(t,x),
	\end{aligned}
\end{equation}
In \eqref{eq:sollin1}, $\Lambda(t,s,x,y)$ is the kernel of the finite sum $\sum_{\alpha\in M_{m-1}}E_{\varnothing \alpha}(t,s)$, that is, $\Lambda$ is (modulo an element of $C^\infty(\Delta_{T'};\mathcal S)$) a linear combination of (iterated integrals of) kernels of parameter-dependent FIOs with symbols of order $(0,0)$.
We have already observed (see the last lines of the proof of Theorem \ref{thm:mainLastSecThm}) that $\displaystyle v_0\in C([0,T'],H^{z,\zeta})$. 
Since, by assumption, $\zeta>\frac{n}{2}$, $v_0$ produces a function which is continuous and $L^2$ with respect to $x\in\R^n$ and $t\in[0,T']$, respectively. We have that

$$\forall (t,x)\in[0,T]\times\R^n,\ v_0(t,x)\ is\ finite.$$
Since all the operators $E_{\varnothing\alpha}(t,s)$ in \eqref{theu} are linear and continuous from $H^{z,\zeta}$ to itself,
and $\gamma\in C([0,T'],H^{z,\zeta})$, the first integral in \eqref{eq:sollin1} certainly makes sense, and also 
$\displaystyle v_1\in C([0,T'],H^{z,\zeta})$. This gives:
$$\forall (t,x)\in[0,T]\times\R^n,\ v_1(t,x)\ is\ finite.$$
Let us now focus on the term $v_2$. We can rewrite it as
\beqs\label{eq:v2}
	v_2(t,x)=\int_0^t\int_{\R^d}\Lambda(t,s,x,y)\sigma(s,y) M(ds, dy),
\eeqs
where $M$ is the martingale measure associated with the stochastic noise $\Xi$, as defined in Section \ref{subs:stochastics}. 
By Proposition \ref{prop:ftofschwartzkernel} we then find
\begin{equation}\label{yeah} 
	|[\caF_{y\mapsto\eta}\Lambda(t,s,x,\cdot)](\eta)|^2 \leq C_{t,s} \x^{0}\langle\eta\rangle^{0}=C_{t,s},
\end{equation}
where $C_{t,s}$ can be chosen to be continuous in $s$ and $t$, since $\Lambda$ differs by an element of $\mathcal C^\infty(\triangle_{T'},\mathcal S)$ from the kernel of a linear combination 
of (iterated integrals of) kernels of (parameter-dependent) SG FIOs with symbol in $\mathcal C(\triangle_{T'},S^{0,0})$.
Using \eqref{yeah}, we get that
condition (A1) in Theorem \ref{ass:A1} is satisfied if
\[
  \int_0^t \left(\sup_{\xi\in\R^n}\int_{\R^n} |[\caF_{y\mapsto\eta}\Lambda(t,s,x,\cdot)](\eta+\xi)|^2\mu(d\eta)\right)|\nu_s|^2_\tv\,ds 
   \lesssim \int_0^t |\nu_s|^2_\tv\, ds \int_{\R^n}\mu(d\eta)<\infty.
\]
In view of the assumptions on $\sigma$, we conclude that assumption $(A1)$ holds true as long as \eqref{eq:measinv} does. 
To check the continuity condition $(A2)$ in Theorem \ref{ass:A1}, since $\Lambda$ is 
regular with respect to $s$ and $t$, it suffices to show that 
\begin{equation}\label{ourstar}
\sup_{r\in(s,s+h)} |\caF(\Lambda(t,s,x)-\Lambda(t,r,x))(\xi+\eta)|^2\leq C_{t,s,h}^2,
\end{equation}
with $C_{t,s,h}\to 0$ as $h\to 0$ and $C_{t,s,h}\leq C_{T'}$ for every $h\in [0,t-s],$ $(t,s)\in\Delta_{T'}$.
Indeed, if \eqref{ourstar} holds, then (A2) holds
via Lebesgue's Dominated Convergence Theorem, in view of assumption \eqref{eq:measinv}, the fact that $|\nu_s|^2_\tv\in L^1[0,T]$, and $C_{t,s,h}\leq C_{T_0}$. 

Then, it only remains to show that \eqref{ourstar} holds true. But this follows from the fact that the function $s\mapsto \caF\Lambda(t,s,\cdot)(\ast)$ is, by \eqref{yeah}, uniformly continuous on $[0,t]$ with values in the Fr\'echet space $S^{0,0}(\R^{2n})$, endowed with the norm
$$||a-b||=\displaystyle\sum_{\ell=0}^\infty\frac1{2^\ell}\frac{\|a-b\|_\ell^{0,0}}{1+\|a-b\|_\ell^{0,0}},$$
so its modulus of continuity, 
\begin{align*}
\omega_{t,s}(h)=\sup_{r\in(s,s+h)} ||(\caF\Lambda(t,s,\cdot)(\ast)-\caF\Lambda(t,r,\cdot)(\ast))||
\end{align*}
tends to $0$ as $h\to 0$. For more details see \cite{AsCoSu:1}.

%
\indent
The argument above shows that we can apply Theorem \ref{ass:A1} to get that $v_2$ is well-defined as a stochastic integral with respect to the martingale measure canonically associated with $\Xi$. 

Summing up, the random-field solution $u(t,x)$ in \eqref{eq:sollin1} makes sense: its deterministic part is well-defined for every $(t,x)\in[0,T]$ and its stochastic part makes sense as a stochastic integral with respect to a martingale measure.%

The regularity claim $\displaystyle\E[u]\in C([0,T'],H^{z,\zeta}(\R^d))$ follows from the regularity properties of the $E_{\varnothing\alpha}$, of $\gamma$ and of the Cauchy data, taking expectation on both sides of \eqref{eq:sollin1}, and recalling the fact that the expected value $\E[v_2]$ of the stochastic integral is zero, being $\Xi$ a Gaussian process with mean zero. It follows that the regularity 
of $\E[u]$ is the same as the one of the solution of the associated deterministic Cauchy problem.

\end{proof}

\begin{remark}
	One could say that the random-field solution $u$ of \eqref{eq:cpstoch} found in Theorem \ref{thm:linearin} ``is unique'' in the following sense.
	First, when $\sigma\equiv 0$, it reduces to the unique solution of the deterministic Cauchy problem \eqref{eq:4.74},
	 with $f\equiv \gamma$ and $s=0$. Moreover, by linearity,
	 if $u_1$ and $u_2$ are two solutions of the linear Cauchy problem \eqref{eq:cpstoch},
	$u=u_1-u_2$ satisfies the deterministic equation $Lu=0$ with trivial initial conditions, and such Cauchy problem admits in $\mathcal S'$ only the trivial solution.
	The latter follows immediately by the $\mathcal S'$ well-posedness (with loss of smoothness and decay) of the Cauchy problem for the homogeneous deterministic linear
	equation $Lu=0$ proved in Theorem \ref{thm:mainLastSecThm}.  
\end{remark}

	
\bibliographystyle{abbrv}

\end{document}